\documentclass[12pt,reqno]{amsart}

\usepackage{amsmath, amsthm, amssymb, amscd, amsfonts}
\usepackage{fontsmpl}
\usepackage{fontsmpl,layout}
\usepackage{color}
\usepackage[all]{xy}             
 \usepackage{hyperref}

\font \smallrm=cmr10 at 10truept
\font \smallsl=cmsl10 at 10truept
\font \ssmallrm=cmr10 at 9truept

\numberwithin{equation}{section}

\newcommand{\subu}[2]{{#1}_{\raise-2pt\hbox{$ \scriptstyle #2 $}}}
\newcommand{\subd}[3]{{#1}_{\raise-2pt\hbox{$ \scriptstyle #2 #3 $}}}

\newtheorem{lema}{Lemma}[subsection]
\newtheorem{theorem}[lema]{Theorem}
\newtheorem{cor}[lema]{Corollary}

\newtheorem{prop}[lema]{Proposition}

\theoremstyle{definition}
\newtheorem{definition}[lema]{Definition}

\newtheorem{rmk}[lema]{Remark}
\newtheorem{rmks}[lema]{Remarks}
\newtheorem{free text}[lema]{}

\theoremstyle{remark}

\font \smallsl=cmsl10 at 9pt
\font \smallbf=cmbx10 at 9pt

\newcommand \id{\operatorname{id}}

\newcommand \gr{\operatorname{gr}}

\newcommand \Hom{\operatorname{Hom}}

\newcommand \Ker{\operatorname{Ker}}

\newcommand \smallast {{\operatorname{\raise1,5pt\hbox{$ \scriptscriptstyle \ast $}}}}
\newcommand \smallp {{\operatorname{\raise1pt\hbox{$ \scriptscriptstyle + $}}}}
\newcommand \smallm {{\operatorname{\raise1pt\hbox{$ \scriptscriptstyle - $}}}}
\newcommand \smallpm {{\operatorname{\raise1pt\hbox{$ \scriptscriptstyle \pm $}}}}
\newcommand \smallmp {{\operatorname{\raise1pt\hbox{$ \scriptscriptstyle \mp $}}}}

\newcommand{\eps}{\epsilon}
\newcommand{\ot}{\otimes}
\newcommand{\com}{\Delta}

\newcommand{\bnu}{{\boldsymbol{\nu}}}

\newcommand{\F}{{\mathcal F}}

\newcommand{\Z}{{\mathcal Z}}

\newcommand{\Kc}{{\mathcal K}}
\newcommand{\Lc}{{\mathcal L}}

\newcommand{\B}{{\mathcal B}}
\newcommand{\G}{{\mathbf G}}

\newcommand{\Hc}{{\mathcal H}}

\def \bq {\mathbf{q}}

\def \NN{\mathbb{N}}
\def \ZZ {\mathbb{Z}}
\def \QQ{\mathbb{Q}}
\def \CC{\mathbb{C}}
\def \FF{\mathbb{F}}
\def \k {\Bbbk}

\def \kh{\Bbbk[[\hbar]]}

\def \Zqqm {{\mathbb{Z}\big[\hskip1pt q,q^{-1}\big]}}

\def\FF{\mathbb{F}}

\def\II{\mathcal{I}}

\def\SS{\mathcal{S}}

\def\G{\mathcal{G}}

\def \Dpicc {{\scriptscriptstyle D}}

\def \erm {\mathrm{e}}

\def \frm {\mathrm{f}}

\def \hrm {\mathrm{h}}
\def \trm {\mathrm{t}}

\def \lieg{\mathfrak{g}}
\def \liegd{\mathfrak{g}_{\raise-2pt\hbox{$ \Dpicc $}}}
\def \liegdotdq{\dot{\lieg}_{\raise-2pt\hbox{$ \Dpicc , \hskip-0,9pt{\scriptstyle \bq} $}}}
\def \liegtildq{\tilde{\lieg}_{\raise-2pt\hbox{$ \Dpicc , \hskip-0,9pt{\scriptstyle \bq} $}}}
\def \liegdq{\lieg_{\raise-2pt\hbox{$ \Dpicc , \hskip-0,9pt{\scriptstyle \bq} $}}}
\def \liegdqcheck{\lieg_{\raise-2pt\hbox{$ \Dpicc , \hskip-0,9pt{\scriptstyle \check{\bq}} $}}}
\def \Gtildeqstar {\widetilde{G}^{\,\raise2pt\hbox{$ \scriptstyle * $}}_{\!\Dpicc , \hskip+0,9pt{\scriptstyle \bq}}}

\def \lieb{\mathfrak{b}}
\def \lieh{\mathfrak{h}}
\def \liehd{\mathfrak{h}_{\raise-2pt\hbox{$ \Dpicc $}}}

\def \lien{\mathfrak{n}}
\def \lieso{\mathfrak{so}}

\def \uhg{U_\hbar(\hskip0,5pt\lieg)}
\def \upsihg{U^{\scriptscriptstyle \Psi}_\hbar(\hskip0,5pt\lieg)}
\def \uhgd{U_\hbar(\hskip0,5pt\liegd)}
\def \upsihgd{U^{\scriptscriptstyle \Psi}_\hbar(\hskip0,5pt\liegd)}

\def \uhh{U_\hbar(\hskip0,0pt\lieh)}
\def \uhbp{U_\hbar(\hskip0,0pt\lieb_+)}
\def \upsihbp{U^{\scriptscriptstyle \Psi}_\hbar(\hskip0,0pt\lieb_+)}
\def \udotpsihbp{{\dot{U}}^{\scriptscriptstyle \Psi}_\hbar(\hskip0,0pt\lieb_+)}
\def \udotpsihbpm{{\dot{U}}^{\scriptscriptstyle \Psi}_\hbar(\hskip0,0pt\lieb_\pm)}
\def \uhbm{U_\hbar(\hskip0,0pt\lieb_-)}
\def \upsihbm{U^{\scriptscriptstyle \Psi}_\hbar(\hskip0,0pt\lieb_-)}
\def \udotpsihbm{{\dot{U}}^{\scriptscriptstyle \Psi}_\hbar(\hskip0,0pt\lieb_-)}

\def \uqg{U_q(\hskip0,8pt\lieg)}
\newcommand \uqgG[1]{U_{q,#1}(\hskip0,5pt\lieg)}

\newcommand \upsiqgG[1]{U^{\scriptscriptstyle \,\Psi}_{q,#1}(\hskip0,5pt\lieg)}
\def \uqgd{U_q(\hskip0,5pt\liegd)}
\newcommand \uqgdG[1]{U_{q,#1}(\hskip0,5pt\liegd)}

\newcommand \upsiqgdG[1]{U^{\scriptscriptstyle \,\Psi}_{q,#1}(\hskip0,5pt\liegd)}
\def \uqh{U_q(\hskip0,0pt\lieh)}
\newcommand \uqhG[1]{U_{q,#1}(\lieh)}
\def \uqbp{U_q(\lieb_+)}
\newcommand \uqbpG[1]{U_{q,#1}(\lieb_+)}

\newcommand \upsiqbpG[1]{U^{\scriptscriptstyle \Psi}_{q,#1}(\lieb_+)}

\newcommand \uhatpsiqbpG[1]{{\hat{U}}^{\scriptscriptstyle \Psi}_{q,#1}(\lieb_+)}
\newcommand \uhatpsiqbpmG[1]{{\hat{U}}^{\scriptscriptstyle \Psi}_{q,#1}(\lieb_\pm)}
\def \uqbm{U_q(\lieb_-)}
\newcommand \uqbmG[1]{U_{q,#1}(\lieb_-)}

\newcommand \upsiqbmG[1]{U^{\scriptscriptstyle \Psi}_{q,#1}(\lieb_-)}

\newcommand \uhatpsiqbmG[1]{{\hat{U}}^{\scriptscriptstyle \Psi}_{q,#1}(\lieb_-)}

\def \QEq{U_\bq(\hskip0,5pt\liegd)}

\def \QEqcheck {U_{\check{\bq}}(\hskip0,5pt\liegd)}

\def \fqgd{F_q[G_{\raise-2pt\hbox{$ \Dpicc $}}]}

\def\pf{\begin{proof}}
\def\epf{\end{proof}}

\theoremstyle{plain}

\begin{document}


\title[Twisted deformations vs.\ cocycle deformations for quantum groups]
{Twisted deformations  \\
 vs.\ cocycle deformations  \\
 for quantum groups}

\author[g.~a.~garc{\'\i}a \ , \ \  f.~gavarini]
{Gast{\'o}n Andr{\'e}s Garc{\'\i}a${}^\flat$ \ ,  \ \ Fabio Gavarini$\,{}^\sharp$}

\address{\newline\noindent Departamento de Matem\'atica, Facultad de Ciencias Exactas
   \newline
 Universidad Nacional de La Plata   --- CMaLP ---  CONICET
   \newline
 C.\ C.\ 172   \, --- \,   1900 La Plata, ARGENTINA
   \newline
  {\ } \quad  {\tt ggarcia@mate.unlp.edu.ar}
 \vspace*{0.5cm}
   \newline
 Dipartimento di Matematica,
   \newline
 Universit\`a degli Studi di Roma ``Tor Vergata''
   \newline
 Via della ricerca scientifica 1   \, --- \,   I\,-00133 Roma, ITALY
   \newline
  {\ } \quad  {\tt gavarini@mat.uniroma2.it}}


\begin{abstract}
  \bigskip
   In this paper we study two deformation procedures for quantum groups:
 deformations by twists, that we call  \textit{``comultiplication twisting''},  as they
 modify the coalgebra structure, while keeping the algebra one ---   and
deformations by  $ 2 $--cocycle, that we call  \textit{``multiplication twisting''},  as they
 deform the algebra structure, but save the coalgebra one.
                                                                     \par
   We deal with quantum universal enveloping algebras, in short QUEA's,
   for which we accordingly consider those arising from twisted deformations (in short TwQUEA's)
   and those arising from  $ 2 $--cocycle  deformations, usually called multiparameter QUEA's (in short MpQUEA's).
   Up to technicalities, we show that the two deformation methods are equivalent, in that they eventually provide
   isomorphic outputs, which are deformations (of either kinds) of the ``canonical'', well-known one-parameter
   QUEA by Jimbo and Lusztig.  It follows that the two notions of TwQUEA's and of MpQUEA's   --- which, in Hopf algebra theoretical terms are naturally dual to each other ---   actually coincide; thus, that there exists in fact only one type of ``pluriparametric deformation'' for QUEA's.   In particular, the link between the realization of any such QUEA as a MpQUEA and that as a TwQUEA is just a (very simple, and rather explicit) change of presentation.
 \vskip-31pt
   {\ }
\end{abstract}

{\ } \vskip-65pt

   \centerline{ \smallrm  {\smallsl Communications in Contemporary Mathematics\/}  {\smallbf 23}  (2021), no.~8 --- 2050084 (56 pages) }
                                                 \par
   \centerline{ \smallrm {\smallbf DOI:}  10.1142/S0219199720500844 }
 \vskip1pt
   \centerline{\smallrm {\smallsl The original publication is available at\/}
\  https://www.worldscientific.com/doi/epdf/10.1142/S0219199720500844}
 {\ }
\vskip1pt

\maketitle

\tableofcontents



\smallskip

\section{Introduction}\label{sec:intro}

\vskip5pt

   Roughly speaking, quantum groups   --- in the form of quantized universal enveloping algebras ---   are Hopf algebra deformations of the universal enveloping algebra  $ U(\lieg) $  of some Lie algebra  $ \lieg \, $.  From this deformation,  $ \lieg $  itself inherits (as ``semiclassical limit'' of the deformed coproduct) a Lie cobracket that makes it into a Lie bialgebra   --- the infinitesimal counterpart of a Poisson group whose tangent Lie algebra is $ \lieg \, $.
                                                               \par
   When  $ \lieg $  is a complex simple Lie algebra, a quantum group in this sense, depending
   on a single parameter, was introduced by Drinfeld  \cite{Dr}  as a formal series deformation
   $ U_\hbar(\lieg) $  defined over a ring of formal power series (in the formal parameter  $ \hbar \, $)
   and by Jimbo and Lusztig (see  \cite{Ji},  \cite{Lu})  as a deformation  $ U_q(\lieg) $  defined over
   a ring of rational series (in the formal parameter  $ q \, $).  Indeed, Jimbo's  $ U_q(\lieg) $  is actually
   a ``polynomial version'' of Drinfeld's  $ U_\hbar(\lieg) \, $.
                                                               \par
   Later on, several authors  (cf.\ \cite{BGH},  \cite{BW1,BW2},  \cite{CM},  \cite{CV1},
   \cite{Hay},  \cite{HLT},  \cite{HPR},  \cite{Ko},  \cite{KT},  \cite{Ma},  \cite{OY},  \cite{Re},  \cite{Su},  \cite{Ta},
   to name a few) introduced many types of deformations of  $ U(\lieg) $  depending on several parameters,
   usually referred to as ``multiparameter quantum groups''.  In turn, these richer deformations induce as
   semiclassical limits corresponding ``multiparameter'' bialgebra structures on  $ \lieg \, $.
   The construction of these multiparameter deformations applies a general procedure, always available
   for Hopf algebras, following two patterns that we recall hereafter.
 \vskip3pt
   Let  $ H $  be any Hopf algebra (in some braided tensor category).
   Among all possible deformations of the Hopf structure of  $ H $,
   we look at those in which only one of either the product or the coproduct is actually modified,
   while the other one is kept fixed.  The general deformation will then be, somehow, an intermediate
   case between two such extremes.  On the one hand, a  \textsl{twist deformation\/}  of  $ H $  is a (new)
   Hopf algebra structure on  $ H $  where the multiplicative structure is unchanged, whereas a new coproduct
   is defined by  $ \; \Delta^{\scriptscriptstyle \F}(x) := \F \, \Delta(x) \F^{-1} \; $  for  $ \, x  \in H \, $:  here  $ \F $
   is an invertible element in  $ H^{\otimes 2} $  satisfying suitable axioms, called a ``twist'' for  $ H $.
   On the other hand, a  \textsl{$ 2 $--cocycle deformation\/}  of  $ H $  is one where the coproduct is unchanged,
   while a new product is defined via a formula which only depends on the old product and on a
   $ 2 $--cocycle  $ \sigma $  of  $ H $
 (as an algebra): again, this procedure can be read as a suitable ``conjugation'' of the
  old product map by the  $ 2 $--cocycle.
                                                              \par
   Inasmuch as a meaningful notion of ``duality'' applies to the Hopf algebras one is dealing with,
   these two constructions of deformations are  {\sl dual to each other},  directly by definition.
   In detail, if  $ H^* $  is a Hopf algebra dual to  $ H $  with respect to a non-degenerate (skew) Hopf pairing,
   e.g.\  $ H $  and  $ \, H^* := H^\circ \, $  (i.e., Sweedler's  \textsl{restricted\/}  dual),
   then the dual of the deformation by twist, resp.\ by  $ 2 $--cocycle,  of  $ H $  is a deformation by  $ 2 $--cocycle,
   resp.\ by twist, of  $ H^* $;  moreover, the  $ 2 $--cocycle, resp.\ the twist, on  $ H^* $
   is uniquely determined by the twist, resp.\ the  $ 2 $--cocycle,  on  $ H $.
 In order to stress this duality between the two types of deformation procedures that we are dealing with,
  as well as the fact that both are in fact ``conjugations'' of some sort, we adopt the terminology
  \textit{``comultiplication twisting''\/}  and  \textit{``multiplication twisting''},  instead of
  \textsl{``deformation by twist''\/}  and of  \textsl{``deformation by  $ 2 $--cocycle''},  respectively.
 \vskip5pt
   It so happens that the large majority of multiparameter quantizations of  $ U(\lieg) $
   considered in literature actually occur as either
comultiplication twistings
 or  multiplication twistings
 of a one-parameter quantization of Drinfeld's type or Jimbo-Lusztig's type.
 Indeed, in both cases the twists and the  $ 2 $--cocycles  taken into account are of special type,
 namely ``toral'' ones, in that (roughly speaking) they are defined only in terms of the (quantum) toral part
 of the one-parameter deformation of  $ U(\lieg) \, $.
                                                               \par
   Technically speaking, Drinfeld's  $ U_\hbar(\lieg) $  is better suited for
comultiplication twistings,
 while Jimbo-Lusztig's  $ U_q(\lieg) $  is typically used for
multiplication twistings
 (see  \cite{Re}, \cite{Ma}, \cite{Su}, \cite{HPR}, \cite{HLT}, \cite{CV1}, \cite{Ta}).
 As we aim to compare both kinds of twistings, we focus on polynomial one-parameter quantum groups
 $ U_q(\lieg) \, $,  and we adapt the very notion of ``twist deformation'',
or ``comultiplication twisting'',
 to them.  Then we consider both
comultiplication twistings
 and
multiplication twistings
 (of ``toral type'', in both cases) of  $ U_q(\lieg) $   ---
thus getting
``twisted quantized universal enveloping algebras (=TwQUEA's)'' and ``multiparameter quantized
universal enveloping algebras (=MpQUEA's)'', respectively ---   and compare them.
Moreover, by natural reasons we restrict ourselves to twists and cocycles that are defined by a
{\sl rational\/}  datum, i.e., a matrix with rational entries.
 \vskip3pt
   As a first result, we
describe
 the link  {\sl \,twist  $ \longleftrightarrow 2 $--cocycle\;}
 under duality.  Namely, quantum Borel (sub)groups  $ U_q(\lieb_\pm) $  of opposite signs are in Hopf duality
 (in a proper sense): then we prove that
any twisting
 on one side   ---
of either comultiplication or multiplication
 ---   and the  {\sl dual\/}  one on the other side   ---
of either multiplication or comultiplication,
 respectively ---   are described by  {\it the same rational datum}.
 Indeed, we provide an explicit bijection between the sets of toral twists and toral  $ 2 $--cocycles.
 \vskip3pt
   As a second, more striking result (the core of our paper, indeed), we find that, in short,
   {\sl twisted quantum groups and multiparameter quantum groups coincide\/}:
   namely, any TwQUEA can be realized as a MpQUEA, and viceversa.
   Even more precisely, the twist and the  $ 2 $--cocycle  involved in either realization are described by
   {\it the same (rational) datum}.  This result is, in a sense, a side effect of the ``autoduality'' of quantum groups
   (in particular Borel ones).  The proof is constructive, and quite explicit: indeed, switching from the realization as
   TwQUEA to that as MpQUEA and viceversa is a sheer  {\sl change of presentation}.
   We can shortly sketch the underlying motivation: any ``standard'' (=undeformed) quantum group is pointed
   (as a Hopf algebra); then any TwQUEA  of ``toral type'' is pointed as well, and it is generated by the quantum torus and
   $ (1,g) $--skew  primitive elements: these new ``homogeneous'' generators yield a new presentation,
   which realizes the TwQUEA as a MpQUEA.
                                                               \par
   The direct consequence of this result is that (roughly speaking, and within the borders of our restrictions)
   there exists only  {\sl one\/}  type of multiparameter quantization of  $ U(\lieg) $,  and consequently only one
   type of corresponding multiparameter Lie bialgebra structure on  $ \lieg $  arising as semiclassical limits, as in
\cite{GG1}.
 \vskip3pt
 All the elements that lead us to the above mentioned results for TwQUEA's and MpQUEA's are also
 available for Hopf algebras that (like Borel quantum subgroups) are bosonizations of
 Nichols algebras of diagonal type; thus, we can replicate our work in that context too.
 In another direction, we extend further on this analysis in the framework of multiparametric
 \textsl{formal\/}  QUEA's, \`a la Drinfeld ---  cf.\  \cite{GG2}.
 \vskip5pt
   We finish with a few words on the structure of the paper.
                                                               \par
   In Section 2 we collect the material on Hopf algebras and their deformations that will be
   later applied to quantum groups. Section 3 is devoted to introduce quantum groups (both in Drinfeld's version
   and in Jimbo-Lusztig's one) and their
 comultiplication twistings
 (of rational, toral type),
i.e., the TwQUEA's: the
 part on  {\sl Drinfeld's\/}  quantum groups could be dropped, yet we present it to explain the deep-rooting motivations
 of our work.  In Section 4, instead, we present the
multiplication twistings
 (of rational, toral type) of Jimbo-Lusztig's quantum groups,
hence the MpQUEA's.
  Finally, in Section 5 we compare TwQUEA's and MpQUEA's,
 proving that   --- in a proper sense,
under some finiteness assumption
 ---   they actually coincide.
 \vskip23pt
   \centerline{\ssmallrm ACKNOWLEDGEMENTS}
 \vskip3pt
   {\smallrm
   We thank the referee for the careful reading of the paper and for her/his suggestions that
   allowed us to improve the presentation.

   This work was supported by the research group GNSAGA of the Istituto Nazionale di Alta Matematica,
   by the International Center for Theoretical Physics of Trieste and the University of Padova (Italy),
   the research programs CONICET and ANPCyT (Argentina), the Visiting Professor Program of the
   Department of Mathematics of the Universidad Nacional de La Plata (Argentina) and by the MIUR
   {\smallsl Excellence Department Project\/}  awarded to the Department of Mathematics of the University of Rome
   ``Tor Vergata'', CUP E83C18000100006.}

\bigskip
 \medskip

\section{Preliminaries}  \label{preliminaries}

\medskip

   In this section we fix the basic material on Hopf algebras and combinatorial data that we shall need later on.
   In particular,  $ \, \NN = \{0, 1,\ldots\} \, $  and  $ \, \NN_+ := \NN \setminus \{0\} \, $.

\medskip

 \subsection{The combinatorial tool box}  \label{tool-case}  \
 \vskip7pt
   The definition of our multiparameter quantum groups requires a full lot of related material that we now present.
   First of all,  $ \k $  will be a field of characteristic zero.

\medskip

\begin{free text}{\bf Root data and Lie algebras.}  \label{root-data_Lie-algs}
 Hereafter we fix  $ \, n \in \NN_+ \, $  and  $ \, I := \{1,\dots,n\} \, $.  Let  $ \, A := {\big(\, a_{ij} \big)}_{i, j \in I} \, $
 be a generalized, symmetrisable Cartan matrix; then there exists a unique diagonal matrix
 $ \, D := {\big(\hskip0,7pt d_i \, \delta_{ij} \big)}_{i, j \in I} \, $  with positive integral, pairwise coprime entries such that  $ \, D A \, $  is symmetric.
 Let  $ \, \lieg =  \lieg_A \, $  be the Kac-Moody algebra over  $ \CC $  associated with  $ A $  (cf.\ \cite{Ka});
 we consider a split integral  $ \ZZ $--form  of  $ \lieg \, $,  and its scalar extension $ \lieg_{R} \, $ from  $ \ZZ $  to any ring  $ R \, $:
 when this ring is  $ \k \, $,  by abuse of notation  {\it the resulting Lie algebra over  $ \k $  will be denoted by  $ \lieg $  again}.
                                                                  \par
   Let  $ \Phi $  be the root system  of  $ \lieg \, $,  with  $ \, \Pi = \big\{\, \alpha_i \,\vert\, i\in I \,\big\} \, $  as a set of simple roots,
   $ \, Q = \bigoplus_{i \in I} \ZZ \, \alpha_i \, $  the associated root lattice,  $ \Phi^+ $  the set of positive roots with respect to
   $ \Pi \, $,  $ \, Q^+ = \bigoplus_{i \in I} \NN \, \alpha_i \, $  the positive root (semi)lattice.
                                                                  \par
   Fix a Cartan subalgebra  $ \lieh $  of  $ \lieg \, $,  whose associated set of roots identifies with
   $ \Phi \, $  (so  $ \, \k{}Q \subseteq \lieh^* \, $);  then for all  $ \, \alpha \in \Phi \, $  we call
   $ \lieg_\alpha $  the corresponding root space.  Now set  $ \, \lieh' := \lieg' \cap \lieh \, $  where
   $ \, \lieg' := [\lieg\,,\lieg] \, $  is the derived Lie subalgebra of  $ \lieg \, $:  then
   $ \, {\big( \lieh'\big)}^* = \k{}Q \subseteq \lieh^* \, $.  We fix a  $ \k $--basis
   $ \, \Pi^\vee := {\big\{\, \hrm_i := \alpha_i^\vee \,\big\}}_{i \in I} \, $  of  $ \lieh' $  so that
   $ \, \big( \lieh \, , \Pi \, , \Pi^\vee \big) \, $  is a  {\it realization\/}  of  $ A \, $,  as in
   \cite[Chapter 1]{Ka};  in particular,  $ \, \alpha_i(\hrm_j) = a_{ji} \, $  for all  $ \, i , j \in I \, $.
                                                                \par
   Let  $ \lieh'' $  be any vector space complement of  $ \lieh' $  inside  $ \lieh \, $.
   Then there exists a unique symmetric  $ \k $--bilinear  pairing on  $ \lieh \, $,
   denoted  $ (\,\ ,\ ) \, $,  such that
%
%
$ \, (\hrm_i\,,\hrm_j) = a_{ji}\,d_i^{-1} \, $,
 $ \, (\hrm_i\,,h''_2) = \alpha_i\big(h''_2\big) d_i^{-1} \, $
 and  $ \, (h''_1\,,h''_2) = 0 \, $,  for all  $ \, i, j \in I \, $,  $ \, h''_1, h''_2 \in \lieh'' \, $;
 in addition, this pairing is invariant and non-degenerate (cf.\ \cite[Lemma 2.1]{Ka}).
 By non-degeneracy, this pairing induces a  $ \k $--linear  isomorphism
 $ \; t : \lieh^* \,{\buildrel \cong \over {\relbar\joinrel\longrightarrow}}\, \lieh \; $,  with
 $ \, t(\alpha_i) = d_i \, \alpha_i^\vee =  d_i \, h_i \, $  for all  $ \, i \in I \, $,  and this in
 turn defines a similar pairing on  $ \, \lieh^* \, $   --- again denoted  $ (\,\ ,\ ) $  ---
 via pull-back, namely
   $ \, \big( t^{-1}(h_1) , t^{-1}(h_2) \big) := (h_1\,,h_2) \, $;  in particular, on simple roots
   this gives  $ \, (\alpha_i \, , \alpha_j) := d_i\,a_{ij} \, $
   for all  $ \, i, j \in I \, $.  In fact, this pairing on  $ \lieh^* $  does restrict to a (symmetric,
   $ \ZZ $--valued,  $ \ZZ $--bilinear)
   pairing on  $ Q \, $;  note that, in terms of the latter pairing on  $ Q \, $,  one has
   $ \, d_i = (\alpha_i\,,\alpha_i) \big/ 2 \, $  and
   $ \, a_{ij} = \frac{\,2\,(\alpha_i , \,\alpha_j)\,}{\,(\alpha_i , \,\alpha_i)\,} \, $  for all  $ \, i, j \in I \, $.
   We also notice that  $ \; t : \lieh^* \,{\buildrel \cong \over {\relbar\joinrel\longrightarrow}}\, \lieh \; $
   restricts to another isomorphism
   $ \; t' : {\big( \lieh' \big)}^* \,{\buildrel \cong \over {\relbar\joinrel\longrightarrow}}\, \lieh' \; $
   for which we use notation  $ \, t_\alpha := t'(\alpha) = t(\alpha) \, $.
 \vskip5pt
%
   When
 $ A $  is of  {\sl finite type}   --- which is equivalent to saying that  $ \, \lieh' = \lieh \, $  ---
   we denote by  $ P $  its associated  {\sl weight lattice},  with basis  $ \, {\big\{\, \omega_i \,\big\}}_{i \in I} \, $
   dual to  $ \, {\big\{\, \alpha_j \,\big\}}_{j \in I} \, $,  so that  $ \, \omega_i(\alpha_j) = \delta_{ij} \; $  for  $ \, i, j \in I \, $.
   If we identify  $ \lieh^* $  with  $ \lieh $  via the isomorphism
   $ \; t : \lieh^* \,{\buildrel \cong \over {\relbar\joinrel\longrightarrow}}\, \lieh \; $  as above, the root lattice  $ Q $
   identifies with a suitable sublattice of  $ P \, $:  then there is also an identification  $ \, \QQ{}P = \QQ{}Q \, $
   --- where hereafter we use such notation as  $ \, \QQ{}Q := \QQ \otimes_\ZZ Q \, $,  etc.\ ---
   and we have such identities as
   $ \; \alpha_i = \sum_{j \in I} a_{ji} \, \omega_j \; $  and
   $ \; (\omega_i \,, \alpha_ j) := d_i \, \delta_{ij} \; $  for all  $ \, i \, , j \in I \, $.
 \vskip7pt
   According to our choice of positive and negative roots, let  $ \lieb_+ \, $,  resp.\  $ \lieb_- \, $,
   be the Borel subalgebra in  $ \lieg $  containing  $ \lieh $  and all positive, resp.\  negative, root spaces.
   There is a canonical, non-degenerate pairing between  $ \lieb_+ $  and  $ \lieb_- \, $,
   using which one can construct a  {\sl Manin double\/}  $ \, \liegd = \mathfrak{b}_+ \oplus \mathfrak{b}_- \, $,
   that is automatically endowed with a structure of Lie bialgebra   --- roughly,  $ \liegd $
   is like  $ \lieg $ {\sl but\/}  with  {\sl two copies of\/}  $ \lieh $  inside it (cf.\  \cite{CP}, \S 1.4),
   namely  $ \, \lieh_+ := \lieh \oplus 0 \, $  inside  $ \lieb_+ $  and  $ \, \lieh_- := 0 \oplus \lieh \, $
   inside  $ \lieb_- \, $;  accor\-dingly, we set also  $ \, \lieh'_+ := \lieh' \oplus 0 \, $  and
   $ \, \lieh'_- := 0 \oplus \lieh' \, $.  By construction both  $ \lieb_+ $  and  $ \lieb_- $  lies in
   $ \liegd $  as Lie sub-bialgebras.  Moreover, there exists a Lie bialgebra epimorphism
   $ \, \pi_{\liegd} \! : \liegd \!\relbar\joinrel\relbar\joinrel\twoheadrightarrow \lieg \; $  which
   maps the copy of  $ \lieb_\pm $  inside  $ \liegd $  identically onto its copy in  $ \lieg \, $.
 \vskip7pt
   For later use we fix generators  $ \, \erm_i , \hrm_i , \frm_i \, (\, i \in I \,) \, $  in  $ \lieg \, $
   as in the usual Serre's presen\-tation of  $ \lieg \, $.  Moreover, for the corresponding elements inside
   $ \, \liegd = \lieb_+ \oplus \lieb_- \, $  we adopt notation  $ \, \erm_i := (\erm_i , 0) \, $,
   $ \, \hrm^+_i := (\hrm_i , 0) \, $,  $ \, \hrm^-_i := (0 , \hrm_i) \, $  and  $ \, \frm_i := (0 , \frm_{i}) \, $,
   for all  $ \, i \in I \, $.  {\sl Notice that\/}  we have by construction
\begin{equation}  \label{Serre-generators}
  \erm_i \in \lieg_{+\alpha_i}  \quad ,
  \qquad  \hrm_i = d_i^{-1} t_{\alpha_i} \in \lieh  \quad ,  \qquad  \frm_i \in \lieg_{-\alpha_i}
  \qquad \qquad  \forall \;\; i \in I
\end{equation}

\end{free text}

\medskip

\begin{free text}{\bf Root twisting.}  \label{root-twisting}
 Applying the twisting procedure to quantized universal enveloping algebras, we shall eventually be lead to consider an
 operation of ``root twisting'', in some sort, that we formalize hereafter.
 \vskip5pt
   Let  $ \hbar $  be a formal variable, and  $ \k((\hbar)) $
   the corresponding field of Laurent formal series with coefficients in  $ \k \, $;
   for simplicity, we eventually will take a field  $ \FF $  of characteristic zero containing  $ \k((\hbar)) \, $.
   Fix a subring  $ \mathcal{R} $  of  $ \k((\hbar)) $  containing  $ \QQ[[\hbar]] \, $,  let  $ \, \mathcal{R}Q \, $
   be the scalar extension of  $ Q $  by  $ \mathcal{R} \, $,  and fix an  $ (n \times n) $--matrix
   $ \; \Psi := {\big( \psi_{ij} \big)}_{i,j \in I} \in M_n\big(\mathcal{R}\big) \; $.  We define the endomorphisms
   $ \; \psi_\pm : \mathcal{R}Q \longrightarrow \mathcal{R}Q \; $  given by
\begin{equation}  \label{endom-psi}
  \psi_+(\alpha_\ell) = \zeta_\ell^{\scriptscriptstyle +} := \! {\textstyle \sum\limits_{i, j \in I}} \psi_{ij} \, a_{j\ell} \, d_i^{-1} \, \alpha_i  \; ,   \;\;
  \psi_-(\alpha_\ell) = \zeta_\ell^{\scriptscriptstyle -} := \! {\textstyle \sum\limits_{i, j \in I}} \psi_{ji} \, a_{j\ell} \, d_i^{-1} \, \alpha_i   \quad \forall \; \ell \in I
\end{equation}
that in matrix notation reads
\begin{equation}  \label{endom-psi matr-notat}  {\ } \hskip-9pt
  {\big( \psi_+(\alpha_\ell) = \zeta_\ell^{\scriptscriptstyle +} \big)}_{\! \ell \in I} \! :=
  A^{\scriptscriptstyle T} \Psi^{\scriptscriptstyle T} D^{-1} {\big( \alpha_k \big)}_{\! k \in I}  \; ,   \;\;
  {\big( \psi_-(\alpha_\ell) = \zeta_\ell^{\scriptscriptstyle -} \big)}_{\! \ell \in I} \! :=
  A^{\scriptscriptstyle T} \Psi D^{-1} {\big( \alpha_k \big)}_{\! k \in I}
\end{equation}
where  $ \,  {\big( \alpha_k \big)}_{\! k \in I} = {\big( \alpha_1 \, , \dots , \alpha_n \big)}^{\!\scriptscriptstyle T} \, $  is
thought of as a column vector, and likewise for  $ \, {\big( \psi_\pm(\alpha_\ell) \big)}_{\! \ell \in I} = {\big( \zeta_\ell^{\scriptscriptstyle \pm} \big)}_{\! \ell \in I} \, $.
 Now, borrowing notation from  \S \ref{root-data_Lie-algs}  we fix the  $ \mathcal{R} $--integral form  $ \lieh'_{{}_{\mathcal{R}}} $
 of  $ \lieh' $  in  $ \lieg_{{}_{\mathcal{R}}} $  spanned by the simple coroots  $ \, \hrm_i \, (\, i \in I \,) \, $,
 and the corresponding isomorphism
 $ \; t' : {\big( \lieh'_{{}_{\mathcal{R}}} \big)}^* \,{\buildrel \cong \over {\relbar\joinrel\longrightarrow}}\, \lieh'_{{}_{\mathcal{R}}} \;
 \big( \alpha \mapsto t_\alpha \big) \; $   --- this does make sense, because the original isomorphism
 $ \; t' : {\big( \lieh' \big)}^* \,{\buildrel \cong \over {\relbar\joinrel\longrightarrow}}\, \lieh' \; $  in  \S \ref{root-data_Lie-algs}  is actually
well-defined over\/  $ \QQ \, $.  Then we define endomorphisms $ \psi^{\lieh'}_\pm $  of  $ \lieh'_{{}_{\mathcal{R}}} $
as  $ \, \psi^{\lieh'}_\pm := t' \circ \psi_\pm \circ {\big( t' \big)}^{-1} \, $;  \,if we set  $ \, T_\ell := t_{\alpha_\ell} \, $  and
$ \, H_\ell := d_\ell^{-1} \, T_\ell \, $ for  $ \, \ell \in I \, $,  these are obviously described by
\begin{eqnarray}  \label{endom-psi-h}
  \psi^{\lieh'}_+(T_\ell)  \, :=  {\textstyle \sum\limits_{i, j \in I}} \psi_{ij} \, a_{j\ell} \, d_i^{-1} \, T_i  \, :=
  {\textstyle \sum\limits_{i, j \in I}} \psi_{ij} \, a_{j\ell} \, H_i \quad   \qquad \forall \;\;\; \ell \in I  \nonumber  \\
    \psi^{\lieh'}_-(T_\ell)  \, :=  {\textstyle \sum\limits_{i, j \in I}} \psi_{ji} \, a_{j\ell} \, d_i^{-1} \, T_i  \, :=
    {\textstyle \sum\limits_{i, j \in I}} \psi_{ji} \, a_{j\ell} \, H_i
     \quad
  \qquad \forall \;\;\; \ell \in I   \nonumber
\end{eqnarray}
that in matrix notation reads
  $$  \displaylines{
   {\big( \psi^{\lieh'}_+(T_\ell) \big)}_{\! \ell \in I} := A^{\scriptscriptstyle T} \,
   \Psi^{\scriptscriptstyle T} D^{-1} {\big( T_k \big)}_{\! k \in I} = A^{\scriptscriptstyle T}\, \Psi^{\scriptscriptstyle T} {\big( H_k \big)}_{\! k \in I}  \cr
   {\big( \psi^{\lieh'}_-(T_\ell) \big)}_{\! \ell \in I} := A^{\scriptscriptstyle T} \,
   \Psi \, D^{-1} {\big( T_k \big)}_{\! k \in I} =  A^{\scriptscriptstyle T}\, \Psi \, {\big( H_k \big)}_{\! k \in I}  }  $$
   \indent   Note that, by definition, we have  $ \, \psi_+ = \psi_- \, $,  or equivalently  $ \, \psi^{\lieh'}_+ = \psi^{\lieh'}_- \, $,
   if and only if  $ \, \Psi^{\scriptscriptstyle T} = \Psi \, $,  i.e.,  $ \Psi $  is symmetric.
                                                              \par
   Finally we introduce the following elements of  $ \, {\big( \lieh'_{{}_{\mathcal{R}}} \big)}_{\!+} \! \oplus {\big( \lieh'_{{}_{\mathcal{R}}} \big)}_{\!-} \, $,
   for all  $ \, i \in I \, $:
  $$  T^{\scriptscriptstyle \Psi}_{i,+} := \big( \text{id}_{\lieh'_+} \! + \psi_+^{\lieh'} \big)\big(T^+_i\big) - \psi_-^{\lieh'}\big(T^-_i\big) \;\; ,  \quad
   T^{\scriptscriptstyle \Psi}_{i,-} \, := \, \big( \text{id}_{\lieh'_-} \! + \psi_-^{\lieh'} \big)\big(T^-_i\big) - \psi_+^{\lieh'}\big(T^+_i\big)  $$
 \vskip7pt
   The following two results will be of use later on:

\vskip9pt

\begin{lema}  \label{psi_pm-psi_mp-antisym}  {\ }
 \vskip3pt
   {\it (a)}\,  The maps
   $ \; \pm \big( \psi_+ \! - \psi_- \big) : {\big( \lieh'_{{}_{\mathcal{R}}} \big)}^* \relbar\joinrel\relbar\joinrel\relbar\joinrel\longrightarrow
   {\big( \lieh'_{{}_{\mathcal{R}}} \big)}^* \; $  are antisymmetric with respect to the (symmetric) bilinear product
   $ \, (\,\ ,\ ) \, $  on  $ {\big( \lieh'_{{}_{\mathcal{R}}} \big)}^* \, $;
 \vskip2pt
   {\it (b)}\,  The maps
   $ \; \pm \big( \psi_+^{\lieh'} \! - \psi_-^{\lieh'} \big) : \lieh'_{{}_{\mathcal{R}}} \relbar\joinrel\relbar\joinrel\relbar\joinrel\longrightarrow
   \lieh'_{{}_{\mathcal{R}}} \; $  are antisymmetric with respect to the (symmetric) bilinear product  $ \, (\,\ ,\ ) \, $  on
   $ \lieh'_{{}_{\mathcal{R}}} \, $.
\end{lema}

\begin{proof}
 Both claims in the statement follow by sheer computation.  Namely, for the map
 $ \, \big( \psi_+ \! - \psi_- \big) \, $  this gives, for all  $ \, h, k \in I \, $,
  $$  \displaylines{
   \Big( \big( \psi_+ \! - \psi_- \big)(\alpha_h) \, , \, \alpha_k \Big) \, + \, \Big( \alpha_h \, ,
\, \big( \psi_+ \! - \psi_- \big)(\alpha_k) \Big)  \, = \,   \hfill  \cr
   \qquad   = \,  {\textstyle \sum\limits_{i,j \in I}} \Big( \big( \psi_{ij} \, a_{jh} \, d_i^{-1} \,
\alpha_i \, , \, \alpha_k \big) \, - \, \big( \psi_{ji} \, a_{jh} \, d_i^{-1} \, \alpha_i \, , \, \alpha_k \big) \Big) \, +   \hfill  \cr
   \hfill   + \, {\textstyle \sum\limits_{i,j \in I}} \Big( \big( \alpha_h \, , \, \psi_{ij} \, a_{jk} \, d_i^{-1} \, \alpha_i \big) \, -
   \, \big( \alpha_h \, , \, \psi_{ji} \, a_{jk} \, d_i^{-1} \, \alpha_i \big) \Big)  \, = \,   \qquad  \cr
   = \,  {\textstyle \sum\limits_{i,j \in I}} \Big( \psi_{ij} \, a_{jh} \, a_{ik} \, - \, \psi_{ji} \, a_{jh} \, a_{ik} \Big) \, +
   \, {\textstyle \sum\limits_{i,j \in I}} \Big( \psi_{ij} \, a_{jk} \, a_{ih} \, - \, \psi_{ji} \, a_{jk} \, a_{ih} \Big)  \, =  \cr
   = \,  {\Big( A^{\scriptscriptstyle T} \, \Psi A \Big)}_k^h - {\Big( A^{\scriptscriptstyle T} \,
\Psi^{\scriptscriptstyle T} A \Big)}_k^h + {\Big( A^{\scriptscriptstyle T} \,
\Psi^{\scriptscriptstyle T} A \Big)}_k^h - {\Big( A^{\scriptscriptstyle T} \, \Psi A \Big)}_k^h  \, = \,  0  }  $$
where  $ M_k^h $  always denotes the  $ (k,h) $--entry  of a matrix  $ M \, $.  So
$ \, \pm \big( \psi_+ \! - \psi_- \big) \, $  is antisymmetric.  The proof for
$ \, \pm \big( \psi_+^\lieh \! - \psi_-^\lieh \big) \, $  is analogous.
\end{proof}

\vskip7pt

\begin{lema}  \label{lem:posivedef}
Assume that the Cartan matrix  $ A $  is of  \textsl{finite type}.
                                                           \par
   (a)\,  The maps  $ \; \big(\id_{\,\lieh'} \pm \big( \psi_+^{\lieh'} - \psi_-^{\lieh'} \big) \big) : \lieh'_{{}_{\mathcal{R}}}
\relbar\joinrel\relbar\joinrel\longrightarrow \lieh'_{{}_{\mathcal{R}}} \; $  are bijective.
                                                           \par
   (b)\,   The maps  $ \,\; {\big( \lieh'_{{}_{\mathcal{R}}} \big)}_{\!\pm} \!
   \relbar\joinrel\longrightarrow {\big( \lieh'_{{}_{\mathcal{R}}} \big)}_{\!+} \!\oplus {\big( \lieh'_{{}_{\mathcal{R}}} \big)}_{\!-} \; $
   defined by  $ \; T^\pm_\ell \mapsto T^{\scriptscriptstyle \Psi}_{\ell,\pm} \; $ are injective.
\end{lema}

\pf
 Set  $ \, \phi = \big( \psi_+^{\lieh'} \! - \psi_-^{\lieh'} \big) \, $:  this is antisymmetric, hence
 $ \, (\id_{\,\lieh'} \! - \, \phi)(\id_{\,\lieh'} \! + \, \phi) = \big( \id_{\,\lieh'} \! - \, \phi^2 \,\big) = \big( \id_{\,\lieh'} \! +
 \, \phi \, \phi^t \,\big) \, $.
 So the claim in  {\it (a)\/}  is the same as claiming that  $ \, \big( \id_{\,\lieh'} \! + \, \phi \, \phi^t \,\big) \, $  is non-singular,
 which in turn is the same as stating that  $ -1 $  is not an eigenvalue of  $ \, \phi \, \phi^t \, $:
 the claim then follows since the latter always holds true.
                                                                   \par
   The claim in  {\it (b)\/}  is a direct consequence of  {\it (a)}.
\epf
\end{free text}

\smallskip

\begin{free text}{\bf  $ q $--numbers.}  \label{q-numbers}
 Throughout the paper we shall consider several kinds of  ``$ q $--numbers''.
 Let  $ \Zqqm $  be the ring of Laurent polynomials with integral coefficients in the indeterminate  $ q \, $.
 For every  $ \, n \in \NN \, $ we define
  $$  \displaylines{
   {(0)}_q  \, := \;  1 \;\; ,  \quad
    {(n)}_q  \, := \;  \frac{\,q^n -1\,}{\,q-1\,}  \; = \;  1 + q + \cdots + q^{n-1}  \; =
    \; {\textstyle \sum\limits_{s=0}^{n-1}} \, q^s  \qquad  \big(\, \in \, \ZZ[q] \,\big)  \cr
   {(n)}_q!  \, := \;  {(0)}_q {(1)}_q \cdots {(n)}_q  := \,  {\textstyle \prod\limits_{s=0}^n} {(s)}_q  \;\; ,  \quad
    {\binom{n}{k}}_{\!q}  \, := \;  \frac{\,{(n)}_q!\,}{\;{(k)}_q! {(n-k)}_q! \,}  \qquad  \big(\, \in \, \ZZ[q] \,\big)  \cr
   {[0]}_q  := \,  1  \; ,  \;\;
    {[n]}_q  := \,  \frac{\,q^n -q^{-n}\,}{\,q-q^{-1}\,}  =  \, q^{-(n-1)} + \cdots + q^{n-1}  =
    {\textstyle \sum\limits_{s=0}^{n-1}} \, q^{2\,s - n + 1}  \!\quad  \big( \in \Zqqm \,\big)  \cr
   {[n]}_q!  \, := \;  {[0]}_q {[1]}_q \cdots {[n]}_q  = \,  {\textstyle \prod\limits_{s=0}^n} {[s]}_q  \;\; ,  \quad
    {\bigg[\, {n \atop k} \,\bigg]}_q  \, := \;  \frac{\,{[n]}_q!\,}{\;{[k]}_q! {[n-k]}_q!\,}  \qquad  \big(\, \in \, \Zqqm \,\big)  }  $$
 \vskip5pt
\noindent
   In particular, we have the identities
  $$  {(n)}_{q^2} = q^{n-1} {[n]}_q \;\; ,  \qquad  {(n)}_{q^2}! = q^{\frac{n(n-1)}{2}} {[n]}_q  \;\; ,
\qquad  {\binom{n}{k}}_{\!\!q^{2}} = q^{k(n-k)} {\bigg[\, {n \atop k} \,\bigg]}_q  \;\; .  $$
   Furthermore, thinking of Laurent polynomials as functions on  $ \FF^\times \, $,  for any  $ \, q \in \FF^\times \, $
   we shall read every symbol above as representing the corresponding element in  $ \, \FF \, $.
 \end{free text}

\medskip

\subsection{Multiparameters}  \label{subsec: multipar}  \
 \vskip7pt
   The main objects of our study will depend on ``multiparameters'', i.e., suitable collections of parameters.
   Hereafter we introduce these gadgets and the technical results about them that we shall need later on.
\smallskip

\begin{free text}{\bf Multiplicative multiparameters.}  \label{mult-multiparameters}
 Let  $ \FF $  be a fixed ground field, and let  $ \, I := \{1,\dots,n\} \, $  be as in  \S \ref{root-data_Lie-algs}  above.
 We fix a matrix  $ \, \bq := {\big(\, q_{ij} \big)}_{i,j \in I} \, $,  whose entries belong to  $ \FF^\times \, $
 and will play the role of ``parameters'' of our quantum groups.  Then inside the lattice
 $ \, \varGamma := \ZZ^n \, $  one has a ``generalized root system'' associated with the diagonal
 braiding given by  $ \bq \, $,  in which the vectors in the canonical basis of  $ \, \varGamma := \ZZ^n \, $
 are taken as (positive) simple roots  $ \alpha_i $  ($ \, i = 1, \dots, n \, $).
                                                                  \par
   We shall say that the matrix  $ \, \bq \, $  is  {\sl of Cartan type\/}
%
%
 if there is a symmetrisable generalized Cartan matrix  $ \, A = {\big( a_{ij} \big)}_{i,j \in I} \, $  such that
\begin{equation}  \label{qij-ident}
  \qquad \qquad \qquad \qquad \quad   q_{ij} \, q_{ji}  \; = \;  q_{ii}^{\,a_{ij}}
  \qquad \qquad   \forall \;\; i, j \in I  \quad
      \end{equation}
%
%
 In such a case, to avoid some irrelevant technicalities  {\sl we shall assume that the Cartan matrix
 $ A $  is indecomposable\/}:  that is,  $ A $  is not expressible, after any permutation of indices,
 as a block-diagonal matrix with more than one block (although this is  {\sl not\/}  necessary for the theory;
 it only helps in having the main results look better, say).  In particular, this implies that for all  $ \, i, j \in I \, $,
 there exists a sequence  $ \, i = k_1 , \ldots , k_\ell = j \, $  in  $ I $  such that  $ \, a_{k_s,k_{s+1}} \neq 0 \, $
 for all  $ \, 1 \leq s < \ell \, $.
 Moreover   --- still in the Cartan case ---   for later use we fix in  $ \FF $  some
 ``square roots'' of all the  $ q_{ii} $'s,  as follows.  From relations  (\ref{qij-ident})  one easily finds
 --- because the Cartan matrix  $ A $  is indecomposable ---   that  {\sl there exists an index
 $ \, j_{\raise-2pt\hbox{$ \scriptstyle 0 $}} \in I \, $  such that
 $ \, q_{ii} = q_{j_{\raise-2pt\hbox{$ \scriptscriptstyle 0 $}}  j_{\raise-2pt\hbox{$ \scriptscriptstyle 0 $}}}^{\,e _i} \, $
 for some  $ \, e_i \in \NN \, $,  for all  $ \, i \in I \, $}.
Now  {\it we assume hereafter that  $ \FF $  contains a square root of
 $ q_{j_{\raise-2pt\hbox{$ \scriptscriptstyle 0 $}} j_{\raise-2pt\hbox{$ \scriptscriptstyle 0 $}}} \, $,
which we fix throughout and denote by  $ \, q_{j_{\raise-2pt\hbox{$ \scriptscriptstyle 0 $}}}
:= \sqrt{\,q_{j_{\raise-2pt\hbox{$ \scriptscriptstyle 0 $}}
j_{\raise-2pt\hbox{$ \scriptscriptstyle 0 $}}}^{\phantom{o}}} \, $,  and also by
$ \, q := \sqrt{\,q_{j_{\raise-2pt\hbox{$ \scriptscriptstyle 0 $}}
j_{\raise-2pt\hbox{$ \scriptscriptstyle 0 $}}}^{\phantom{o}}} \, \big( = q_{j_0} \big) \, $.}
 Then we set also  $ \, q_i := q_{j_{\raise-2pt\hbox{$ \scriptscriptstyle 0 $}}}^{\,e _i} \, $
 (a square root of  $ q_{ii} \, $)  for all  $ \, i \in I \, $.
 \vskip5pt
   As recorded in  \S \ref{root-data_Lie-algs}
above, we fix positive, relatively prime integers
 $ d_1 $,  $ \dots $,  $ d_n $  such that the diagonal matrix  $ \, D=\text{\sl diag}\,(d_1,\dots,d_n) \, $
 symmetrizes  $ A \, $,  i.e.,  $ \, D\,A \, $  is symmetric; in fact,
these  $ d_i $'s  coincide with the exponents  $ e_i $  mentioned above.
                                               \par
   We introduce now two special cases of Cartan type multiparameter matrices.
 \vskip7pt
   {\sl  $ \underline{\hbox{Integral type}} $:}\,  We say that
   $ \, \bq := {\big(\hskip0,7pt q_{ij} \big)}_{i,j \in I} \, $  is  {\it of integral type\/}
   if it is of Cartan type with associated Cartan matrix  $ \, A = {\big( a_{ij} \big)}_{i, j \in I} \, $,
   and there exist  $ \, p \in \FF^\times \, $  and  $ \, b_{ij} \in \ZZ \, $  ($ \, i, j \in I \, $)  such that
   $ \, b_{ii} = 2\,d_i $  and  $ \, q_{ij} = p^{\,b_{ij}} \, $  for  $ \, i \not= j \in I \, $.
   The Cartan condition  \eqref{qij-ident}  yields  $ \; b_{ij} + \, b_{ji} = 2 \, d_i \, a_{ij} \, $,  \,for  $ \, i, j \in I \, $
   (with  $ d_i $'s  as above).  To be precise,  {\it we say also that  $ \, \bq $  is ``of integral type  $ (\,p\,,B) $''},
   with  $ \, B := {\big(\hskip0,7pt b_{ij} \big)}_{i,j \in I} \; $.
 \vskip5pt

   {\sl  $ \underline{\hbox{Canonical multiparameter}} $:}\,  Given  $ \, q \in \FF^\times \, $
   and a symmetrisable (generalised) Cartan matrix  $ A = {\big( a_{ij} \big)}_{i, j \in I}$,  consider
\begin{equation}  \label{qij-canon}
  \qquad \qquad \qquad \qquad \quad   \check{q}_{ij}  \; := \;  q^{\, d_i a_{ij}}   \qquad \qquad
  \forall \;\; i, j \in I  \quad
\end{equation}
with  $ \, d_i \; (\,i \in I\,) \, $  given as above.  Then these special values of the
$ \, q_{ij} = \check{q}_{ij} \, $'s  do satisfy condition  (\ref{qij-ident}),
hence they provide a special example of matrix  $ \, \bq = \check{\bq} \, $  of
Cartan type, to which we shall refer to hereafter as  {\it the ``$ \, q $--canonical''  case}.
                                                          \par
   Note also that  $ \check{\bq} $  is of integral type  $ (\,q\,,DA) \, $.
 \vskip7pt
   By the way, when the multiparameter matrix  $ \, \bq := {\big(\, q_{ij} \big)}_{i,j \in I} \, $  is  {\sl symmetric},
   i.e., $ \, q_{ij} = q_{ji} \, $  (for all  $ \, i, j \in I \, $),  then the ``Cartan conditions''  $ \, q_{ij} \, q_{ji} = q_{ii}^{a_{ij}} \, $
   read  $ \, q_{ij}^{\,2} = q^{\, 2 \, d_i a_{ij}} \, $,  hence  $ \, q_{ij} = \pm q^{\, d_i a_{ij}} \, $  (for all  $ \, i, j \in I \, $).
   Thus every symmetric multiparameter is ``almost the  $ \, q $--canonical''  one, as indeed it is the
   $ \, q $--canonical  one ``up to sign(s)''.
\end{free text}

\vskip9pt

\begin{free text}{\bf Equivalence and group action for multiparameters.}  \label{equiv-act_x_mprmts}
 Let  $ \, \bq := {\big(\, q_{ij} \big)}_{i,j \in I}$  be a multiparameter matrix.
 The  {\it generalized Dynkin diagram}  $ D(\bq) \, $  associated with  $ \bq $  is a
 {\sl labelled\/}  graph whose set of vertices is  $ I $,  where the  $ i $-th
 vertex is labelled with  $ q_{ii} \, $,  and
there exists
 an edge between the vertices  $ i $  and  $ j $  only if  $ \, \widetilde{q_{ij}} := q_{ij} q_{ji} \neq  1 \, $,
 in which case the edge is decorated by  $ \widetilde{q_{ij}} \, $.  An automorphism of a generalized
 Dynkin diagram  $ D(\bq) \, $  is an automorphisms as a  {\sl labelled\/}  graph.
 In the set  $ \, M_n(\FF^\times) \, $  of all  $ \FF $--valued  multiparameters, we consider the relation
 $ \, \sim \, $  defined by
\begin{equation}  \label{def: twist-equiv x mprmts}
 \bq' \, \sim \, \bq''  \; \iff \;\,  q'_{ij} \, q'_{ji} = q''_{ij} \, q''_{ji} \; , \;\; q'_{ii} = q''_{ii}  \quad \qquad
 \forall \;\;\; i , j \in I := \{1,\dots,n\} \, .
\end{equation}
 This is obviously an equivalence, which is known in the literature as  {\sl twist equivalence}.
Indeed,  $ \, \sim \, $  is nothing but the equivalence relation
 associated with the map  $ \; \bq \mapsto D(\bq) \; $, which to every multiparameter  $ \bq $
 associates its unique generalized Dynkin diagram  $ D(\bq) \, $.
 In particular, if  $ \theta $  is an automorphism of  $ D(\bq) \, $
  --- in the obvious sense ---   then
 we have  $ \, \theta(\bq) := {\big(\, q_{\theta(i),\theta(j)} \big)}_{i, j \in I} \sim \bq \, $.
 \vskip5pt
   Observe also that  $ \, M_n(\FF^\times) \, $  is a group (isomorphic to the direct product of  $ n^2 $
   copies of  $ F^\times \, $)  with the Hadamard product, i.e., the operation   --- denoted  $ \odot $  ---
   is the componentwise multiplication.  For  $ \, \bq = {\big( q_{ij} \big)}_{i,j \in I} \, \in M_n(\FF^\times) \, $,
   we have  that $ \, \bq^{-1} = {\big( q_{ij}^{-1} \big)}_{i,j \in I} \, $;  moreover, we set
   $ \, \bq^T = {\big( q_{ji} \big)}_{i,j \in I} \, $  and
   $ \, \bq^{-T} := {\big( \bq^{-1} \big)}^T = {\big( \bq^T \big)}^{-1} = {\big( q_{ji}^{-1} \big)}_{i,j \in I} \, $.
   The group  $ \, \big(\, M_n(\FF^\times) \, ; \, \odot \,\big) \, $  acts on itself by the adjoint action
\begin{equation}  \label{def: mprmts-acting-on-mprmts}
   \bnu.\,\bq \, := \, \bnu \odot \bq \odot \bnu^{-T}   \qquad \qquad  \forall \;\;\; \bnu \, , \, \bq \, \in M_n(\FF^\times)
\end{equation}
%
%
%
 In down-to-earth terms,
%
%
 this action is given   --- for all
 $ \; \bnu = {\big( \nu_{ij} \big)}_{i, j \in I} \, , \, \bq = {\big( q_{ij} \big)}_{i, j \in I} \, \in M_n(\FF^\times) \, $  ---   by
\begin{equation}  \label{def: EXPL_mprmts-acting-on-mprmts}
   \bnu.\,\bq \, := \, \bq^\bnu \,= \, {\big(\, q^\bnu_{ij} := \nu_{ij} q_{ij} \nu_{ji}^{-1} \,\big)}_{i, j \in I}
\end{equation}
                                                               \par
   The following lemma shows that the equivalence in  $ M_n(\FF^\times) $ induced
   by this action is just the twist equivalence introduced above (the proof is a trivial calculation):

\smallskip

\begin{lema}  \label{lemma: mprmtr-action=>twist-equiv}
 Let  $ \, \tilde{\bq} = {\big( \tilde{q}_{ij} \big)}_{i, j \in I} \, ,
 \hat{\bq} = {\big( \hat{q}_{ij} \big)}_{i, j \in I} \in M_n(\FF^\times) \, $.  Then
  $$  \tilde{\bq} \,\sim\, \hat{\bq}  \,\; \iff \;\,  \exists \;\; \bnu \in M_n(\FF^\times) \, : \,
  \hat{\bq} = \bnu.\,\tilde{\bq}   \leqno \indent \text{(a)}  $$
It follows then that the  $ \, \sim \, $--equivalence  classes coincide with the orbits of the
$ M_n(\FF^\times) $--action  on itself.  In particular, if  $ \; \tilde{\bq} \sim \hat{\bq} \, $  then:
 \vskip5pt
   (b)\;  the  $ \bnu $'s  in  $ M_n(\FF^\times) $  satisfying  $ \; \hat{\bq} = \bnu.\,\tilde{\bq} \, $
   are those such that  $ \; \nu_{ij} \, \nu_{ji}^{-1} =  \hat{q}_{ij} \, \tilde{q}_{ij}^{-1} \; $;
 \vskip5pt
   (c)\;  an explicit  $ \, \bnu $  such that  $ \; \hat{\bq} = \bnu.\,\tilde{\bq} \, $  is given by
  $ \quad  \bnu \, = \, \begin{cases}
              \; \hat{q}_{ij} \, \tilde{q}^{\,-1}_{ij}  &   \quad  \forall \; i \leq j  \\
              \; 1  &  \forall \; i > j
                                    \end{cases} $
 \vskip5pt
   (d)\;  if\/  $ \FF $  contains all square roots of the  $ \tilde{q}_{ij} $'s  and  $ \hat{q}_{ij} $'s,
   and we fix them so that for the multiparameters formed by these square roots we have
   $ \, \tilde{\bq}^{1/2} \sim \hat{\bq}^{1/2} \, $  (as it is always possible) then another explicit  $ \, \nu $
   such that  $ \; \hat{\bq} = \bnu.\,\tilde{\bq} \, $  is given by   $ \; \bnu = \hat{\bq}^{1/2} \, \tilde{\bq}^{-1/2} \; $.
   \qed
\end{lema}

\medskip

   For multiparameters of Cartan type we also have the following:

\medskip

\begin{prop}
 Each multiparameter\/  $ \bq $  of Cartan type is  $ \, \sim $--equivalent
 to some multiparameter\/  $ \check{\bq} $  of canonical type.
\end{prop}

\pf
 Let  $ \, \bq = {\big( q_{ij} \big)}_{i, j \in I} \in M_n(\FF^\times) \, $  be a multiparameter of Cartan type.
 By definition,   --- cf.\ \S \ref{mult-multiparameters}  ---   this means that there exists  $ \, q \in \FF^\times \, $
 and a (generalised) symmetrisable Cartan matrix  $ \, A = {\big( a_{ij} \big)}_{i, j \in I} \, $  such that
 $ \; q_{ij} \, q_{ji} \, = \, q_{ii}^{a_{ij}} \; $   --- see  \eqref{qij-ident}  ---   and  $ \, q_{ii} = q^{2\,d_i} \, $
 for all  $ \, i, j \in I \, $;  altogether these imply
  $$  q_{ij} \, q_{ji}  \; = \;  q_{ii}^{a_{ij}}  \; = \;  q^{2\,d_i{}a_{ij}}  \; = \;  q^{d_i{}a_{ij}} \, q^{d_j{}a_{ji}}  \;
  = \; \check{q}_{ij} \, \check{q}_{ji}   \eqno  \forall \;\; i, j \in I   \qquad  $$
while  $ \, q_{ii} = q^{2\,d_i} \, $  also reads  $ \; q_{ii}= q^{2\,d_i} = \check{q}_{ii} \, $,  so that in the end we have
$ \, \bq \sim \check{\bq} \, $  with  $ \, \check{\bq} := {\big( q^{d_i{}a_{ij}} \big)}_{i, j \in I} \, $
a multiparameter of canonical type as in the claim.
\epf

\end{free text}

\medskip

 \subsection{Conventions for Hopf algebras} \label{conv-Hopf}  \
 \vskip7pt
   Our main references for the theory of Hopf algebras are  \cite{Mo}  and  \cite{Ra}.
   We use standard notation: the comultiplication is denoted  $ \com $  and the antipode  $ \SS \, $.
   For the first, we use the Sweedler-Heyneman notation but with the summation sign suppressed,
   namely we write the coproduct as  $ \, \com(x) = x_{(1)} \otimes x_{(2)} \, $.
                                                              \par
   Although in most parts of the paper it is not needed, we assume that the antipode  $ \SS $
   is bijective and denote by $ \SS^{-1} $  its composition inverse.
                                                              \par
   Hereafter by  $ \k $  we denote the ground ring of our algebras, coalgebras, etc.
                                                              \par
   In any coalgebra  $ C \, $,  the set of group-like elements is denoted by  $ G(C) \, $;
   also, we denote by  $ \, C^+ := \Ker(\epsilon) \, $  the augmentation ideal, where
   $ \, \epsilon : C \longrightarrow \k \, $  is the counit map.
   If  $ \, g, h \in G(C) \, $,  the set of  $ (g,h) $--primitive  elements is defined as
  $ \; P_{g,h}(C)  \, := \,  \big\{\, x \in C \,\vert\, \com(x) = x \ot g + h \ot x \,\big\} \; $.
                                                                      \par
   If  $ H $  is a Hopf algebra (or just a bialgebra),  we write  $ H^{\text{op}} \, $,  resp.\  $ H^{\text{cop}} \, $,
   for the Hopf algebra (or bialgebra) given by taking in  $ H $  the opposite product, resp.\ coproduct.
 \vskip3pt
   Finally, we recall the notion of  {\it skew-Hopf pairing\/}  between two Hopf algebras (taken from  \cite{AY},  \S 2.1,
   but essentially standard):

\smallskip

\begin{definition}  \label{def_(skew-)Hopf-pairing}
 Given two Hopf algebras  $ H $  and  $ K $  with bijective antipode over the ring  $ \k \, $,
 a  $ \k $--linear  map  $ \, \eta : H \otimes_\k K \longrightarrow \k \, $  is called a  {\sl skew-Hopf pairing\/}
 (between  $ H $  and  $ K \, $)  if, for all  $ \, h \in H \, $,  $ \, k \in K \, $,  one has
  $$  \displaylines{
   \eta\big( h \, , \, k' \, k'' \big) \, = \, \eta\big( h_{(1)} \, , \, k' \big) \, \eta\big( h_{(2)} \, , \, k'' \big)  \;\; ,
   \qquad
 \eta\big( h' \, h'' \, , \, k \big) \, = \, \eta\big( h' \, , \, k_{(2)} \big) \, \eta\big( h'' \, , \, k_{(1)} \big)  \cr
   \eta\big( h \, , 1 \big) \, = \, \epsilon(h)  \;\; ,  \;\quad  \eta\big( 1 \, , k \big) \, = \, \epsilon(k)  \quad ,  \;\;\qquad
   \eta\big( \SS^{\pm 1}(h) \, , k \big) \, = \, \eta\big( h \, , \SS^{\mp 1}(k) \big)  }  $$
 \vskip3pt
\end{definition}

\vskip3pt

   Recall that, given two Hopf  algebras  $ H $  and  $ K $ over  $ \k \, $,
   and a skew-Hopf pairing, say  $ \, \eta: H\otimes_\k K \relbar\joinrel\relbar\joinrel\longrightarrow \k \, $,
   the  {\it Drinfeld double\/}  $ D(H,K,\eta) $  is the quotient algebra  $ \, T(H \oplus K) \big/ \II \, $
   where  $ \II $  is the (two-sided) ideal generated by the relations
  $$  \displaylines{
   \qquad   1_H  = \,  1  \, = \,  1_K  \,\; ,  \qquad  a \otimes b  \, = \, a\,b
   \qquad \qquad  \forall  \;\;  a \, , b \in H  \text{\;\;\ or\;\;\ }  a \, , b \in K \;\; ,  \cr
   x_{(1)} \otimes y_{(1)} \ \eta(y_{(2)},x_{(2)})  \, =
   \,  \eta(y_{(1)},x_{(1)})\ y_{(2)} \otimes x_{(2)} \qquad  \forall \;\; x \in K \, ,  \; y \in H  \;\; ;  }  $$
such a quotient  $ \k $--algebra  is also endowed with a standard Hopf algebra structure,
which is  {\sl consistent},  in that both  $ H $  and  $ K $  are Hopf  $ \k $--subalgebras  of it.

\medskip

 \subsection{Hopf algebra deformations}  \label{defs_Hopf-algs}  \
 \vskip7pt
   There exist two standard methods to deform Hopf algebras,
usually called  ``$ 2 $--cocycle  deformations'' and ``twist deformations'': we
  shortly recall both, giving them new names;
 then later on in the paper we shall apply them to quantum groups.

\vskip11pt

\begin{free text}  \label{cocyc-defs}  \
{\bf Multiplication twistings (or ``2-cocycle deformations'').}
 Let us consider a bialgebra  $ \, \big( H, m, 1, \Delta, \epsilon \big) \, $  over a ring  $ \k \, $.
 A convolution invertible linear map  $ \sigma $  in  $ \, \Hom_{\Bbbk}(H \otimes H, \k \,) \, $
 is called a  {\it normalized multiplicative (or Hopf) 2-cocycle\/}  if
  $$  \sigma(b_{(1)},c_{(1)}) \, \sigma(a,b_{(2)}c_{(2)})  \; = \;
   \sigma(a_{(1)},b_{(1)}) \, \sigma(a_{(2)}b_{(2)},c)  $$
and  $ \, \sigma (a,1) = \eps(a) = \sigma(1,a) \, $  for all  $ \, a, b, c \in H \, $,  see  \cite[Sec.\ 7.1]{Mo}.
We will simply call it a  $ 2 $--cocycle  if no confusion arises.
                                                                        \par
   Using a  $ 2 $--cocycle  $ \sigma $  it is possible to define a new algebra structure on  $ H $
   by deforming the multiplication.  Indeed, define
   $ \, m_{\sigma} = \sigma * m * \sigma^{-1} : H \otimes H \longrightarrow H \, $  by
  $$  m_{\sigma}(a,b)  \, = \,  a \cdot_{\sigma} b  \, = \,  \sigma(a_{(1)},b_{(1)}) \, a_{(2)} \, b_{(2)} \,
  \sigma^{-1}(a_{(3)},b_{(3)})  \eqno \forall \;\, a, b \in H  \quad  $$
 If in addition  $ H $  is a Hopf algebra with antipode  $ \SS \, $,  then define also
 $ \, \SS_\sigma : H \longrightarrow H \, $  as  $ \, \SS_{\sigma}  : H \longrightarrow H \, $
where
  $$  \SS_{\sigma}(a)  \, = \,  \sigma(a_{(1)},\SS(a_{(2)})) \, \SS(a_{(3)}) \, \sigma^{-1}(\SS(a_{(4)}),a_{(5)})
  \eqno \forall \;\, a \in H  \quad  $$
 It is known  that  $ \, \big( H, m_\sigma, 1, \Delta, \eps \big) \, $  is in turn a bialgebra, and
 $ \, \big( H, m_\sigma, 1, \Delta, \eps, \SS_\sigma \big) \, $
 is a Hopf algebra (see \cite{DT} for details): this new bialgebra or Hopf algebra structure on
  $ H $,  graphically denoted by  $ H_\sigma \, $,  is usually called ``cocycle deformation'' of the old one;
  we adopt instead hereafter the terminology  \textsl{``multiplication twisting''}.
\end{free text}

\vskip11pt

\begin{free text}  \label{twist-defs}
 {\bf Comultiplication twistings (or ``twist deformations'').}
 Let  $ H $  be a Hopf algebra (over a commutative ring), and let  $ \, \F \in H \otimes H \, $
 be an invertible element in  $ H^{\otimes 2} $  (later called a ``{\it twist\/}'',  or  ``{\it twisting element\/}'')
 such that
  $$  \F_{1{}2} \, \big( \Delta \otimes \text{id} \big)(\F) \, = \,  \F_{2{}3} \, \big( \text{id}
  \otimes \Delta \big)(\F)  \quad ,  \qquad  \big( \epsilon \otimes \text{id} \big)(\F)
  = 1 = \big( \text{id} \otimes \epsilon \big)(\F)  $$
 Then  $ H $  bears a second Hopf algebra structure, denoted  $ H^\F $,
 with the old product, unit and counit, but with new ``twisted'' coproduct  $ \Delta^\F $
 and antipode  $ \SS^\F $  given by
\begin{equation}  \label{twist-coprod}
    \qquad   \Delta^{\!\F}(x)  \, := \,  \F \, \Delta(x) \, \F^{-1}  \quad ,  \qquad  \SS^\F(x) \, := \, v\,\SS(x)\,v^{-1}
    \quad \qquad  \forall \;\; x \in H
\end{equation}
where  $ \, v := \sum_\F \SS(f'_1)\,f'_2 \, $   --- with  $ \, \sum_\F f'_1 \otimes f'_2 = \F^{-1} \, $
---   is invertible in  $ H \, $;  see  \cite{CP},  \S 4.2.{\sl E},  and references therein, for further details.
 Note also that if  $ H $  is just a bialgebra then the procedure applies as well, with  $ H^\F $  being just a bialgebra.
                                                                      \par
   The Hopf algebra (or just bialgebra)  $ H^\F $  is usually said to be a ``twist deformation'' of  $ H \, $:
   we adopt here instead the terminology ``comultiplication twisting'' to denote both this
   deformation procedure and its final outcome  $ H^F $.
\end{free text}

\vskip10pt

\begin{free text}  \label{deforms-vs-duality}
 {\bf Deformations and duality.}  The two notions of  ``$ 2 $--cocycle''  and of ``twist'' are
 dual to each other with respect to Hopf duality.  In detail, we have the following:

\vskip11pt

\begin{prop}  \label{prop: duality-deforms}
 Let  $ H $  be a Hopf algebra over a field, $ H^* $  be its linear dual, and  $ H^\circ $  its Sweedler dual.
 \vskip3pt
   {\it (a)}\,  Let  $ \F $  be a twist for  $ H \, $,  and  $ \sigma_{{}_\F} $  the image of
   $ \F $  in $ {(H^{*} \otimes H^{*})}^* $  for the natural composed embedding
   $ \, H \otimes H \lhook\joinrel\relbar\joinrel\longrightarrow H^{**}
   \otimes H^{**} \lhook\joinrel\relbar\joinrel\longrightarrow {\big( H^* \otimes H^* \big)}^* \, $.
   Then  $ \sigma_{{}_\F} $  is a  $ 2 $--cocycle  for  $ H^\circ \, $,
   and there exists a canonical isomorphism
   $ \, {\big( H^\circ \big)}_{\sigma_{{}_\F}} \cong {\big( H^\F \,\big)}^\circ \, $.
 \vskip3pt
   {\it (b)}\,  Let  $ \sigma $  be a  $ 2 $--cocycle  for  $ H \, $;
   assume that  $ H $  is finite-dimensional, and let  $ \F_\sigma $
   be the image of  $ \sigma $  in the natural identification
   $ \, {(H \otimes H)}^* = H^* \otimes H^* \, $.  Then  $ \F_\sigma $  is a twist for  $ H^* \, $,
   and there exists a canonical isomorphism
   $ \, {\big( H^* \big)}^{\F_\sigma} \cong {\big( H_\sigma \big)}^* \, $.
\end{prop}

\begin{proof}
 The proof is an exercise in Hopf duality theory, left to the reader
 --- a matter of reading the identities that characterize a twist or a  $ 2 $--cocycle  in dual terms.
\end{proof}

\vskip9pt

   As a final remark, let us mention that the notions of  $ 2 $--cocycle  and twist and of the associated deformations
 (by ``twisting'', as precised above)
 can also be extended to more general contexts where ``Hopf algebra'' has a broader meaning
 --- essentially, taking place in more general tensor categories than vector spaces over a field
 with algebraic tensor product.  In these cases, if a suitable notion of Hopf ``duality'' is established
 (possibly involving different tensor categories), then  Proposition \ref{prop: duality-deforms}  still makes sense,
 with no need of the finite-dimensionality assumption in claim  {\it (b)}.  For instance, this applies to
 Drinfeld's ``quantized universal enveloping algebras''  $ \uhg $  and their
 dual ``quantized formal series Hopf algebras''  $ F_\hbar[G] \, $,
 that both are topological Hopf algebras (with respect to different topologies) over  $ \kh \, $.
\end{free text}

\bigskip

 \section{Quantum groups (as QUEA's) and their comultiplication twistings}
\label{form-MpQUEA's_deform's}

\smallskip

   In this section we briefly recall the notion of formal and polynomial ``quantized universal enveloping algebra''
   (or QUEA in short), following Drinfeld and Jimbo-Lusztig. Then we discuss their
comultiplication twistings (via so-called
 ``toral twists''), and we provide an alternative presentation of these twisted QUEA's.

\medskip

\subsection{Formal QUEA's (``\`a la Drinfeld'')}  \label{memo_formal-QUEAs}
 {\ }

\smallskip

  We begin by the description of  formal quantized universal enveloping algebra following Drinfeld and others,
  and some related tools.

\vskip11pt

\begin{free text}  \label{intro-Uhg}
 {\bf The formal QUEA  $ \uhg \, $.}  Let  $ \kh $  be the ring of formal power series in
 $ \hbar \, $. Let $ \, A := {\big(\, a_{ij} \big)}_{i, j \in I} \, $  be a generalized symmetrisable Cartan matrix.
 The quantized universal enveloping algebra, or QUEA in short,  $ \, U_\hbar\big(\lieg'(A)\big) = \uhg \, $
 is the associative, unital, topologically complete  $ \kh $--algebra  with generators
 $ H_i \, $, $ E_i \, $  and  $ F_i $  ($ \, i \in I := \{1,\dots,n\} \, $)  and relations (for all  $ i, j \in I $)
 \begin{eqnarray}
   H_{i} H_{j} - H_{j} H_{i}\; = \;  0  \;\; ,  \quad
   H_i E_j - E_j H_i  \; = \;  +a_{ij} \, E_j  \;\; ,  \quad  H_i F_j - F_j H_i  \; = \;  -a_{ij} \, F_j  \nonumber  \\
   E_i F_j - F_j E_i  \; = \;  \delta_{ij} \, {{\, e^{+\hbar \, d_i H_i} - e^{-\hbar \, d_i H_i} \,}
\over {\, e^{+\hbar \, d_i} - e^{-\hbar \, d_i} \,}}   \qquad \qquad \qquad \qquad \qquad   \nonumber  \\
   \sum\limits_{k = 0}^{1-a_{ij}} (-1)^k {\left[ { 1-a_{ij} \atop k }
\right]}_{q_i} E_i^{1-a_{ij}-k} E_j E_i^k  \; = \;  0   \quad \qquad   (\, i \neq j \,)   \qquad \quad
 \label{eq:QUEA-Drinfeld-qSerre1}
 \end{eqnarray}
 \begin{eqnarray}
     \qquad \qquad \quad
   \sum\limits_{k = 0}^{1-a_{ij}} (-1)^k {\left[ { 1-a_{ij} \atop k }
\right]}_{q_i} F_i^{1-a_{ij}-k} F_j F_i^k  \; = \;  0   \quad \qquad   (\, i \neq j \,)   \quad
 \label{eq:QUEA-Drinfeld-qSerre2}
 \end{eqnarray}
 where hereafter we use shorthand notation  $ \, e^X := \exp(X) \, $,
 $ \, q:= e^\hbar \, $,  $ \, q_i:= q^{d_i} = e^{\hbar \, d_i} \, $.  It is known that  $ \, \uhg \, $
 has a structure of (topological) Hopf algebra, given by
  $$  \displaylines{
   \Delta(E_i) := E_i \otimes 1 + e^{+\hbar \, d_i H_i} \otimes E_i  \;\; ,  \qquad
 \SS(E_i) := - e^{-\hbar \, d_i H_i} E_i  \;\; ,  \qquad  \epsilon(E_i) := 0  \cr
   \Delta(H_i) := H_i \otimes 1 + 1 \otimes H_i  \;\quad ,  \;\;\qquad
 \SS(H_i):= -H_i \;\quad ,  \;\;\qquad  \epsilon(H_i) := 0  \cr
   \Delta(F_i) := F_i \otimes e^{-\hbar \, d_i H_i} + 1 \otimes F_i  \;\; ,  \;\qquad
 \SS(F_i) := - F_i \, e^{+\hbar \, d_i H_i}  \;\; ,  \;\qquad  \epsilon(F_i) := 0  }  $$
 (for all  $ \, i \in I \, $)  where the coproduct takes values in the  $ \hbar$--adic
 completion  $ {\uhg}^{\widehat{\otimes}\, 2} $  of the algebraic tensor square
 $ {\uhg}^{\otimes 2} \, $   --- see, e.g.,  \cite{CP}  (and references therein) for details,
 taking into account that we adopt slightly different normalizations.
 \vskip3pt
   Finally, the semiclassical ``limit''   $ \, \uhg \Big/ \hbar\,\uhg \, $  of  $ \uhg $  is isomorphic, as a Hopf algebra, to
   $ U(\hskip0,8pt\lieg) \, $   --- via
   $ \, E_i \mapsto \erm_i \, $,  $ \, H_i \mapsto \hrm_i \, $,  $ \, F_i \mapsto \frm_i \, $
   (notation of  \S \ref{root-data_Lie-algs})  ---   hence  $ \uhg $  is a deformation quantization of
   $ U(\hskip0,8pt\lieg'(A)) \, $,  the universal enveloping algebra of the derived algebra of the Kac-Moody
   Lie algebra  $ \lieg(A) \, $.  Then the latter is endowed with a structure of  {\sl co-Poisson\/}
   Hopf algebra, which makes  $ \lieg'(A) $  into a Lie bialgebra.
\end{free text}

 \vskip1pt
\begin{rmk}  \label{rmk: from U_h(g') to U_h(g)}
 Clearly, one can slightly modify the definition of  $ \, U_\hbar\big(\lieg'(A)\big) = \uhg \, $
 by allowing additional generators corresponding to the elements of the  $ \k $--basis
 of any vector space complement  $ \lieh'' $  of  $ \lieh' $  inside  $ \lieg $
 (notation of  \S \ref{root-data_Lie-algs}),  along with the obvious, natural relations:
 this provides a larger (if $ \, \lieg' \subsetneqq \lieg \, $)  quantum Hopf algebra which is a formal quantization of
 $ U_\hbar\big(\lieg(A)\big) $  and so endows the latter with a structure of  {\sl co-Poisson\/}
 Hopf algebra, thus making  $ \lieg(A) $  into a Lie bialgebra for whom  $ \lieg'(A) $  is then a Lie sub-bialgebra.
 {\it For this larger quantum Hopf algebra one can repeat,  {\sl up to minimal changes},  all what we do from now on}.
\end{rmk}

\vskip5pt

\begin{free text}  \label{Cartan-Borel_form-QUEAs}
 {\bf Quantum Borel (sub)algebras and their Drinfeld double.}
 We denote hereafter by  $ \uhh \, $,  resp.\  $ \uhbp \, $,  resp.\  $ \uhbm \, $,  the  $ \hbar $--adically
 complete subalgebra of  $ \uhg $  generated by all the  $ H_i $',  resp.\  the  $ H_i $'  and the
 $ E_i $'s,  resp.\ the  $ H_i $'  and the  $ F_i $'s.  We refer to  $ \uhh $  as  {\sl quantum Cartan (sub)algebra\/}
 and to  $ \uhbp \, $,  resp.\  $ \uhbm \, $,  as  {\sl quantum positive},  resp.\  {\sl negative},  {\sl Borel (sub)algebra}.
 It follows directly from definitions that  {\sl $ \uhh \, $,  $ \uhbp $  and  $ \uhbm $  all are Hopf subalgebras of
 $ \uhg \, $}.
                                                                        \par
   It is also known that the quantum Borel subalgebras are related via a skew-Hopf pairing
   $ \; \eta : \uhbp \otimes_{\kh} \uhbm \relbar\joinrel\longrightarrow \kh \; $  given by
\begin{equation}  \label{h-Borel_pairing}
   \eta\big(H_{i\,},H_j\big) = -d_i a_{ij}  \; ,  \,\;  \eta\big(E_{i\,},\!F_j\big) = \delta_{ij} \, {{\,1\,}
   \over {\,q_i^{-1\,} \!\! - q_i\,}}  \; ,  \,\;  \eta\big(E_{i\,},H_j\big) = 0 = \eta\big(H_{i\,},\!F_j\big)
\end{equation}
 Using this skew-Hopf pairing, one considers  as in  \S \ref{conv-Hopf}  the corresponding Drinfeld double
 $ \, D\big( {\uhbp} , \uhbm , \eta \big)$;  in the sequel we denote the latter by  $ \, \uhgd \,$.
                                                           \par
   By construction, there is a Hopf algebra epimorphism
   $ \; \pi_\lieg : \uhgd \relbar\joinrel\relbar\joinrel\twoheadrightarrow \uhg \; $.
   In order to describe it, we use the identification
  $$  \uhgd \, = \, D\big( {\uhbp}, \uhbm , \eta \big)  \; \cong \;  {\uhbp} \otimes_{\kh} \uhbm  $$
 as coalgebras, and adopt such shorthand notation as  $ \, E_i = E_i \otimes 1 \, $,
 $ \, H^+_i = H_i \otimes 1 \, $,  $ \, H^-_i = 1 \otimes H_i \, $,  $ \, F_i = 1 \otimes F_i \, $
 (for all  $ \, i \in I \, $);  then the projection  $ \pi_\lieg $  is determined by
\begin{equation}  \label{proj_Uhgd-->Uhg}
  \pi_\lieg \, :  \;\quad   E_i \mapsto E_i \; ,  \quad  H^+_i \mapsto H_i \; ,
  \quad  H^-_i \mapsto H_i \; ,  \quad  F_i \mapsto F_i  \qquad  \forall \;\; i \in I
\end{equation}
                                                        \par
   Furthermore,  $ \uhgd $  can be explicitly described as follows: it is the associative, unital,
   topologically complete  $ \kh $--algebra  with generators  $ E_i \, $,  $ H^+_i \, $,  $ H^-_i $  and
   $ F_i $  ($ \, i \in I := \{1,\dots,n\} \, $)  satisfying the relations  \eqref{eq:QUEA-Drinfeld-qSerre1},
   \eqref{eq:QUEA-Drinfeld-qSerre2}  and the following:
  $$  \displaylines{
   H^+_i H^+_j - H^+_j H^+_i \; = \; 0  \;\; ,  \quad  H^+_i H^-_j - H^-_j H^+_i \; = \; 0  \;\; ,  \quad
   H^-_i H^-_j - H^-_j H^-_i \; = \; 0  \cr
   H^\pm_i E_j - E_j H^\pm_i  \; = \;  \pm a_{ij} \, E_j   \quad ,  \qquad  H^\pm_i F_j - F_j H^\pm_i
   \; = \;  \mp a_{ij} \, F_j  \cr
     E_i F_j - F_j E_i  \; = \;  \delta_{ij} \, {{\, e^{+ \hbar \, d_i H^+_i} - e^{- \hbar \, d_i H^-_i} \,}
   \over {\, e^{+\hbar \, d_i} - e^{-\hbar \, d_i} \,}}  \cr}  $$
 In addition, the structure of (topological) Hopf algebra of  $ \uhgd $  is given
by
  $$  \displaylines{
   \Delta(E_i) \! := E_i \otimes 1 + e^{+\hbar \, d_i H^+_i} \otimes E_i  \;\;\; ,  \qquad
   \SS(E_i) := - e^{-\hbar \, d_i H^+_i} E_i  \;\;\; ,  \qquad  \epsilon\,(E_i) := 0  \cr
   \Delta\big(H^\pm_i\big) := H^\pm_i \otimes 1 + 1 \otimes H^\pm_i  \;\quad ,  \;\;\qquad
   \SS\big(H^\pm_i\big):= -H^\pm_i  \;\quad ,  \;\;\qquad  \epsilon\,\big(H^\pm_i\big) := 0  \cr
   \Delta(F_i) := F_i \otimes e^{-\hbar \, d_i H^-_i} + 1 \otimes F_i  \;\;\; ,  \qquad
   \SS(F_i) := - F_i \, e^{+\hbar \, d_i H^-_i}  \;\;\; ,  \qquad  \epsilon\,(F_i) := 0  }  $$
 \vskip3pt
   By construction, both  $ \uhbp $  and  $ \uhbm $  are Hopf subalgebras of  $ \uhgd \, $:  the natural embeddings
   $ \, \uhbp \lhook\joinrel\longrightarrow \uhgd \, $  and  $ \, \uhbm \lhook\joinrel\longrightarrow \uhgd \, $  are described by
  $$  E_i \mapsto E_i \; ,  \quad  H_i \mapsto H^+_i   \qquad \text{and} \qquad
  F_i \mapsto F_i \; ,  \quad  H_i \mapsto H^-_i   \eqno \forall \;\; i \in I  \qquad  $$
 \vskip5pt
   Like for  $ \uhg \, $,  if we look at the semiclassical limit of  $ \uhh \, $,  resp.\  of  $ \uhbp \, $,
   resp.\  of  $ \uhbm \, $,  resp.\  of  $ \uhgd \, $,  we find  $ U\big(\hskip0,8pt\lieh'(A)\big) \, $,
   resp.\  $ U\big(\hskip0,8pt\lieb_+'(A)\big) \, $,  resp.\  $ U\big(\hskip0,8pt\lieb_-'(A)\big) \, $,
   resp.\  $ U\big(\hskip0,8pt\liegd'(A)\big) \, $,  where  $ \liegd'(A) $  is the Manin double (see  \S \ref{root-data_Lie-algs}).
   This entails that all these universal enveloping algebras also are  {\sl co-Poisson\/}
   Hopf algebras, whose co-Poisson structure in the first three cases is just the restriction of that of
   $ U\big(\hskip0,8pt\lieg'(A)\big) \, $;  in particular, all of  $ \lieh'(A) \, $,  $ \lieb'_+(A) $  and  $ \lieb'_-(A) $
   are Lie sub-bialgebras of  $ \lieg'(A) \, $.
                                                      \par
 Finally, a suitably modified version of  Remark \ref{rmk: from U_h(g') to U_h(g)} applies to  $ \uhgd $  as well.
\end{free text}

\smallskip

  \subsection{Comultiplication twistings of formal QUEA's}  \label{twist_formal-QUEAs}
 {\ }

\smallskip

   Following an idea of Reshetikhin  (cf.\  \cite{Re}),  we shall consider special twisting elements for
   $ \uhg $  and use them to provide a new, twisted QUEA denoted  $ \upsihg \, $.
   Then we shall extend the same method to  $ \uhgd $  as well.

\vskip9pt

\begin{free text}  \label{twist-Uhg}
{\bf Comultiplication twistings of  $ \uhg \, $.}
 Let us consider an  $ (n \times n) $--matrix  $ \; \Psi := {\big( \psi_{ij} \big)}_{i,j \in I} \in M_n\big(\kh\big) \; $.
 A straightforward check shows that the element
\begin{equation}  \label{Resh-twist_F-uhg}
  \F_\Psi  \,\; := \;\,  \exp\bigg( \hbar \, {\textstyle \sum\limits_{i,j=1}^n} \psi_{ij} \, H_i \otimes H_j \bigg)
\end{equation}
is actually a  {\sl twist\/}  of  $ U_\hbar(\lieg) $  in the sense of  \S \ref{twist-defs}.
 Indeed,
 since  $ U_\hbar(\liehd) $  is commutative, all computations follow from the simple fact that
 $ \; \exp(\hbar a) \exp(\hbar b) = \exp (\hbar(a+b)) \; $  for all commuting elements  $ a $  and  $ b \, $;
 for instance, this implies that  $ \F_\Psi $  is invertible in  $ \uhgd^{\widehat{\otimes}\, 2} $  with inverse
 $ \; \F_\Psi^{-1} \, := \, \exp\bigg(\!-\hbar\,{\textstyle \sum\limits_{i,j=1}^n} \psi_{ij} \, H_i \otimes H_j \bigg) \; $.
 Also, since  $ \epsilon $  is an algebra map (which is continuous in the  $ \hbar $--adic topology) we have that
  $$  \begin{aligned}
   \big(\epsilon \otimes \id\big)(\F_\Psi)  &  \; = \;  \exp\bigg(\hbar \, {\textstyle \sum\limits_{i,j=1}^n} \psi_{ij}
   \, \epsilon (H_i) \otimes H_j \bigg)  \; = \;  1  \\
   \big(\id \otimes \epsilon\big)(\F_\Psi)  &  \; = \;  \exp\bigg(\hbar \, {\textstyle \sum\limits_{i,j=1}^n} \psi_{ij}
   \, H_i \otimes\epsilon ( H_j )\bigg)  \; = \;  1
      \end{aligned}   \leqno  \textrm{and}  $$
 Finally, let us check that  $ \F_\Psi $  satisfies the cocycle condition.  Since  $ \com $
 is also an algebra map (again, continuous in the  $ \hbar $--adic  topology), we have
  $$  \begin{aligned}
   &  (\F_{\Psi})_{1{}2} \, \big( \Delta \otimes \text{id} \big)(\F_\Psi)  \, = \,
 \exp\!\bigg( \hbar {\textstyle \sum\limits_{i,j=1}^n} \psi_{ij} \, H_i \otimes H_j \otimes 1 \!\bigg) \,
  \exp\!\bigg( \hbar {\textstyle \sum\limits_{i,j=1}^n} \psi_{ij} \, \com\big(H_i\big) \otimes H_j \!\bigg)  \\
   &  \qquad  = \,  \exp\!\bigg( \hbar \, {\textstyle \sum\limits_{i,j=1}^n} \psi_{ij} \,
   \big( H_i \otimes H_j \otimes 1 + H_i \ot 1 \otimes H_j + 1 \otimes H_i \otimes H_j \big) \!\bigg)
 \\
   &  (\F_{\Psi})_{2{}3} \, \big( \text{id} \otimes \Delta \big)(\F_\Psi)  \, = \,
 \exp\!\bigg( \hbar {\textstyle \sum\limits_{i,j=1}^n} \psi_{ij} \, 1 \otimes H_i \otimes H_j \!\bigg) \,
  \exp\!\bigg( \hbar {\textstyle \sum\limits_{i,j=1}^n} \psi_{ij} \, H_i \otimes \com\big(H_j\big) \!\bigg)  \\
   &  \qquad  = \,  \exp\!\bigg( \hbar \, {\textstyle \sum\limits_{i,j=1}^n} \psi_{ij} \,
   \big( 1 \otimes H_i \otimes H_j + H_i \otimes H_j \otimes 1 + H_i \otimes 1 \otimes H_j \big) \!\bigg)
 \end{aligned}  $$
 so that  $ \; (\F_{\Psi})_{1{}2} \, \big( \Delta \otimes \text{id} \big)(\F_\Psi) \, = \, (\F_{\Psi})_{2{}3} \,
 \big( \text{id} \otimes \Delta \big)(\F_\Psi) \; $  as expected.
                                                                 \par
   Therefore, the recipe in  \S \ref{twist-defs}  endows  $ U_\hbar(\lieg) $  with a new,
``twisted'' coproduct  $ \, \Delta^{\scriptscriptstyle (\Psi)} := \Delta^{\!\F_\Psi} \, $
which gives us a new Hopf structure  $ U_\hbar(\lieg)^{\F_\Psi} \, $.  An easy computation proves
that the new coproduct reads as follows on generators (for  $ \, \ell \in I \, $):
  $$  \displaylines{
   \Delta^{\scriptscriptstyle \!(\Psi)}\!\big(E_\ell\big)  \, = \,
   E_\ell \otimes e^{+\hbar \sum_{i, j \in I} \psi_{ij} \, a_{i\ell} H_j} \, + \,
   e^{+\hbar \, d_\ell H_\ell \, + \, \hbar \sum_{i, j \in I} \psi_{ij} \, a_{j\ell} H_i} \otimes E_\ell  \cr
   \Delta^{\scriptscriptstyle \!(\Psi)}\big(H_\ell\big)  \, = \,  H_\ell \otimes 1 \, + \, 1 \otimes H_\ell  \cr
   \Delta^{\scriptscriptstyle \!(\Psi)}\!\big(F_\ell\big)  \, = \,
   F_\ell \otimes e^{-\hbar \, d_\ell H_\ell - \hbar \sum_{i, j \in I} \psi_{ij} \,
   a_{i\ell} H_j} \, + \, e^{-\hbar \sum_{i, j \in I} \psi_{ij} \, a_{j\ell} H_i} \otimes F_\ell  }  $$
Using notation of  \S \ref{root-twisting}  with  $ \, \mathcal{R} := \kh \, $  and
$ \, T_\ell = d_\ell H_\ell \, $  for all  $ \, \ell \in I \, $,  these formulas read  (for  $ \, \ell \in I \, $)
  $$  \displaylines{
   \Delta^{\scriptscriptstyle \!(\Psi)}\!\big(E_\ell\big)  \, = \,
   E_\ell \otimes e^{+\hbar \, \psi_-^{\lieh'}(T_\ell)} \, + \,
   e^{+\hbar \, \big( \id_{\lieh'} + \, \psi_+^{\lieh'} \big)(T_\ell)} \otimes E_\ell  \cr
   \Delta^{\scriptscriptstyle \!(\Psi)}\big(H_\ell\big)  \, = \,
   H_\ell \otimes 1 \, + \, 1 \otimes H_\ell  \cr
   \Delta^{\scriptscriptstyle \!(\Psi)}\!\big(F_\ell\big)  \, = \,
   F_\ell \otimes e^{-\hbar \, \big( \id_{\lieh'} + \, \psi_-^{\lieh'} \big)(T_\ell)} \, +
   \, e^{-\hbar \, \psi_+^{\lieh'}(T_\ell)} \otimes F_\ell  }  $$
                                                             \par
   Similarly, the ``twisted'' antipode  $ \, \SS^{\scriptscriptstyle (\Psi)} := \SS^{\F_\Psi} \, $
   is expressed by  (for all  $ \, \ell \in I \, $)
  $$  \displaylines{
   \SS^{\scriptscriptstyle (\Psi)}\big(E_\ell\big)  \, = \,
 - e^{-\hbar \big( \id_{\lieh'} + \, \psi_+^{\lieh'}\big)(T_\ell)} \, E_\ell \, e^{-\hbar \, \psi_-^{\lieh'}(T_\ell)}  \cr
   \SS^{\scriptscriptstyle (\Psi)}\big(H_\ell\big)  \, = \,  -H_\ell  \cr
   \SS^{\scriptscriptstyle (\Psi)}\big(F_\ell\big)  \, = \,
 - e^{+\hbar \, \psi_+^{\lieh'}(T_\ell)} \, F_\ell \, e^{+\hbar \big( \id_{\lieh'} + \, \psi_-^{\lieh'} \big)(T_\ell)}  }  $$
and the (untwisted!) counit by  $ \, \epsilon\big(E_\ell\big) = 0 \, $,  $ \, \epsilon\big(H_\ell\big) = 0 \, $,
$ \, \epsilon\big(F_\ell\big) = 0 \, $  again  (for  $ \, \ell \in I \, $).
                                                             \par
   In the sequel we shall use the shorter notation  $ \upsihg $  to denote  $ {\uhg}^{\F_\Psi} \, $.
                                                             \par
   The previous formulas show that the quantum Borel subalgebras  $ \uhbp $  and  $ \uhbm $
   are still Hopf subalgebras inside  $ \upsihg $  as well.  In fact, the element  $ \F_\Psi $
   can also be seen as a twisting element for both  $ \uhbp $  and  $ \uhbm \, $,  and then the corresponding twisted
   Hopf algebras obviously sit as Hopf subalgebras inside  $ \upsihg \, $.
 \vskip5pt
   On the other hand, the previous formulas suggest a different presentation of the quantum Borel subalgebras.
  \textsl{Hereafter, we assume for simplicity that the Cartan matrix  $ A $}  \textit{be of finite type\/};
  actually, this assumption can be lifted   --- indeed, we do it in  \cite{GG2}  ---
  but it makes our discussion definitely simpler.
                                                                 \par
   Namely, let us consider the ``twisted generators''
  $$  E^{\scriptscriptstyle \Psi}_i := \, e^{-\hbar \, \psi_-^{\lieh'}(T_i)} E_i  \quad ,
  \qquad  T^{\scriptscriptstyle \Psi}_{i,+} := \big( \text{id} + \psi_+^{\lieh'} \! - \psi_-^{\lieh'} \big)(T_i)  \quad
  \eqno \forall \;\; i \in I  \quad  $$
 in $ \upsihbp \, $:  \,since, by  Lemma \ref{lem:posivedef}\textit{(a)}
--- which exploits the finiteness assumption on  $ A $  ---
 the set  $ \, {\big\{ T^{\scriptscriptstyle \Psi}_{i,+} \big\}}_{i \in I} \, $  is a $ \kh $--basis  of  $ \lieh'_{\kh} \, $,
 we have that it generates  $ U^{\scriptscriptstyle \Psi}_\hbar(\lieh) = \uhh \, $;  \,in particular, the
 $ \kh $-subalgebra  generated (in topological sense) by these twisted generators coincide with  $ \upsihbp \, $.
 Analogously, consider the ``twisted generators''
  $$  T^{\scriptscriptstyle \Psi}_{i,-} := \big( \text{id} + \psi_-^{\lieh'} \! - \psi_+^{\lieh'} \big)(T_i)  \quad ,
  \qquad  F^{\scriptscriptstyle \Psi}_i := \, e^{+\hbar \, \psi_+^{\lieh'}(T_i)} F_i  \quad   \eqno \forall \;\; i \in I
  \quad  $$
for  $ \upsihbm \, $.  Then the above formulas for  $ \Delta^{\scriptscriptstyle (\Psi)} $  give (for all  $ \, \ell \in I \, $)
\begin{equation}  \label{eq: coprod_x_Upsihg}
 \begin{aligned}
   \Delta^{\scriptscriptstyle \!(\Psi)}\! \big(E^{\scriptscriptstyle \Psi}_\ell\big)
\, = \,  E^{\scriptscriptstyle \Psi}_\ell \otimes 1 \, +
\, e^{+\hbar T^{\scriptscriptstyle \Psi}_{\ell,+}} \otimes E^{\scriptscriptstyle \Psi}_\ell  \\
   \Delta^{\scriptscriptstyle \!(\Psi)}\big(T^{\scriptscriptstyle \Psi}_{\ell,\pm}\big)
\, = \,  T^{\scriptscriptstyle \Psi}_{\ell,\pm} \otimes 1 \, + \,
1 \otimes T^{\scriptscriptstyle \Psi}_{\ell,\pm}  \;\;\;  \\
   \Delta^{\scriptscriptstyle \!(\Psi)}\!\big(F^{\scriptscriptstyle \Psi}_\ell\big)  \, = \,
   F^{\scriptscriptstyle \Psi}_\ell \otimes e^{-\hbar T^{\scriptscriptstyle \Psi}_{\ell,-}} \, +
   \, 1 \otimes F^{\scriptscriptstyle \Psi}_\ell
 \end{aligned}
\end{equation}
and those for  $ \, \epsilon^{\scriptscriptstyle (\Psi)} :=
\epsilon \, $  and  $ \SS^{\scriptscriptstyle (\Psi)} $  yield
(for all  $ \, \ell \in I \, $)
\begin{equation}  \label{eq: antipode_x_Upsihg}
   \quad
 \begin{matrix}
   \epsilon^{\scriptscriptstyle (\Psi)}\big(E^{\scriptscriptstyle \Psi}_\ell\big) \, = \, 0  \;\;\; ,  &
   \;\quad  \epsilon^{\scriptscriptstyle (\Psi)}\big(T^{\scriptscriptstyle \Psi}_{\ell,\pm}\big) \, = \, 0  \;\;\; ,  &
   \;\quad  \epsilon^{\scriptscriptstyle (\Psi)}\big(F^{\scriptscriptstyle \Psi}_\ell\big) \, = \, 0  \;\;\;\;  \\
   \SS\big(E^{\scriptscriptstyle \Psi}_\ell\big)  \, = \,  -
   e^{-\hbar \, T^{\scriptscriptstyle \Psi}_{\ell,+}} E^{\scriptscriptstyle \Psi}_\ell  \;\; ,  &
   \;\;\;  \SS\big(T^{\scriptscriptstyle \Psi}_{\ell,\pm}\big)  \, = \,  -
   T^{\scriptscriptstyle \Psi}_{\ell,\pm}  \;\; ,  &
   \;\;\;  \SS\big(F^{\scriptscriptstyle \Psi}_\ell\big)  \, = \,  -
   F^{\scriptscriptstyle \Psi}_\ell \, e^{+\hbar \, T^{\scriptscriptstyle \Psi}_{\ell,-}}
 \end{matrix}
\end{equation}

\vskip9pt
   \indent   Indeed, using these generators we may write down a complete presentation by
   generators and relations of  $ \upsihg $  and similarly also of  $ \upsihbp $  and  $ \upsihbm \, $;  in fact,
   {\sl a complete set of relations is the following}  (for all  $ \, i, j \in I \, $)
\begin{equation}  \label{eq: multip-comm-rel's_x_Upsihg}
 \begin{aligned}
   {\big( \text{id} + \psi_+^{\lieh'} \! - \psi_-^{\lieh'} \big)}^{-1}(T^{\scriptscriptstyle \Psi}_{i,+})  \,\; =
   \;\,  {\big( \text{id} + \psi_-^{\lieh'} \! - \psi_+^{\lieh'} \big)}^{-1}(T^{\scriptscriptstyle \Psi}_{i,-})   \qquad \qquad \qquad  \\
   T^{\scriptscriptstyle \Psi}_{i,\pm} \, T^{\scriptscriptstyle \Psi}_{i,\pm}  \; =
   \; T^{\scriptscriptstyle \Psi}_{j,\pm} \, T^{\scriptscriptstyle \Psi}_{i,\pm}  \quad  ,
 \qquad \qquad \qquad  T^{\scriptscriptstyle \Psi}_{i,\pm} \, T^{\scriptscriptstyle \Psi}_{j,\mp}  \; =
 \;  T^{\scriptscriptstyle \Psi}_{j,\mp} \, T^{\scriptscriptstyle \Psi}_{i,\pm}   \quad \qquad \quad  \\
   T^{\scriptscriptstyle \Psi}_{i,+} \, E^{\scriptscriptstyle \Psi}_j \, - \,
   E^{\scriptscriptstyle \Psi}_j \, T^{\scriptscriptstyle \Psi}_{i,+}  \; =
   \;  +p_{i,j} \, E^{\scriptscriptstyle \Psi}_j  \quad  ,
 \qquad  T^{\scriptscriptstyle \Psi}_{i,-} \, E^{\scriptscriptstyle \Psi}_j \, - \,
 E^{\scriptscriptstyle \Psi}_j \, T^{\scriptscriptstyle \Psi}_{i,-}  \; = \;
 +p_{j,i} \, E^{\scriptscriptstyle \Psi}_j   \quad  \\
   T^{\scriptscriptstyle \Psi}_{i,+} \, F^{\scriptscriptstyle \Psi}_j \, - \,
   F^{\scriptscriptstyle \Psi}_j \, T^{\scriptscriptstyle \Psi}_{i,+}  \; = \;
   -p_{i,j} \, F^{\scriptscriptstyle \Psi}_j  \quad  ,
 \qquad  T^{\scriptscriptstyle \Psi}_{i,-} \, F^{\scriptscriptstyle \Psi}_j \, - \,
 F^{\scriptscriptstyle \Psi}_j \, T^{\scriptscriptstyle \Psi}_{i,-}  \; = \;
 -p_{j,i} \, F^{\scriptscriptstyle \Psi}_j   \quad  \\
   E^{\scriptscriptstyle \Psi}_i \, F^{\scriptscriptstyle \Psi}_j \, - \,
   F^{\scriptscriptstyle \Psi}_j \, E^{\scriptscriptstyle \Psi}_i  \; = \;
   \delta_{i,j} \, {{\; e^{+\hbar \, T^{\scriptscriptstyle \Psi}_{i,+}} \, - \,
   e^{-\hbar \, T^{\scriptscriptstyle \Psi}_{i,-}} \;} \over {\; e^{+\hbar\,p_{i{}i}/2} \,
   - \, e^{-\hbar\,p_{i{}i}/2} \;}}   \qquad \qquad \qquad \quad  \\
   \sum\limits_{k = 0}^{1-a_{ij}} (-1)^k {\left[ { 1-a_{ij} \atop k }
\right]}_{e^{\hbar\,p_{i{}i}/2}} {\big( E^{\scriptscriptstyle \Psi}_i \big)}^{1-a_{ij}-k}
E^{\scriptscriptstyle \Psi}_j {\big( E^{\scriptscriptstyle \Psi}_i \big)}^k  \; = \;
0   \quad \qquad   (\, i \neq j \,)  \\
   \sum\limits_{k = 0}^{1-a_{ij}} (-1)^k {\left[ { 1-a_{ij} \atop k }
\right]}_{e^{\hbar\,p_{i{}i}/2}} {\big( F^{\scriptscriptstyle \Psi}_i \big)}^{1-a_{ij}-k}
F^{\scriptscriptstyle \Psi}_j {\big( F^{\scriptscriptstyle \Psi}_i \big)}_i^k  \; = \;
0   \quad \qquad   (\, i \neq j \,)
 \end{aligned}
\end{equation}
with  $ \; p_{i,j} := \, \alpha_j\big(\big(\, \text{id} + \psi_+^{\lieh'} \! - \psi_-^{\lieh'} \big)(T_i)\big) =
{\big( D A + A^{\scriptscriptstyle T} \big( \Psi^{\scriptscriptstyle T} \! - \Psi \big) A \big)}_i^j \, \in \, \kh \; $
--- where the very first relation above makes sense because of  Lemma \ref{lem:posivedef}{\it (a)}.
                                                  \par
   Note,  by the way, that the matrix  $ \; P := {\big(\, p_{i,j} \big)}_{i, j \in I} = D A + A^{\scriptscriptstyle T}
   \big( \Psi^{\scriptscriptstyle T} \! - \Psi \big) A \; $  is also described by
   $ \; p_{i,j} = \alpha_j\big(\big(\, \text{id} + \psi_-^{\lieh'} \! - \psi_+^{\lieh'} \big)(T_i)\big) \; $.
   Moreover, we notice that the matrix  $ \, P - DA = A^{\scriptscriptstyle T} \big( \Psi^{\scriptscriptstyle T} \! -
   \Psi \big) A \, $  is antisymmetric.
                                                        \par
   As we pointed out, thanks to  Lemma \ref{lem:posivedef}{\it (a)}, both sets
   $ \, {\big\{ T^{\scriptscriptstyle \Psi}_{i,+} \big\}}_{i \in I} \, $  and
   $ \, {\big\{ T^{\scriptscriptstyle \Psi}_{i,-} \big\}}_{i \in I} \, $  are  $ \kh $--basis  of  $ \lieh'_{\kh} \, $,
   so each one of them is enough to generate  $ U^{\scriptscriptstyle \Psi}_\hbar(\lieh) = \uhh \, $.
   It follows that  $ \upsihg $  can be generated by either one of the smaller sets
   $ \, {\big\{ E^{\scriptscriptstyle \Psi}_i , T^{\scriptscriptstyle \Psi}_{i,+} \, , F^{\scriptscriptstyle \Psi}\big\}}_{i \in I} \, $  or
   $ \, {\big\{ E^{\scriptscriptstyle \Psi}_i , T^{\scriptscriptstyle \Psi}_{i,-} \, , F^{\scriptscriptstyle \Psi}\big\}}_{i \in I} \, $;
   in both case, however, the presentation would be less good-looking than using the bigger set
   $ \, {\big\{ E^{\scriptscriptstyle \Psi}_i , T^{\scriptscriptstyle \Psi}_{i,+} \, ,
   T^{\scriptscriptstyle \Psi}_{i,-} \, , F^{\scriptscriptstyle \Psi}\big\}}_{i \in I} \, $.
\end{free text}

\vskip9pt

\begin{free text}  \label{twist-Uhgd}
 {\bf Comultiplication twistings of  $ \uhgd \, $.}
 Let  $ \; \Psi := {\big( \psi_{ij} \big)}_{i,j \in I} \in M_n\big(\kh\big) \; $  as in  \S \ref{twist-Uhg}  above.
 Again, a direct check shows that the element
\begin{equation}  \label{Resh-twist_F-uhgd}
  \F_\Psi  \,\; := \;\,  \exp\bigg( \hbar \, {\textstyle \sum\limits_{i,j=1}^n} \psi_{ij} \, H^+_i \otimes H^-_j \bigg)
\end{equation}
is actually a  {\sl twist\/}  for  $ U_\hbar(\liegd) $  in the sense of  \S \ref{twist-defs}.
 Henceforth, we have analogous formulas for the new coproduct on generators similar to those in
 \S \ref{twist-Uhg},  namely
  $$  \displaylines{
   \Delta^{\scriptscriptstyle \!(\Psi)}\!\big(E_\ell\big)  \, = \,  E_\ell \otimes e^{+ \hbar \sum_{i, j \in I}
   \psi_{ij} \, a_{i\ell} H^-_j} \, + \, e^{+\hbar \, d_\ell H^+_\ell +
   \, \hbar \sum_{i, j \in I} \psi_{ij} \, a_{j\ell} H^+_i} \otimes E_\ell  \cr
   \Delta^{\scriptscriptstyle \!(\Psi)}\big(H^\pm_\ell\big)  \, = \,  H^\pm_\ell \otimes 1 \, + \, 1 \otimes H^\pm_\ell  \cr
   \Delta^{\scriptscriptstyle \!(\Psi)}\!\big(F_\ell\big)  \, = \,  F_\ell \otimes e^{-\hbar \, d_\ell H^-_\ell
   - \, \hbar \sum_{i, j \in I} \psi_{ij} \, a_{i\ell} H^-_j} \,
   + \, e^{-\hbar \sum_{i, j \in I} \psi_{ij} \, a_{j\ell} H^+_i} \otimes F_\ell  }  $$
 Using notation of  \S \ref{root-twisting}  with  $ \, \mathcal{R} := \kh \, $  and  $ \, {\{ T^\pm_i \}}_{i \in I} \, $
 thought of as a basis of  $ \lieh'_\pm \, $,  the formulas above read
  $$  \displaylines{
   \Delta^{\scriptscriptstyle \!(\Psi)}\!\big(E_\ell\big)  \, = \,  E_\ell \otimes
   e^{+\hbar \, \psi_-^{\lieh'}(T^-_\ell)} \, + \, e^{+\hbar \, \big( \id_{\lieh^{'}_+} \! +
   \, \psi_+^{\lieh'} \big)(T^+_\ell)} \otimes E_\ell  \cr
   \Delta^{\scriptscriptstyle \!(\Psi)}\big(H_\ell\big)  \, = \,  H_\ell \otimes 1 \, + \, 1 \otimes H_\ell  \cr
   \Delta^{\scriptscriptstyle \!(\Psi)}\!\big(F_\ell\big)  \, = \,  F_\ell \otimes
   e^{-\hbar \, \big( \id_{\lieh^{'}_{-}} \! + \, \psi_-^{\lieh'} \big)(T^-_\ell)} \, + \,
   e^{-\hbar \, \psi_+^{\lieh'}(T^+_\ell)} \otimes F_\ell  }  $$
                                                             \par
   Similarly, the ``twisted'' antipode  $ \, \SS^{\scriptscriptstyle (\Psi)} := \SS^{\F_\Psi} \, $  and the (untwisted!) counit
   $ \, \epsilon^{\scriptscriptstyle (\Psi)} := \epsilon \, $  are expressed by  (for all  $ \, \ell \in I \, $)
  $$  \begin{matrix}
   \SS^{\scriptscriptstyle (\Psi)}\big(E_\ell\big)  \, = \,  - e^{-\hbar \big( \id_{\lieh^{'}_{+}} + \, \psi_+^{\lieh'}\big)(T^+_\ell)} \,
   E_\ell \, e^{-\hbar \, \psi_-^{\lieh'}(T^-_\ell)}  \quad ,  &  \qquad  \epsilon\big(E_\ell\big) \, = \, 0  \\
   \SS^{\scriptscriptstyle (\Psi)}\big(H^\pm_\ell\big)  \, = \,  -H^\pm_\ell  \quad ,
   \phantom{\Big|^|}  &  \qquad  \epsilon\big(H^\pm_\ell\big) \, = \, 0  \\
   \SS^{\scriptscriptstyle (\Psi)}\big(F_\ell\big)  \, = \,  - e^{+\hbar \, \psi_+^{\lieh'}(T^+_\ell)}
   \, F_\ell \, e^{+\hbar \big( \id_{\lieh^{'}_{-}} + \,
   \psi_-^{\lieh'} \big)(T^-_\ell)}  \quad ,  &  \qquad  \epsilon\big(F_\ell\big) \, = \, 0
\end{matrix}  $$
 \vskip5pt
   Once more, acting again like in  \S \ref{twist-Uhg}
 --- and again assuming for simplicity that the Cartan matrix  $ A $  \textit{be of finite type}  ---
 we can consider the ``twisted'' generators (in topological sense, as usual)
  $$  E^{\scriptscriptstyle \Psi}_i := \, e^{-\hbar \, \psi_-^{\lieh'}(T^-_i)} E_i  \quad ,
  \qquad  T^{\scriptscriptstyle \Psi}_{i,+} := \big( \text{id} + \psi_+^{\lieh'} \big)(T^+_i) -
  \psi_-^{\lieh'}(T^-_i)  \quad   \eqno \forall \;\; i \in I  \quad  $$
for  $ \upsihbp $  and
  $$  T^{\scriptscriptstyle \Psi}_{i,-} := \big( \text{id} + \psi_-^{\lieh'} \big)(T^-_i) - \psi_+^{\lieh'}(T^+_i)  \quad ,
  \qquad  F^{\scriptscriptstyle \Psi}_i := \, e^{+\hbar \, \psi_+^{\lieh'}(T^+_i)} F_i  \quad   \eqno \forall \;\; i \in I  \quad  $$
for  $ \upsihbm \, $.  Then the above formulas for the Hopf operations along with the commutation relations in
$ \, \upsihgd = \uhgd \, $   --- an identity of algebras! ---   yield a presentation for
$ \upsihgd \, $, quite similar to that for  $ \upsihg \, $.
 \vskip5pt
   Finally, we remark that the epimorphism  $ \; \pi_\lieg : \uhgd \relbar\joinrel\relbar\joinrel\twoheadrightarrow \uhg \; $
   (see  \S \ref{Cartan-Borel_form-QUEAs}) of (topological) Hopf algebras is also an epimorphism for the twisted
   Hopf structure on both sides, i.e.,  it is an epimorphism
   $ \; \pi^{\scriptscriptstyle \Psi}_\lieg \! := \pi_\lieg : \upsihgd \relbar\joinrel\relbar\joinrel\twoheadrightarrow \upsihg \; $.
   Indeed, this is a direct, easy consequence of the fact that  $ \pi_\lieg^{\otimes 2} $  maps
   the twist element  \eqref{Resh-twist_F-uhgd}  of  $ \upsihgd $  onto the twist element  \eqref{Resh-twist_F-uhg}
   of  $ \upsihg \, $.  As a matter of description, using either presentation   --- the one with the  $ E_i $'s,  the
   $ T_{i,\pm} $'s  and the
 $ F_i $'s  or the one with the  $ E^{\scriptscriptstyle \Psi}_i $'s,  the  $ T^{\scriptscriptstyle \Psi}_{i,\pm} $'s
 and the  $ F^{\scriptscriptstyle \Psi}_i $'s  ---   for both algebras, it is clear that  $ \pi^{\scriptscriptstyle \Psi}_\lieg $
 maps each generator of  $ \upsihgd $  onto the same name generator of  $ \upsihg \, $.
\end{free text}

\vskip9pt

\begin{free text}  \label{twist-Borel-in-Uhgd}
 {\bf Twisting Borel subalgebras in  $ \uhgd$.}
 The formulas for the twisted coproduct in  \S \ref{twist-Uhgd}  above show that
 {\sl the quantum Borel subalgebras  $ \uhbp $  and  $ \uhbm $  are no longer (in general)
 Hopf subalgebras inside  $ \upsihgd \, $}.  Instead, let  $ \udotpsihbp \, $,  resp.\  $ \udotpsihbm \, $,
 be the complete, unital subalgebra of  $ \upsihgd $  generated by
 $ \, {\big\{ E_i \, , \, \big(\id_{\lieh^{'}_{+}} \! + \psi_+^{\lieh'}\big)\big(T^+_i\big) \, , \,
 \psi_-^{\lieh'}\big(T^-_i\big) \,\big\}}_{i \in I} \, $,  resp.\ by
 $ \, {\big\{ \psi_+^{\lieh'}\big(T^+_i\big) \, , \, \big( \id_{\lieh^{'}_{-}} \! + \psi_-^{\lieh'}\big)\big(T^-_i\big) \, ,
 F_i \,\big\}}_{i \in I} $.
 \break
 Then the same formulas yield
 \vskip5pt
   \centerline{\it  $ \udotpsihbm $  and  $ \, \udotpsihbm $  are Hopf subalgebras of  $ \, \upsihgd \, $.}
 \vskip5pt
   Clearly, the semiclassical limits of these twisted subalgebras are
  $$  \udotpsihbp \Big/ \hbar \, \udotpsihbp \; \cong \; U\big( {\dot{\lieb}}^{\scriptscriptstyle \Psi}_+ \big)
  \qquad\;\;\;  \text{and}  \;\;\;\qquad  \udotpsihbm \Big/ \hbar \, \udotpsihbm \; \cong \;
  U\big( {\dot{\lieb}}^{\scriptscriptstyle \Psi}_- \big)  $$
where (with notation of  \S \ref{root-data_Lie-algs})
 \vskip3pt
   \centerline{\sl  $ \, {\dot{\lieb}}^{\scriptscriptstyle \Psi}_+ \, := \, $
   Lie subalgebra of  $ \, \liegd'(A) $  generated by
   $ \; {\big\{ \erm_i \, , \, \big( \id_{\lieh^{'}_{-}} \! + \psi_+^{\lieh'} \big)\big( \trm^+_i \big) \, ,
   \, \psi_-^{\lieh'}\big( \trm^-_i \big) \big\}}_{i \in I} \, $ }
 \vskip3pt
   \centerline{\sl  $ \, {\dot{\lieb}}^{\scriptscriptstyle \Psi}_- \, := \, $  Lie subalgebra of
   $ \, \liegd'(A) $  generated by $ \; {\big\{ \psi_+^{\lieh'}\big( \trm^+_i \big) \, ,
   \, \big( \id_{\lieh^{'}_{-}} \! + \psi_-^{\lieh'} \big)\big( \trm^-_i \big) \, , \, \frm_i \big\}}_{i \in I} \, $ }
 \vskip3pt
\noindent
 --- where  $ \; \psi_\pm^{\lieh'}\big( \trm^\pm_i \big) := \psi_\pm^{\lieh'}\big( T^\pm_i \big) \,\;
 \text{mod} \; \hbar \, \kh \, \in {\lieh'}^\pm \subseteq \lieb_\pm'(A) \; $  ---
 so that a more inspiring, self-explaining notation might be
 $ \; U_\hbar\big( {\dot{\lieb}}^{\scriptscriptstyle \Psi}_\pm \big) := \udotpsihbpm \; $.
                                                                \par
   Note, however, that  $ {\dot{\lieb}}^{\scriptscriptstyle \Psi}_+ $  and  $ {\dot{\lieb}}^{\scriptscriptstyle \Psi}_- $
   both have a larger ``Cartan subalgebra'' than  $ \lieh' \, $,  so they cannot be correctly thought of as ``twisted''
   Borel subalgebras inside  $ \liegd'(A) \, $.
 \vskip5pt
   On the other hand, let us consider the complete, unital subalgebras  $ \upsihbp $  and  $ \upsihbm $  of  $ \upsihgd $
   generated respectively by the ``twisted generators''
  $$  E^{\scriptscriptstyle \Psi}_i \, := \, e^{-\hbar \, \psi_-^{\lieh'}(T^-_i)} E_i  \quad ,
  \qquad  T^{\scriptscriptstyle \Psi}_{i,+} := \big(\, \text{id}_{\lieh'_+} \! + \psi_+^{\lieh'} \,\big)\big(T^+_i\big) -
  \psi_-^{\lieh'}\big(T^-_i\big)   \eqno \forall \;\; i \in I  \quad  $$
and
  $$  F^{\scriptscriptstyle \Psi}_i \, := \, e^{+\hbar \, \psi_+^{\lieh'}(T^+_i)} F_i  \quad ,
  \qquad  T^{\scriptscriptstyle \Psi}_{i,-} \, := \,  \big(\, \text{id}_{\lieh'_-} \! + \psi_-^{\lieh'} \,\big)\big(T^-_i\big) -
  \psi_+^{\lieh'}\big(T^+_i\big)   \eqno \forall \;\; i \in I  \quad  $$
given in  \S \ref{twist-Uhgd}
 --- still assuming for simplicity that the matrix  $ A $  be of finite type.
 Then the above formulas for  $ \Delta^{\scriptscriptstyle (\Psi)} $,
$ \SS^{\scriptscriptstyle (\Psi)} $  and  $ \, \epsilon^{\scriptscriptstyle (\Psi)} := \epsilon \, $
altogether yield (for all  $ \, \ell \in I \, $)
  $$  \begin{matrix}
   \Delta^{\scriptscriptstyle \!(\Psi)}\! \big(E^{\scriptscriptstyle \Psi}_\ell\big)
\, = \,  E^{\scriptscriptstyle \Psi}_\ell \otimes 1 \, + \, e^{+\hbar \, T^{\scriptscriptstyle \Psi}_{\ell,+}}
\otimes E^{\scriptscriptstyle \Psi}_\ell  \quad ,  &  \qquad
   \Delta^{\scriptscriptstyle \!(\Psi)}\big(T^{\scriptscriptstyle \Psi}_{\ell,+}\big)  \, = \,
   T^{\scriptscriptstyle \Psi}_{\ell,+} \otimes 1 \, + \, 1 \otimes T^{\scriptscriptstyle \Psi}_{\ell,+}  \\
   \SS^{\scriptscriptstyle (\Psi)}\big(E^{\scriptscriptstyle \Psi}_\ell\big)  \, = \,
 - e^{-\hbar \, T^{\scriptscriptstyle \Psi}_{\ell,+}} E^{\scriptscriptstyle \Psi}_\ell  \quad ,  &  \qquad
 \SS^{\scriptscriptstyle (\Psi)}\big(T^{\scriptscriptstyle \Psi}_{\ell,+}\big)  \, = \,
 - T^{\scriptscriptstyle \Psi}_{\ell,+}  \\
   \epsilon\big(E^{\scriptscriptstyle \Psi}_\ell\big) \, = \, 0  \quad ,  \phantom{\big|^|_|}  &
   \qquad  \epsilon\big(T^{\scriptscriptstyle \Psi}_{\ell,\pm}\big) \, = \, 0  \\
\end{matrix}  $$
on the generators of  $ \upsihbp $  and
  $$  \begin{matrix}
   \Delta^{\scriptscriptstyle \!(\Psi)}\! \big(F^{\scriptscriptstyle \Psi}_\ell\big)
\, = \,  F^{\scriptscriptstyle \Psi}_\ell \otimes e^{-\hbar \, T^{\scriptscriptstyle \Psi}_{\ell,-}} \, +
\, 1 \otimes F^{\scriptscriptstyle \Psi}_\ell  \quad ,  &  \qquad
   \Delta^{\scriptscriptstyle \!(\Psi)}\big(T^{\scriptscriptstyle \Psi}_{\ell,-}\big)  \, =
   \,  T^{\scriptscriptstyle \Psi}_{\ell,-} \otimes 1 \, + \, 1 \otimes T^{\scriptscriptstyle \Psi}_{\ell,-}  \\
   \SS^{\scriptscriptstyle (\Psi)}\big(F^{\scriptscriptstyle \Psi}_\ell\big)  \, = \,
 - F^{\scriptscriptstyle \Psi}_\ell \, e^{+\hbar \, T^{\scriptscriptstyle \Psi}_{\ell,-}}  \quad ,  &
 \qquad  \SS^{\scriptscriptstyle (\Psi)}\big(T^{\scriptscriptstyle \Psi}_{\ell,-}\big)  \, = \,
 - T^{\scriptscriptstyle \Psi}_{\ell,-}  \\
   \epsilon\big(F^{\scriptscriptstyle \Psi}_\ell\big) \, = \, 0  \quad ,  \phantom{\big|^|_|}  &
   \qquad  \epsilon\big(T^{\scriptscriptstyle \Psi}_{\ell,-}\big) \, = \, 0  \\
\end{matrix}  $$
(for all  $ \, \ell \in I \, $)  on those of  $ \upsihbm \, $.  Altogether, these formulas show that
 {\it both  $ \upsihbp $  and  $ \upsihbm $  are Hopf subalgebras of  $ \upsihgd \, $}.
 \vskip5pt
   Now, at the semiclassical level, let us consider the elements
  $$  \trm^{\scriptscriptstyle \Psi}_{i,+} := \big( \text{id}_{\lieh^{'}_{+}} \! + \,
  \psi_+^{\lieh'} \big)\big( \trm^+_i \big) - \, \psi_-^{\lieh'}\big( \trm^-_i \big)  \;\; ,  \quad
  \trm^{\scriptscriptstyle \Psi}_{i,-} := \big( \text{id}_{\lieh^{'}_{-}} \! + \,
  \psi_-^{\lieh'}\big)\big( \trm^-_i \big) - \, \psi_+^{\lieh'}\big( \trm^+_i \big)   \eqno \quad \forall \;\; i \in I \;\;  $$
and the Lie subalgebras
 \vskip3pt
   \centerline{\sl  $ \, \lieb^{\scriptscriptstyle \Psi}_+ \, := \, $  Lie subalgebra of  $ \, \liegd'(A) $
   generated by  $ \, {\big\{\, \erm_i \, , \, \trm^{\scriptscriptstyle \Psi}_{i,+} \big\}}_{i \in I} \, $ }
 \vskip3pt
   \centerline{\sl  $ \, \lieb^{\scriptscriptstyle \Psi}_- \, := \, $  Lie subalgebra of  $ \, \liegd'(A) $
   generated by  $ \; {\big\{\, \trm^{\scriptscriptstyle \Psi}_{i,-} \, , \, \frm_i \big\}}_{i \in I} \, $ }
 \vskip3pt
\noindent
 Notice that the toral part of  $ \, \lieb^{\scriptscriptstyle \Psi}_\pm \, $,  that is the Lie
 subalgebra generated by the new elements  $ \trm^{\scriptscriptstyle \Psi}_{i,\pm} \, $,
 is isomorphic to  $ \lieh' \, $,  thanks to  Lemma \ref{lem:posivedef}{\it (b)}.
                                                                  \par
 Then one easily sees that the semiclassical limits of  $ \upsihbp $  and  $ \upsihbm $  are
  $$  \upsihbp \Big/ \hbar \, \upsihbp \; \cong \; U\big( \lieb^{\scriptscriptstyle \Psi}_+ \big)  \qquad\;\;\;  \text{and}
  \;\;\;\qquad  \upsihbm \Big/ \hbar \, \upsihbm \; \cong \;  U\big( \lieb^{\scriptscriptstyle \Psi}_- \big)  $$
so that the self-explaining notation  $ \; U_\hbar\big( \lieb^{\scriptscriptstyle \Psi}_+ \big) := \upsihbp \; $
and  $ \; U_\hbar\big( \lieb^{\scriptscriptstyle \Psi}_- \big) := \upsihbm \; $  can be also adopted.
All this, in turn, also reflects the fact that  {\sl both\/  $ \lieb^{\scriptscriptstyle \Psi}_+ $  and\/
$ \lieb^{\scriptscriptstyle \Psi}_- $
are Lie sub-bialgebras of the Lie bialgebra\/  $ \liegd'(A) $}  (and have the correct size for ``Borel''!).
                                                            \par
   Note that, like in  \S \ref{twist-Uhg}  above, from the previous observations we can also find
   out a complete presentation by generators (the ``twisted'' ones) and relations for  $ \upsihbp $
   and  $ \upsihbm \, $,  easily deduced from that for  $ \upsihgd \, $:  in comparison with those of the
   corresponding untwisted quantum algebras, in these presentations
   {\sl the formulas for the coproduct will read the same},  while
   {\sl the commutation relations will be deformed\/}  instead --- just the converse of what occurs
   in the original presentations.
\end{free text}

\smallskip

\begin{rmk}  \label{rmk:twistedpairing}
 Using the skew-Hopf pairing
 \eqref{h-Borel_pairing}  between  $ \uhbp $  and  $ \uhbm $  and the formulas in
 \S \ref{root-twisting},  one may define a ``twisted'' skew-Hopf pairing  $ \eta^{\scriptscriptstyle \Psi} $
 between  $ \upsihbp $  and  $ \upsihbm \, $,  in such a way that the corresponding Drinfeld double
 $ D \big( \upsihbp, \upsihbm, \eta^{\Psi} \big) $  is canonically isomorphic to  $ \upsihgd \, $.
\end{rmk}

\vskip5pt

\begin{rmk}
 It is worth noting that, strictly speaking,  $ \, U_\hbar\big( {\dot{\lieb}}^{\scriptscriptstyle \Psi}_\pm \big)\, $
 \textsl{cannot be thought of as a
comultiplication twisting
 of the isomorphic copy of  $ \, U_\hbar(\lieb_\pm) $  inside}  $ \, \uhgd \, $;
 indeed, this occurs because the twisting element
  $$  \F_\Psi  \; := \;\,  \exp\bigg( \hbar \, {\textstyle \sum\limits_{i,j=1}^n} \psi_{ij} \, H^+_i \otimes H^-_j \bigg)  $$
that we used to deform  $ \uhgd $   --- cf.\  \S \ref{twist-Uhgd}  ---   does not belong to  $ \, U_\hbar(\lieb_\pm)  \, $.
                                                             \par
   On the other hand, when the (generalized) Cartan matrix  $ A $  {\sl is of  {\it finite}  type\/}
   (which is equivalent to saying that  $ A^{-1} $  exists!), let us write  $ \, {\big( a'_{hk} \big)}_{h,k \in I} = A^{-1} \, $
   and consider the elements  $ \, T^\pm_{\omega_i} := \sum_{j=1}^n a'_{ji} \, T^\pm_i \; (\, i \in I \,) \, $.
   Then by a straightforward calculation we can re-write  $ \F_\Psi $  as
\begin{equation}  \label{F-twist_Uhbp}
  \F_\Psi  \; := \;\,  \exp\bigg( \hbar \, {\textstyle \sum\limits_{k=1}^n} \,
  d_k^{-1} \, T^+_{\omega_k} \otimes \psi_-^\lieh\big(T^-_k\big) \bigg)
\end{equation}
and also as
\begin{equation}  \label{F-twist_Uhbm}
  \F_\Psi  \; := \;\,  \exp\bigg( \hbar \, {\textstyle \sum\limits_{k=1}^n} \,
  \psi_+^\lieh\big(T^+_k\big) \otimes d_k^{-1} \, T^-_{\omega_k} \bigg)
\end{equation}
   \indent   Lead by these formulas, we define  $ \, {\check{U}}_\hbar(\lieb_+) \, $
   to be the complete, unital subalgebra of  $ \upsihgd $  generated by
  $ \; {\Big\{ E_i \, , \, d_i^{-1} \, T^+_{\omega_i} \, , \, \big( \id_\lieh \! + \psi_+^\lieh\big)\big(T^+_i\big) \, ,
  \, \psi_-^\lieh\big(T^-_i\big) \Big\}}_{i \in I} \, $,
 \, and  $ \, {\check{U}}_\hbar(\lieb_-) \, $  the one generated by
  $ \; {\Big\{\, d_i^{-1} \, T^-_{\omega_i} \, , \, \psi_+^\lieh\big(T^+_i\big) \, , \,
  \big( \id_\lieh \! + \psi_-^\lieh\big)\big(T^-_i\big) \, , \, F_i \,\Big\}}_{i \in I} \, $.
 Then  $ {\check{U}}_\hbar(\lieb_\pm) $  is a Hopf subalgebra of both  $ \uhgd $  and  $ \upsihgd \, $,
 hence we shall write
 $ {\check{U}}^{\scriptscriptstyle \Psi}_\hbar(\lieb_-) $  when we consider  $ {\check{U}}_\hbar(\lieb_-) $
 endowed with the Hopf
 (sub)algebra structure induced from the (twisted) Hopf algebra structure of  $ \uhgd \, $.
 Now, by  \eqref{F-twist_Uhbp}  and  \eqref{F-twist_Uhbm}  we have
  $$  \F_\Psi \, \in \, {\check{U}}^{\scriptscriptstyle \Psi}_\hbar(\lieb_+)
  \,\widehat{\otimes}\, {\check{U}}^{\scriptscriptstyle \Psi}_\hbar(\lieb_+)  \quad\qquad
  \text{and}
  \qquad\quad  \F_\Psi \, \in \, {{\check{U}}^{\scriptscriptstyle \Psi}_\hbar(\lieb_-)} \,\widehat{\otimes}\,
  {\check{U}}^{\scriptscriptstyle \Psi}_\hbar(\lieb_-)  $$
so that  $ \F_\Psi $  can indeed be rightfully called ``twisting element'' for  $ {\check{U}}_\hbar(\lieb_\pm) \, $
--- as it does belong to  $ \, {\check{U}}_\hbar(\lieb_\pm) \,\widehat{\otimes}\, {\check{U}}_\hbar(\lieb_\pm) \, $
and twisting  $ {\check{U}}_\hbar(\lieb_\pm) $  by it we actually get
$ {\check{U}}^{\scriptscriptstyle \Psi}_\hbar(\lieb_\pm) \, $.
\end{rmk}

\medskip

\subsection{Polynomial QUEA's (``\`a la Jimbo-Lusztig'')}  \label{memo_polyn-QUEAs}
 {\ }

\smallskip

  We shortly recall hereafter the ``polynomial version'' of the notion of QUEA, as introduced by Jimbo,
  Lusztig and others, as well as some related material.

\smallskip

   Let  $ \k_q $  be the subfield of  $ \kh $  generated by
   $ \; \k \,\cup\, \big\{\, q^{\pm 1/m} := e^{\pm \hbar/m} \,\big|\, m \! \in \! \NN_+ \big\} \; $;
   \,in particular  $ \, q^{\pm 1} := e^{\pm \hbar} \in \k_q \, $  and  $ \, q^{\pm 1}_i := q^{\pm d_i} \in \k_q \, $
   for all  $ \, i \in I \, $.  Note that  $ \k_q $  is the injective limit of all the fields of rational functions
   $ \, \k\big(q^{1/m}\big) \, $,  but in specific cases we shall be working with a specific bound on  $ m \, $,
   fixed from scratch, so in fact we can adopt as ground ring just one such ring  $ \, \k\big(q^{1/N}\big) \, $
   for a single, large enough  $ N $.
                                                                    \par
   As a general matter of notation, hereafter by  $ q^r $  for any  $ \, r \in \QQ \, $  we shall always mean
   $ \, q^r = q^{a/d} := {\big( q^{1/d} \big)}^a \, $  if  $ \, r = a/d \, $  with  $ \, a \in \ZZ \, $  and  $ \, d \in \NN_+ \, $.

\smallskip

\begin{rmk}
  As a matter of fact, one could still work over subring(s) of  $ \kh $  generated by elements of the set
  $$  \big\{\, q^c := \exp\big( \hbar \, c \,\big) \,\big|\, c \in \kh \,\big\}  \,\; = \;\,
  \big\{\, 1 + \hbar \, c \,\big|\, c \in \kh \,\big\}  $$
 This implies that one considers sublattices in  $ \lieh' $  which might generate different
 $ \QQ $--vector spaces than  $ \lieh'_\QQ \, $,  but this makes no difference when it comes to
 {\sl polynomial\/}  QUEA's  and their multiparameter versions.  As a byproduct, later on we
 do not need to restrict ourselves to  {\sl ``rational twists''},  i.e., (toral) twists coming from
 $ \, \Psi \in M_n(\QQ) \, $.
\end{rmk}

\vskip9pt

\begin{free text}  \label{intro-Uqg}
 {\bf The polynomial QUEA  $ \uqg \, $.}  We introduce the ``polynomial'' QUEA for
  $ \lieg'(A) \, $,
 hereafter denoted
 $  U_{q}(\lieg'(A))= \uqg \, $,  as being the unital  $ \k_q $--subalgebra  of  $ \uhg $  generated by
 $ \, \big\{ E_i \, , \, K_i^{\pm 1} \! := \! e^{\pm \hbar \, d_i H_i} , \, F_i \;\big|\; i \! \in \! I \,\big\} \, $.
 From this definition and from the presentation of  $ \uhg $  we deduce that  $ \uqg $ can be
 presented as the associative, unital
 $ \k_q $--algebra  with generators  $ E_i \, $,  $ K_i^{\pm 1} $  and  $ F_i $  ($ \, i \in I := \{1,\dots,n\} \, $)
 and relations (for all  $ i, j \in I $)
\begin{eqnarray}
   K_i K_j \; = \; K_j K_i  \quad ,  \qquad  K_i^{\pm 1} K_i^{\mp 1} \; = \; 1 \; = \; K_i^{\mp 1} K_i^{\pm 1}
   \qquad\qquad \nonumber\\
   K_i E_j K_{i}^{-1} \, = \;  q_i^{+a_{ij}} E_j   \quad ,  \qquad  K_i F_j K_{i}^{-1} \, = \;  q_i^{-a_{ij}} F_j
   \qquad\qquad \nonumber\\
   E_i F_j - F_j E_i  \; = \;  \delta_{ij} \, {{\, K_i - K_i^{-1} \,}
\over {\, q_i - q_i^{-1} \,}} \qquad\qquad \qquad \nonumber  \\
    \qquad \qquad  \sum\limits_{k = 0}^{1-a_{ij}} (-1)^k {\left[ { 1-a_{ij} \atop k }
\right]}_{q_i} E_i^{1-a_{ij}-k} E_j E_i^k  \; = \;  0   \qquad  \qquad  (\, i \neq j \,)
\label{eq:QUEA-Jimbo-aSerre1}\\
    \qquad \qquad  \sum\limits_{k = 0}^{1-a_{ij}} (-1)^k {\left[ { 1-a_{ij} \atop k }
\right]}_{q_i} F_i^{1-a_{ij}-k} F_j F_i^k  \; = \;  0   \qquad  \qquad  (\, i \neq j \,)
\label{eq:QUEA-Jimbo-aSerre2}
\end{eqnarray}
The formulas for the coproduct, antipode and counit in  $ \uhg $  now give
  $$  \displaylines{
   \Delta(E_i) := E_i \otimes 1 + K_i \otimes E_i  \;\; ,  \qquad
   \SS(E_i) := - K_i^{-1} E_i  \;\; ,   \qquad  \epsilon(E_i) := 0  \cr
   \Delta\big( K_i^{\pm 1} \big) := K_i^{\pm 1} \otimes K_i^{\pm 1}  \;\quad ,
   \;\;\qquad  \SS\big( K_i^{\pm 1} \big) := K_i^{\mp 1}  \;\quad ,  \;\;\qquad
   \epsilon\big( K_i^{\pm 1} \big) := 1  \cr
   \Delta(F_i) := F_i \otimes K_i^{-1} + 1 \otimes F_i  \;\; ,  \;\qquad
   \SS(F_i) := - F_i \, K_i \;\; ,  \;\qquad  \epsilon(F_i) := 0  }  $$
--- for all  $ \, i \in I \, $  ---   so that  {\it  $ \uqg $  is actually a
{\sl (standard, i.e., non-topological)}
 Hopf subalgebra inside  $ \uhg \, $};  cf.\  \cite{CP},
 and references therein for details, taking into account that we adopt slightly different normalizations.
\end{free text}

\vskip7pt

\begin{free text}  \label{Cartan-Borel_pol-QUEAs}
 {\bf Polynomial quantum Borel (sub)algebras and their double.}
 We consider now the ``polynomial version'' of the quantum Cartan subalgebra  $ \uhh \, $,
 the quantum Borel subalgebras  $ U_\hbar(\hskip0,8pt\lieb_\pm) $  and the ``quantum double''  $ \uhgd \, $.
                                                      \par
   First, we define  $ \uqh $  as being the unital  $ \k_q $--subalgebra  of  $ \uhg $  generated by
   $ \, \big\{\, K_i^{\pm 1} \;\big|\; i \in I \,\big\} \, $;  this is clearly a Hopf subalgebra of both  $ \uhh $  and  $ \uqg $
   --- in fact, it coincides with  $ \, \uhh \cap \uqg \, $  ---   isomorphic to the group algebra over  $ \k_q \, $,
   with its canonical Hopf structure, of the Abelian group  $ \ZZ^n \, $.
                                                      \par
   Similarly, we define  $ \uqbp \, $,  resp.\  $ \uqbm \, $,  as being the unital  $ \k_q $--subalgebra  of
   $ \uhg $  generated by  $ \, \big\{\, E_i \, , K_i^{\pm 1} \;\big|\; i \in I \,\big\} \, $,  resp.\  by
   $ \, \big\{\, K_i^{\pm 1} , \, F_i \;\big|\; i \in I \,\big\} \, $.  Then  $ U_q(\hskip0,8pt\lieb_\pm) $
   is a Hopf subalgebra of both  $ U_\hbar(\hskip0,8pt\lieb_\pm) $  and  $ \uqg \, $,  coinciding with
   $ \, U_\hbar(\hskip0,8pt\lieb_\pm) \cap \uqg \, $.
                                                      \par
   From the presentation of  $ \uqg $ one can easily deduce a similar presentation for  $ \uqh \, $,
   $ \uqbp $  and  $ \uqbm $  as well.
 \vskip7pt
   The skew-Hopf pairing  $ \; \eta : \uhbp \otimes_{\kh} \uhbm \relbar\joinrel\longrightarrow \kh \; $
 in  \S \ref{Cartan-Borel_form-QUEAs}
 restricts to a similar skew-pairing  $ \; \eta : \uqbp \otimes_{\k_q} \uqbm\relbar\joinrel\longrightarrow \k_q \; $
 between  {\sl polynomial\/}  quantum Borel subalgebras, described by the formulas
\begin{equation}  \label{q-Borel_pairing}  \hskip5pt
  \eta\big(K_{i\,},\!K_j\big) = q^{-d_i a_{ij}}  \, ,  \,\;
  \eta\big(E_{i\,},\!F_j\big) = {{\,\delta_{ij}\,} \over {\,q_i^{-1\,} \!\! - q_i\,}}  \; ,
  \,\;  \eta\big(E_{i\,},\!K_j\big) = 0 = \eta\big(K_{i\,},\!F_j\big)
\end{equation}
 Using this pairing one constructs the corresponding Drinfeld double as in  \S \ref{conv-Hopf},  in the
 sequel denoted also by  $ \, \uqgd := D\big( {\uqbp} , \uqbm , \eta \big) \, $.
                                                            \par
   Tracking the whole construction, one realizes that  $ \uqgd $  naturally identifies with the unital
   $ \k_q $--subalgebra  of  $ \uhgd $  generated by
  $$ \, \big\{\, E_i \, , \, K_{i,\smallp} \! := e^{\pm d_i H^+_i} \! = e^{\pm T^+_i} \! , \, K_{i,-} :=
  e^{\pm d_i H^-_i} \! = e^{\pm T^-_i} , \, F_i \;\big|\; i \in I \,\big\}  $$
 Then $ \uqgd $  can be presented as the associative, unital  $ \k_q $--algebra  with generators
 $ E_i \, $,  $ K_{i,\smallpm} $  and  $ F_i $   --- with  $ \, i \in I \, $  ---   satisfying \eqref{eq:QUEA-Jimbo-aSerre1}  and
 \eqref{eq:QUEA-Jimbo-aSerre2},  together with the relations (for all  $ i, j \in I $)
  $$  \displaylines{
   K_{i,\smallp} \, K_{j,\smallp} \; = \; K_{j,\smallp,j} \, K_{i,\smallp}  \!\quad ,  \!\!\qquad K_{i,\smallp} \, K_{j,-} \; =
   \; K_{j,-} \, K_{i,\smallp}  \!\quad ,  \!\!\qquad   K_{i,-} \, K_{j,-} \; = \; K_{j,-} \, K_{i,-}  \cr
   K_{i,\smallp}^{\pm 1} \, K_{i,\smallp}^{\mp 1}  \, = \;  1  \; = \;  K_{i,\smallp}^{\mp 1} \, K_{i,\smallp}^{\pm 1}
   \quad ,  \qquad   K_{i,-}^{\pm 1} \, K_{i,-}^{\mp 1}  \, = \;  1  \; = \; K_{i,-}^{\mp 1} \, K_{i,-}^{\pm 1}  \cr
   K_{i,\smallpm}\, E_j \, K_{i,\smallmp}^{-1} \, = \;  q^{+d_i a_{ij}} E_j   \qquad ,  \quad \qquad
   K_{i,\smallpm} \, F_j \, K_{i,\smallpm}^{-1} \, = \;  q^{-d_i a_{ij}} F_j  \cr
   E_i \, F_j - F_j \, E_i  \,\; = \;\,  \delta_{ij} \, {{\, K_{i,\smallp} - K_{i,-}^{-1} \,} \over {\, q_i - q_i^{-1} \,}}  }  $$
 while coproduct, antipode and counit are described by the formulas
  $$  \displaylines{
   \Delta(E_i) := E_i \otimes 1 + K_{i,\smallp}\otimes E_i  \;\; ,  \qquad
   \SS(E_i) := - K_{i,\smallp}^{-1} E_i  \;\; ,  \qquad  \epsilon(E_i) := 0  \cr
   \Delta\big( K_{i,\smallp}^{\pm 1} \big) := K_{i,\smallp}^{\pm 1} \otimes K_{i,\smallp}^{\pm 1}  \;\quad ,
   \;\;\qquad  \SS\big( K_{i,\smallp}^{\pm 1} \big) := K_{i,\smallp}^{\mp 1}  \;\quad ,
   \;\;\qquad  \epsilon\big( K_{i,\smallp}^{\pm 1} \big) := 1  \cr
   \Delta\big( K_{i,-}^{\pm 1} \big) := K_{i,-}^{\pm 1} \otimes K_{i,-}^{\pm 1}  \;\quad ,
   \;\;\qquad  \SS\big( K_{i,-}^{\pm 1} \big) := K_{i,-}^{\mp 1}  \;\quad ,
   \;\;\qquad  \epsilon\big( K_{i,-}^{\pm 1} \big) := 1  \cr
   \Delta(F_i) := F_i \otimes K_{i,-}^{-1} + 1 \otimes F_i  \;\; ,  \;\qquad
   \SS(F_i) := - F_i \, K_{i,-}  \;\; ,  \;\qquad  \epsilon(F_i) := 0  }  $$
for all  $ \, i \in I \, $   --- so that  {\it  $ \uqgd $  is actually a  {\sl Hopf subalgebra\/}  inside  $ \uhgd \, $}.
                                                     \par
  In terms of this description, and using the canonical identification
  $$  \uqgd \, := \, D\big( {\uqbp} , \uqbm , \eta \big)  \; \cong \;  {\uqbp} \otimes_{\k_q} \uqbm  $$
 as coalgebras, we have  $ \, E_i = E_i \otimes 1 \, $,
 $ \, K_{i,\smallp}^{\pm 1} = K_i^{\pm 1} \otimes 1 \, $,  $ \, K_{i,-}^{\pm 1} = 1 \otimes K_i^{\pm 1} \, $,
 $ \, F_i = 1 \otimes F_i \, $  (for all  $ \, i \in I \, $).  Moreover, the natural embeddings of  $ \uqbp $  and
 $ \uqbm $  as Hopf subalgebras inside  $ \, \uqgd \, \cong \, {\uqbp} \otimes_{\k_q} \uqbm \, $  are described by
  $$  E_i \mapsto E_i \; ,  \quad  K_i^{\pm 1} \mapsto K_{i,\smallp}^{\pm 1}
  \qquad \text{and} \qquad  F_i \mapsto F_i \; ,  \quad  K_i^{\pm 1} \mapsto K_{i,-}^{\pm 1}
  \eqno \qquad \forall \;\; i \in I  \quad  $$
 \vskip3pt
   Also by construction, the projection  $ \; \pi_\lieg : \uhgd \relbar\joinrel\relbar\joinrel\twoheadrightarrow \uhg \; $
   --- an epimorphism of formal Hopf algebras over  $ \kh $  ---   restricts to a Hopf  $ \k_q $--algebra
   epimorphism  $ \; \pi_\lieg : \uqgd \relbar\joinrel\relbar\joinrel\twoheadrightarrow \uqg \; $,  \, which is
   explicitly described by
\begin{equation}  \label{proj_Uqgd-->Uqg}
  \pi_\lieg \, :  \;\quad   E_i \mapsto E_i \; ,  \quad  K_{i,\smallp}^{\pm 1} \mapsto K_i^{\pm 1} \; ,
  \quad  K_{i,-}^{\pm 1} \mapsto K_i^{\pm 1} \; ,  \quad  F_i \mapsto F_i  \;\;\qquad  \forall \;\; i \in I
\end{equation}
\end{free text}

\vskip11pt

\begin{free text}  \label{larger-tori_QUEA's}
 {\bf Larger tori for (polynomial) QUEA's.}
 By definition, the ``toral part'' of a (polynomial) quantum group  $ \uqg $  is its Cartan subalgebra  $ \uqh \, $:
 the latter identifies with the group (Hopf) algebra of the Abelian group  $ \ZZ^n \, $,
 which in turn is identified with the free Abelian group  $ \mathcal{K}_Q $  generated by the
 $ K_i $'s  via  $ \, K_i \mapsto e_i \, $  (the  $ i $--th  element in the canonical  $ \ZZ $--basis  of
 $ \ZZ^n \, $).  As  $ \ZZ^n $  is isomorphic to the root lattice
 $ \, Q := \sum_{i \in I} \ZZ \, \alpha_i \, $  via  $ \, e_i \mapsto \alpha_i \, $,  we have also an isomorphism
 $ \; \varOmega_Q : Q \;{\buildrel \cong \over {\lhook\joinrel\relbar\joinrel\relbar\joinrel\twoheadrightarrow}}\;
 \mathcal{K}_Q \; $  given by  $ \, \alpha_i \mapsto K_i \; $;  \, we then fix notation
 $ \, K_\alpha := \varOmega_Q(\alpha) \; \big(\! \in \mathcal{K}_Q \big) \, $  for every  $ \, \alpha \in Q \, $.
                                                               \par
   With this notation in use, note that the commutation relations between the  $ K_i^{\pm 1} $'s
   and the  $ E_j $'s  or the  $ F_j $'s  generalize to
  $$  K_\alpha E_j K_\alpha^{-1} \, = \;  q^{+(\alpha,\alpha_j)} E_j   \quad ,  \!\qquad
   K_\alpha F_j K_\alpha^{-1} \, = \;  q^{-(\alpha,\alpha_j)} F_j    \eqno \;\qquad \forall \;\;
   \alpha \in Q \, , \; j \in I  \!\!\quad  $$
where  $ (\ \,,\ ) $  is the symmetric bilinear pairing on  $ \, \QQ{}Q \times \QQ{}Q \, $  introduced in
\S \ref{root-data_Lie-algs}.  Note that  $ \, \pm(\alpha,\alpha_j) \in \ZZ \, $  so that
$ \, q^{\pm(\alpha,\alpha_j)} \in \k\big[ q , q^{-1} \big] \subseteq \k_q \, $:  in particular,
$ \uqg $  {\it is actually defined over the smaller ring  $ \, \k\big[ q , q^{-1} \big] \, $  too}
--- no need of the whole  $ \k_q \, $.
 \vskip5pt
   Let  $ \varGamma $  be a sublattice of  $ \QQ{}Q $  of rank  $ n $  with  $ \, Q \subseteq \varGamma \, $.
   For any basis  $ \big\{ \gamma_1 , \dots \gamma_n \big\} $  of  $ \varGamma \, $,  let
   $ \, C := {\big( c_{ij} \big)}_{i=1,\dots,n;}^{j=1,\dots,n;} \, $  be the matrix of integers such that
   $ \, \alpha_i = \sum_{j=1}^n c_{ij}\,\gamma_j \, $  for every  $ \, i \in I = \{1,\dots,n\} \, $.
   Write $ \, c := \big| \det(C) \big| \in \NN_+ \, $:  as it is known, it equals the index (as a subgroup) of
   $ Q $  in  $ \varGamma \, $;  in particular, it is independent of any choice of bases.
   Write $ \, C^{-1} = {\big( c'_{ij} \big)}_{i=1, \dots, n;}^{j=1, \dots, n;} \, $  for the inverse matrix to
   $ C \, $:  then  $ \, \gamma_i = \sum_{j=1}^n c'_{ij}\,\alpha_j \, $ for each  $ \, i \in I = \{1,\dots,n\} \, $,
   where $ \, c'_{ij} \in \, c^{-1} \, \ZZ \, $,  for all  $ \, i, j \in I \, $.
                                                        \par
  We denote by  $ \, \uqhG{\varGamma} \, $  the group algebra over  $ \k_q $  of the lattice
  $ \varGamma \, $,  with its canonical structure of Hopf algebra.  So, each element
  $ \, \gamma \in \varGamma \, $  corresponds to an element  $ \, K_\gamma \in \uqhG{\varGamma} \, $.
  Denote by  $ \mathcal{K}_\varGamma $  the subgroup of  $ \uqhG{\varGamma} $  generated by all
  these  $ K_\gamma $'s.  This clearly yields a group isomorphism
  $ \; \varOmega_\varGamma : \varGamma \;{\buildrel \cong \over
  {\lhook\joinrel\relbar\joinrel\relbar\joinrel\twoheadrightarrow}}\; \mathcal{K}_\varGamma \; $  given by
  $ \, \gamma \mapsto K_\gamma \, $  that extends the  $ \varOmega_Q $  given above.
  Moreover, as the group  $ Q $  embeds into  $ \varGamma $  we have a corresponding Hopf
  algebra embedding  $ \, \uqh = \uqhG{Q} \lhook\joinrel\relbar\joinrel\longrightarrow \uqhG{\varGamma} \, $.
  In particular, each  $ \, K_i = K_{\alpha_i} \, $  is expressed by the formula
  $ \, K_{\alpha_i} \! = \prod_{j \in I} K_{\gamma_j}^{\,c_{ij}} \, $.
                                                        \par
   Note that each of these extended quantum Cartan (sub)algebras still embeds, in a natural way, inside
   $ \uhh \, $.  To see it, we use again notation from  \S \ref{root-data_Lie-algs}:  consider the
   $ \QQ $--span  of the  $ H_i $'s  as a  $ \QQ $--integral  form of  $ \lieh \, $,
   take the associated isomorphism
   $ \, t : \lieh_{{}_\QQ}^* \,{\buildrel \cong \over {\relbar\joinrel\longrightarrow}}\, \lieh_{{}_\QQ} \, $
   and look at the elements  $ \, T_i := t_{\alpha_i} \! = d_i{}H_i \, $  for all  $ \, i \in I \, $.  By construction, we have
  $$  K_{\alpha_i}  \equiv \,  K_i  \, := \,  e^{\hbar \, d_i H_i}  \, = \,  e^{\hbar \, T_i}   \eqno \forall \;\; i \in I  \quad  $$
and more in general for  $ \, \alpha \in Q \, $  written as  $ \, \alpha = \sum_{j \in I} z_j \, \alpha_j \, $
(with  $ \, z_j \in \ZZ \, $)  we have
  $$  K_\alpha  \, := \,  {\textstyle \prod\limits_{j \in I}} K_j^{z_j}  \, = \,
  {\textstyle \prod\limits_{j \in I}} e^{\hbar \, z_j d_i H_i}  \, = \,  {\textstyle \prod\limits_{j \in I}} e^{\hbar \, z_j T_j}  \, = \,
  e^{\hbar \sum_{j \in I} z_j T_j}  \, = \,  e^{\hbar \, T_\alpha}   \eqno \forall \;\; \alpha \in Q  \quad  $$
Now this picture extends to any other lattice  $ \varGamma $  in  $ \, \lieh_{{}_\QQ}^* \equiv \QQ{}Q \, $
containing $ Q \, $.  Indeed, mapping  $ \, K_\gamma \mapsto e^{\hbar \, T_\gamma} \, $
--- for all  $ \, \gamma \in \varGamma \, $  ---   provides a unique, well-defined monomorphism of
Hopf algebras from  $ \uqhG{\varGamma} $  to  $ \uhh \, $.  In other words,  {\sl  $ \uqhG{\varGamma} $
canonically identifies with the  $ \k_q $--subalgebra  of  $ \uhh $  generated by
$ \, \big\{ K_\gamma \! := e^{\hbar \, T_\gamma} \,\big|\, \gamma \in \varGamma \,\big\} \, $}.
\end{free text}

\vskip9pt

\begin{free text}  \label{pol-qgroups_larg-tor}
 {\bf Polynomial QUEA's with larger torus.}
 We aim now to introduce new polynomial QUEA's having a ``larger Cartan subalgebra'',
 modeled on those of the form  $ \uqhG{\varGamma} $  presented in  \S \ref{larger-tori_QUEA's}  above.
 \vskip4pt
   To begin with, let  $ \varGamma $  be any lattice in  $ \QQ{}Q $
   such that  $ \, Q \subseteq \varGamma \, $.  Like we did for  $ \uqg \, $,
   we define  $ \uqgG{\varGamma} $  as being the unital  $ \k_q $--subalgebra  of
   $ \uhg $  generated by  $ \, \big\{\, E_i \, , \, K_\gamma \! := e^{\hbar \, T_\gamma} ,
   \, F_i \;\big|\; i \! \in \! I , \, \gamma \! \in \! \varGamma \,\big\} \, $.
   From this definition and from the presentation of  $ \uhg $
   we deduce that  $ \uqgG{\varGamma} $  can be presented as
   the associative, unital  $ \k_q $--algebra  with generators
   $ E_i \, $,  $ K_\gamma $  and  $ F_i $  ($ \, i \in I := \{1,\dots,n\} \, $,
   $ \, \gamma \in \varGamma \, $) satisfying \eqref{eq:QUEA-Jimbo-aSerre1} and
   \eqref{eq:QUEA-Jimbo-aSerre2},
   together with the relations
  $$  \displaylines{
   K_{\gamma'} K_{\gamma''} \, = \, K_{\gamma'+\gamma''} \, = \, K_{\gamma''}
   K_{\gamma'}  \quad ,  \qquad  K_{+\gamma} K_{-\gamma} \; = \; 1 \; = \;
   K_{-\gamma} K_{+\gamma}  \cr
   K_\gamma E_j K_\gamma^{-1} \, = \;  q^{+(\gamma,\alpha_j)} E_j   \quad ,
   \qquad  K_\gamma F_j K_\gamma^{-1} \, = \;  q^{-(\gamma,\alpha_j)} F_j  \cr
   E_i F_j - F_j E_i  \; = \;  \delta_{ij} \, {{\, K_{\alpha_i} - K_{\alpha_i}^{-1} \,}
\over {\, q_i - q_i^{-1} \,}}  \cr }  $$
 (for all  $ i, j \in I \, $,  $ \, \gamma , \gamma' , \gamma'' \in \varGamma \, $),
 where  $ \, q_i:= q^{d_i} = e^{\hbar \, d_i} \, $  as before.
 The formulas for the coproduct, antipode and counit in  $ \uhg $  then give
  $$  \displaylines{
   \Delta(E_i) \, = \, E_i \otimes 1 + K_{\alpha_i} \otimes E_i  \;\; ,
   \qquad  \SS(E_i) \, = \, - K_{\alpha_i}^{-1} E_i  \;\; ,  \qquad  \epsilon(E_i) \,
   = \, 0  \cr
   \Delta\big( K_\gamma^{\pm 1} \big) \, = \, K_\gamma^{\pm 1} \otimes
   K_\gamma^{\pm 1}  \;\quad ,  \;\;\qquad  \SS\big( K_\gamma^{\pm 1} \big) \,
   = \, K_\gamma^{\mp 1}  \;\quad ,  \;\;\qquad  \epsilon\big( K_\gamma^{\pm 1} \big) \,
   = \, 1  \cr
   \Delta(F_i) \, = \, F_i \otimes K_{\alpha_i}^{-1} + 1 \otimes F_i  \;\; ,
   \;\qquad  \SS(F_i) \, = \, - F_i \, K_{\alpha_i}  \;\; ,  \;\qquad  \epsilon(F_i) \, = \, 0  }  $$
(for all  $ \, i \! \in \! I \, $,  $ \, \gamma \! \in \! \varGamma \, $)  so that
{\it  $ \uqgG{\varGamma} $  is (again) a  {\sl Hopf subalgebra} inside  $ \uhg \, $}.
                                                                       \par
   With notation of  \S \ref{larger-tori_QUEA's},  let  $ \, C := {\big( c_{ij} \big)}_{i=1,\dots,n;}^{j=1,\dots,n;} \, $
   be the matrix of change of  $ \QQ $--basis  (for  $ \QQ{}Q \, $)  from any basis of  $ \varGamma $
   to the basis  $ \big\{ \alpha_1 , \dots \alpha_n \big\} \, $  of simple roots
 and set  $ \, c := \big| \text{\sl det}(C) \big| \in \NN_+ \, $.  Then from the above presentation one easily sees that
 {\it  $ \uqgG{\varGamma} $  is actually well-defined over the subfield  $ \, \k\big(q^{1/c}\big) $  of  $ \, \k_q \, $}.
 \vskip3pt
   With much the same method we define also quantum Borel subalgebras with larger torus, modeled on the lattice
   $ \varGamma $,  hereafter denoted  $ \, \uqbpG{\varGamma} \, $,  resp.\  $ \, \uqbmG{\varGamma} \, $,
   simply dropping the  $ F_i $',  resp.\ the  $ E_i $'s,  from the set of generators.
                                                                       \par
   Similarly, if we take as generators only the  $ K_\gamma $'s  we get a Hopf  $ \k_q $--subalgebra  of
   $ \uhh $  (hence of  $ \uhg $  as well) that is isomorphic to  $ \uqhG{\varGamma} \, $
   --- cf.\  \S \ref{larger-tori_QUEA's}  above.
                                                                       \par
   Note also that definitions give
   $ \, U_q(\lieb_\pm) = U_{q,Q}(\lieb_\pm) \subseteq U_{q,\varGamma}(\lieb_\pm) \subseteq U_\hbar(\lieb_\pm) \, $.
   Moreover,  {\it all these algebraic objects are actually well-defined over  $ \, \k\big(q^{1/c}\big) $  as well}.
 \vskip2pt
   Now let  $ \varGamma_+ $  and  $ \varGamma_- $  be any two lattices in  $ \QQ{}Q $ such that
   $ \, Q \subseteq \varGamma_\pm \, $,  and let
   $ \, \varGamma_{\!\bullet} := \varGamma_{\!+} \times \varGamma_{\!-} \, $.
   Then we define  $ \uqgdG{\varGamma_{\!\bullet}} $  as being the unital  $ \k_q $--subalgebra  of  $ \uhgd $
   generated by
   $ \, \big\{\, E_i \, , \, K_{\gamma_\pm} \! := e^{\hbar \, T_{\gamma_\pm}} , \, F_i \;\big|\; i \! \in \! I ,
   \, \gamma_\pm \! \in \! \varGamma_{\!\pm} \,\big\} \, $.
   From the presentation of  $ \uhgd $  we deduce that  $ \uqgG{\varGamma_{\!\bullet}} $
   can be presented as the associative, unital  $ \k_q $--algebra
   with generators  $ E_i \, $,  $ K_{\gamma_\pm} $  and  $ F_i $  ($ \, i \in I := \{1,\dots,n\} \, $,
   $ \, \gamma_\pm \in \varGamma_{\!\pm} \, $)
   satisfying the relations \eqref{eq:QUEA-Jimbo-aSerre1} and \eqref{eq:QUEA-Jimbo-aSerre2},
   together with the following:
  $$  \displaylines{
   K_{\gamma_+} K_{\gamma_-}  \; = \;  K_{\gamma_-} K_{\gamma_+}  \cr
   K_{\gamma'_\pm} K_{\gamma''_\pm} \, = \, K_{\gamma'_\pm + \gamma''_\pm} \, = \,
   K_{\gamma''_\pm} K_{\gamma'_\pm}  \quad ,  \qquad  K_{+\gamma_\pm} K_{-\gamma_\pm} \; = \; 1
   \; = \; K_{-\gamma_\pm} K_{+\gamma_\pm}  \cr
   K_{\gamma_\pm} E_j K_{\gamma_\pm}^{-1} \, = \;  q^{+(\gamma_\pm\,,\,\alpha_j)} E_j  \quad ,
   \qquad  K_{\gamma_\pm} F_j K_{\gamma_\pm}^{-1} \, = \;  q^{-(\gamma_\pm\,,\,\alpha_j)} F_j  \cr
   E_i F_j - F_j E_i  \; = \;  \delta_{ij} \, {{\, K_{\alpha^+_i} - K_{\alpha^-_i}^{-1} \,}
\over {\, q_i - q_i^{-1} \,}}  }  $$
 (for all  $ i, j \in I \, $,  $ \, \gamma_\pm, \gamma'_\pm , \gamma''_\pm \in \varGamma_{\!\pm} \, $),
 where  $ \, \alpha^\pm_i $  is the copy of  $ \, \alpha_i \, (\! \in Q) \, $  inside  $ \varGamma_\pm \, $
 and  $ \, q_i:= q^{d_i} = e^{\hbar \, d_i} \, $  as usual.  Moreover, coproduct, antipode and counit of  $ \uhgd $
 on these generators read
  $$  \displaylines{
   \Delta(E_i) \, = \, E_i \otimes 1 + K_{\alpha^+_i} \otimes E_i  \;\; ,
   \qquad  \SS(E_i) \, = \, - K_{\alpha^+_i}^{-1} E_i  \;\; ,  \qquad  \epsilon(E_i) \, = \, 0  \cr
   \Delta\big( K_{\gamma_\pm}^{\pm 1} \big) \, = \, K_{\gamma_\pm}^{\pm 1}
   \otimes K_{\gamma\pm}^{\pm 1}  \;\quad ,  \;\;\qquad
   \SS\big( K_{\gamma_\pm}^{\pm 1} \big) \, = \, K_{\gamma_\pm}^{\mp 1}  \;\quad ,
   \;\;\qquad  \epsilon\big( K_{\gamma_\pm}^{\pm 1} \big) \, = \, 1  \cr
   \Delta(F_i) \, = \, F_i \otimes K_{\alpha^-_i}^{-1} + 1 \otimes F_i  \;\; ,
   \;\qquad  \SS(F_i) \, = \, - F_i \, K_{\alpha^-_i} \;\; ,  \;\qquad  \epsilon(F_i) \, = \, 0  }  $$
(for all  $ \, i \! \in \! I \, $,  $ \, \gamma_\pm \! \in \! \varGamma_{\!\pm} \, $).
Thus  {\it  $ \uqgdG{\varGamma_{\!\bullet}} $  is also a  {\sl Hopf subalgebra} inside  $ \uhgd \, $}.
 \vskip5pt
   Fix again two lattices  $ \varGamma_+ $  and  $ \varGamma_- $  in  $ \QQ{}Q $  such that
   $ \, Q \subseteq \varGamma_\pm \, $.
   As  $ \, U_{q,\varGamma_{\!\pm}}(\lieb_\pm) \subseteq U_\hbar(\lieb_\pm) \, $,  the skew-pairing
   $ \; \eta : \uhbp \otimes_{\kh} \uhbm \relbar\joinrel\longrightarrow \kh \; $  restricts to a similar skew-Hopf pairing
   $ \; \eta : \uqbpG{\varGamma_{\!+}} \otimes_{\k_q} \! \uqbmG{\varGamma_{\!-}} \relbar\joinrel\longrightarrow \k_q \, $,
   which in turn extends the one on  $ \, \uqbp \otimes_{\k_q} \! \uqbm = \uqbpG{Q} \otimes_{\k_q} \! \uqbmG{Q} \, $;
   the formulas
  $$  \eta\big(K_{\gamma_+},K_{\gamma_-}\big) = q^{-(\gamma_+,\gamma_-)}  \; ,
  \;\;\;  \eta\big(E_{i\,},F_j\big) = {{\,\delta_{ij}\,} \over {\,q_i^{-1\,} \!\! - q_i\,}}  \; ,
  \;\;\;  \eta\big(E_{i\,},K_j\big) = 0 = \eta\big(K_{i\,},F_j\big)  $$
 uniquely determines this pairing.  Using the latter, we construct the corresponding Drinfeld double
 $ \, D\big( {\uqbpG{\varGamma_{\!+}}}, \uqbmG{\varGamma_{\!-}} , \eta \big) \, $  as in  \S \ref{conv-Hopf}.
 It follows by construction that this double coincides with the Hopf algebra  $ \uqgdG{\varGamma_{\!\bullet}} $
 considered right above.
 \vskip5pt
   Again by construction, we also have that the projection
   $ \; \pi_\lieg : \uhgd \relbar\joinrel\relbar\joinrel\twoheadrightarrow \uhg \; $
   yields by restriction a Hopf  $ \k_q $--algebra  epimorphism
   $ \; \pi_\lieg : \uqgdG{\varGamma_{\!\bullet}} \relbar\joinrel\relbar\joinrel\twoheadrightarrow \uqgG{\varGamma_{\!*}} \; $,
   \, where
   $ \, \varGamma_{\!*} := \varGamma_{\!+} + \varGamma_{\!-} \; \big(\! \subseteq \QQ{}Q \big) \, $
   whose explicit description is obvious.
 \vskip7pt
   An alternative method to construct these QUEA's with larger toral part goes as follows.
   Fix a lattice  $ \varGamma $  in  $ \QQ{}Q $  such that  $ \, Q \subseteq \varGamma \, $,
   and consider the associated Hopf algebra  $ \uqhG{\varGamma} $  as in
   \S \ref{larger-tori_QUEA's}  and the canonical embedding
   $ \, \uqh = \uqhG{Q} \subseteq \uqhG{\varGamma} \, $.
   The natural (adjoint) action of  $ \uqh $  onto  $ \uqg $  extends (uniquely) to a  $ \uqhG{\varGamma} $--action
  $ \; \cdot :\, \uqhG{\varGamma} \times \uqg \relbar\joinrel\longrightarrow \uqg \; $
 given (for all  $ \, i \in I \, $  and  $ \, \gamma \in \varGamma \, $)  by
  $$  K_\gamma \,\cdot \, E_j  \, = \,  q^{+(\gamma\,,\alpha_j)} \, E_j  \quad ,  \qquad
     K_\gamma \,\cdot \, F_j  \, = \,  q^{-(\gamma\,,\alpha_j)} \, F_j  \quad ,  \qquad
    K_\gamma \,\cdot \, K_j  \, = \,   K_j  $$
 that makes  $ \uqg $  into a  $ \uqhG{\varGamma} $--module algebra.  Then we can consider the Hopf algebra
 $ \, \uqhG{\varGamma} \ltimes \uqg \, $  given by the {\it smash product\/}  of  $ \uqhG{\varGamma} $  and  $ \uqg \, $:
 the underlying vector space is just $ \, \uqhG{\varGamma} \otimes_{\k_q} \uqg \, $,  the coalgebra structure is the one
 given by the tensor product of the corresponding coalgebras, and the product is given by the formula
\begin{equation}\label{eq:smashprod}
  (h \ltimes x) \, (k\ltimes y)  \; = \;  h \, k_{(1)} \ltimes \big( \SS(k_{(2)}) \cdot x \big) \, y
\end{equation}
 for all  $ \, h , k \in \uqhG{\varGamma} \, $,  $ \, x , y \in \uqg \, $.  Since  $ \uqhG{\varGamma} $  contains
 $ \, \uqh \, $  as a Hopf subalgebra, it follows that  $ \uqhG{\varGamma} $  itself is a right
 $ \uqh $--module  Hopf algebra with respect to the adjoint action.  Under these conditions,
 it is easy to see that the smash product  $ \, \uqhG{\varGamma} \ltimes \uqg \, $
 maps onto a Hopf algebra structure on the  $ \k_q $--module
 $ \, \uqhG{\varGamma} \hskip-3pt \mathop{\otimes}\limits_{\uqh} \hskip-3pt \uqg \, $,
 which we can denote by  $ \, \uqhG{\varGamma} \hskip-3pt \mathop{\ltimes}\limits_{\uqh} \hskip-3pt \uqg \, $,
 see  \cite[Theorem 2.8]{Len}.  Finally, tracking the whole construction one easily sees that this Hopf algebra
 $ \, \uqhG{\varGamma} \hskip-3pt \mathop{\ltimes}\limits_{\uqh} \hskip-3pt \uqg \, $
actually coincides with the Hopf algebra  $ \uqgG{\varGamma} $  considered above.
 \vskip5pt
   With the same approach, one can also realize  $ U_{q,\varGamma_{\!\pm}}(\lieb_\pm) \, $,  resp.\
   $ \uqbmG{\varGamma} \, $,  as Hopf algebra structure on
   $ \, \uqhG{\varGamma_{\!\pm}} \!\! \mathop{\otimes}\limits_{\uqh} \!\! U_q(\lieb_\pm) \, $,  resp.\
   $ \, \uqhG{\varGamma_{\!\bullet}} \hskip-17pt \mathop{\otimes}\limits_{\raise-3pt\hbox{$ \scriptstyle \uqh \otimes \uqh $}}
   \hskip-17pt \uqgd \, $,  obtained as quotient of the smash product Hopf algebra
   $ \, \uqhG{\varGamma_{\!\pm}} \ltimes U_q(\lieb_\pm) \, $,  resp.\  $ \, \uqhG{\varGamma_{\!\bullet}} \ltimes \uqgd \, $.
\end{free text}

\medskip

\subsection{ Comultiplication twistings of polynomial QUEA's (=TwQUEA's)}  \label{twist_polyn-QUEAs}
 {\ }

\smallskip

   We introduce now the polynomial version of twisted QUEA's, or twisted polynomial QUEA's,
   just by matching what we did in  \S \ref{memo_polyn-QUEAs}  and  \S \ref{twist_formal-QUEAs}.

\vskip11pt

\begin{free text}  \label{twist-Uqg}
 {\bf Comultiplication twistings of  $ \uqg \, $.}
 Let  $ \, \k_q := \lim\limits_{\buildrel {\leftarrow\joinrel\relbar\joinrel\relbar}
 \over {m \in \NN}} \k\big(q^{1/m}\big) \; \big(\! \subseteq \kh \big) \, $  and  $ \uqg $  be as in
 \S \ref{memo_polyn-QUEAs}.  Fix also an  $ (n \times n) $--matrix
 $ \; \Psi := {\big( \psi_{ij} \big)}_{i,j \in I} \in M_n\big(\kh\big) \; $  as in  \S \ref{twist-Uhg},
 but now  {\sl we make the stronger assumption that  $ \; \Psi := {\big( \psi_{ij} \big)}_{i,j \in I} \in M_n\big(\QQ\big) \; $}.
 \vskip5pt
   Via the recipe in  \S \ref{twist-Uhg},  we pick the corresponding twisted formal QUEA  $ \upsihg \, $.
   Then inside the latter we consider the unital  $ \k_q $--subalgebra  $ \uqg \, $,
   generated by
   $ \, \big\{ E_i \, , \, K_i^{\pm 1} \! := \! e^{\pm \hbar \, d_i H_i} \! = \! e^{\pm \hbar \, T_i} , \, F_i \;\big|\; i \! \in \! I \,\big\} \, $,
   for which we have an explicit presentation.
 On the other hand, the new, twisted Hopf structure of  $ \upsihg $  on the generators of  $ \uqg $
 --- with notation of  \S \ref{root-twisting}  and \S \ref{larger-tori_QUEA's}  ---   reads
  $$  \displaylines{
   \Delta^{\scriptscriptstyle \!(\Psi)}\!\big(E_\ell\big)  \, =
   \,  E_\ell \otimes K_{+\psi_-(\alpha_\ell)} + K_{+(\text{id} + \psi_+)(\alpha_\ell)} \otimes E_\ell  \, =
   \,  E_\ell \otimes K_{+\zeta^-_\ell} + K_{+\alpha_\ell + \zeta^+_\ell} \otimes E_\ell
 \cr
   \Delta^{\scriptscriptstyle \!(\Psi)}\big(K_\ell^{\pm 1}\big)  \, = \,  K_\ell^{\pm 1} \otimes K_\ell^{\pm 1}  \cr
   \Delta^{\scriptscriptstyle \!(\Psi)}\!\big(F_\ell\big)  \, = \,
   F_\ell \otimes K_{-(\text{id} + \psi_-)(\alpha_\ell)} + K_{-\psi_+(\alpha_\ell)} \otimes F_\ell  \, =
   \,  F_\ell \otimes K_{-\alpha_\ell - \zeta^-_\ell} + K_{-\zeta^+_\ell} \otimes F_\ell  \cr
   \SS^{\scriptscriptstyle (\Psi)}\big(E_\ell\big)  \, = \,
 - K_{-(\id_\lieh + \, \psi_+)(\alpha_\ell)} \, E_\ell \, K_{-\psi_-(\alpha_\ell)}  \, = \,
 - K_{-\alpha_\ell - \zeta^+_\ell} \, E_\ell \, K_{-\zeta^-_\ell}  \;\; ,  \qquad  \epsilon\big(E_\ell\big)  \, = \,  0  \cr
   \SS^{\scriptscriptstyle (\Psi)}\big(K_\ell^{\pm 1}\big)  \, = \,  K_\ell^{\mp 1}  \qquad ,
 \qquad \qquad \qquad \qquad  \epsilon\big(K_\ell^{\pm 1}\big)  \, = \,  1  \cr
   \SS^{\scriptscriptstyle (\Psi)}\big(F_\ell\big)  \, = \,
 - K_{+\psi_+(\alpha_\ell)} \, F_\ell \, K_{+(\id_\lieh + \, \psi_-)(\alpha_\ell)}  \, =
 \,  - K_{+\zeta^+_\ell} \, F_\ell \, K_{+\alpha_\ell + \zeta^-_\ell}  \;\; ,
 \qquad  \epsilon\big(F_\ell\big)  \, = \,  0  }  $$
 (for all  $ \, \ell \in I \, $).  This shows explicitly that
 \vskip4pt
   \centerline{\it  $ \uqg $  is a  {\sl Hopf subalgebra}  (over\/  $ \k_q \, $)  inside  $ \upsihg $
   {\sl if and only if}  $ \, \psi_\pm(Q) \subseteq Q \, $.}
 \vskip4pt
   In order to settle this point, we pick the sublattice  in  $ \QQ{}Q \, $ given by
   $ \, Q^{\scriptscriptstyle \Psi} := Q + \psi_+(Q) + \psi_-(Q) \, $,  and for each lattice
   $ \varGamma $  in  $ \QQ{}Q $  containing  $ Q^{\scriptscriptstyle \Psi} $
   we consider the corresponding polynomial QUEA, namely  $ \uqgG{\varGamma} $  as in
   \S \ref{pol-qgroups_larg-tor}.  Then  $ \uqgG{\varGamma} $  is naturally embedded inside  $ \uhg \, $,
   and repeating the previous analysis we see that
\begin{equation}  \label{UpsiqgGpsi-subHopf-Upsihg}
   \text{\it  $ \uqgG{\varGamma} $  is a  {\sl Hopf subalgebra}  (over\/  $ \k_q \, $)  inside  $ \upsihg \, $.}
\end{equation}
                                                             \par
   To sum up, the previous analysis allows us to give the following definition:

\vskip11pt

\begin{definition}  \label{def: UpsiqgG}
 We denote by  $ \upsiqgG{\varGamma} $  the Hopf algebra defined in  \eqref{UpsiqgGpsi-subHopf-Upsihg},
 whose Hopf structure is given by restriction from  $ \upsihg \, $.
 We call any such Hopf algebra  {\sl twisted polynomial QUEA}   --- or  {\sl TwQUEA},  in short ---
 saying it is obtained from  $ \uqgG{\varGamma} $  by twisting (although, strictly speaking, this is not entirely correct).
\end{definition}

\vskip5pt

\begin{rmk}
 The multiparameter quantum groups  $ U^{\varphi}_q(\lieg) $  introduced by Costantini and Varagnolo in
 \cite{CV1,CV2,CV3}  are a particular case of a twisted polynomial QUEA, obtained by taking
 $ \, 2\,\psi_- = -2\,\psi_+ = \varphi \, $.  More precisely, they fix assumptions on  $ \varphi $
 --- hence on  $ \psi $  ---   that guarantee that it is enough to take, once and for all,
 the ``larger torus'' modeled on the lattice  $ \, \varGamma = P \, $.
\end{rmk}
\end{free text}

\vskip11pt

\begin{free text}  \label{twist-Uqgd}
 {\bf Comultiplication twistings of  $ \uqgd \, $.}
 Let again  $ \, \k_q := \lim\limits_{\buildrel {\leftarrow\joinrel\relbar\joinrel\relbar} \over {m \in \NN}} \k \big(q^{1/m} \big) \, $
 and  $ \; \Psi := {\big( \psi_{ij} \big)}_{i,j \in I} \in M_n\big(\QQ\big) \; $  be as in  \S \ref{twist-Uqg}.
 Following  \S \ref{twist-Uhgd},  we pick the twisted formal QUEA  $ \upsihgd $ \, :
 inside the latter, we pick the unital  $ \k_q $--subalgebra  $ \uqgd $  of  \S \ref{Cartan-Borel_pol-QUEAs},
 for which we know an explicit presentation.

\vskip3pt

   Since we are working with a double copy of  $ \lieh' $,  hence also with
   $ \, \QQ{}(Q \times Q) = \QQ{}Q \times \QQ{}Q \, $,  we write  $ \, Q_\pm \, $  to mean
   $ \, Q_+ := Q \times \{0\} \, $  and  $ \, Q_- := \{0\} \times Q \, $,  and also   --- for every
   $ \, \gamma \in \QQ{}Q \, $  ---   we set  $ \, \gamma_+ := (\gamma,0) \, $  and
   $ \, \gamma_- := (0,\gamma) \, $ inside  $ \, \QQ{}Q \times \QQ{}Q \, $.
   Besides this, hereafter the notation is that of  \S \ref{root-twisting}  and \S \ref{larger-tori_QUEA's}.

\vskip3pt

   The twisted Hopf structure of  $ \upsihgd $  on the generators of  $ \uqgd $  yields
  $$  \displaylines{
   \Delta^{\scriptscriptstyle \!(\Psi)}\!\big(E_\ell\big)  \, =
   \,  E_\ell \otimes K_{+\psi_-(\alpha^-_\ell)} + K_{+(\text{id} + \psi_+)(\alpha^+_\ell)} \otimes E_\ell \, =
   \,  E_\ell \otimes K_{+\zeta^-_{\ell,-}} + K_{+\alpha^+_\ell + \zeta^+_{\ell,+}} \otimes E_\ell  \cr
   \Delta^{\scriptscriptstyle \!(\Psi)}\big(K_\ell^{\pm 1}\big)  \, =
   \,  K_\ell^{\pm 1} \otimes K_\ell^{\pm 1}  \cr
   \Delta^{\scriptscriptstyle \!(\Psi)}\!\big(F_\ell\big)  \, =
   \,  F_\ell \otimes K_{-(\text{id} + \psi_-)(\alpha^-_\ell)} + K_{-\psi_+(\alpha^+_\ell)} \otimes F_\ell  \, =
   \,  F_\ell \otimes K_{-\alpha^-_\ell - \zeta^-_{\ell,-}} + K_{-\zeta^+_{\ell,+}} \otimes F_\ell  \cr
%
  }  $$
  $$  \displaylines{
   \SS^{\scriptscriptstyle (\Psi)}\big(E_\ell\big)  \, = \,
%
%
 - K_{-(\id_\lieh + \, \psi_+)(\alpha^+_\ell)} \, E_\ell \, K_{-\psi_-(\alpha^-_\ell)}  \, = \,
 - K_{-\alpha^+_\ell - \zeta^+_{\ell,+}} \, E_\ell \, K_{-\zeta^-_{\ell,-}}  \;\; ,
 \!\!\!\qquad  \epsilon\big(E_\ell\big)  \, = \,  0  \cr
   \SS^{\scriptscriptstyle (\Psi)}\big(K_\ell^{\pm 1}\big)  \, = \,  K_\ell^{\mp 1}  \qquad ,
 \qquad \qquad \qquad \qquad  \epsilon\big(K_\ell^{\pm 1}\big)  \, = \,  1  \cr
   \SS^{\scriptscriptstyle (\Psi)}\big(F_\ell\big)  \, = \,
%
%
 - K_{+\psi_+(\alpha^+_\ell)} \, F_\ell \, K_{+(\id_\lieh + \, \psi_-)(\alpha^-_\ell)}  \, =
 \,  - K_{+\zeta^+_{\ell,+}} \, F_\ell \, K_{+\alpha^-_\ell + \zeta^-_{\ell,-}}  \;\; ,
 \!\!\qquad  \epsilon\big(F_\ell\big)  \, = \,  0  }  $$
for all  $ \, \ell \in I \, $.  From the formulas above we see that
 \vskip4pt
   \centerline{\it  $ \uqgd $  is a  {\sl Hopf subalgebra}  (over\/  $ \k_q \, $)  in
   $ \upsihgd $  {\sl if and only if}  $ \, \psi_\pm(Q_\pm) \subseteq Q_\pm \, $.}
 \vskip4pt
   To fix this issue, we consider the sublattices
   $ \, Q^{\scriptscriptstyle \Psi}_{(\pm)} := (\text{id} + \psi_\pm)(Q_\pm) + \psi_\mp(Q_\mp) \, $
%
%
 inside  $ \, \QQ{}Q \times \QQ{}Q \, $  and their sum
 $ \, Q^{\scriptscriptstyle \Psi}_\ast := Q^{\scriptscriptstyle \Psi}_{(+)} + Q^{\scriptscriptstyle \Psi}_{(-)} \; $;
 then for each lattice  $ \varGamma_{\!\circ} $  inside  $ \, \QQ{}Q \times \QQ{}Q \, $ containing
 $ Q^{\scriptscriptstyle \Psi}_\ast $  we consider the associated polynomial QUEA
 $ \uqgdG{\varGamma_{\!\circ}} $  as in  \S \ref{pol-qgroups_larg-tor}.  By construction
 $ \uqgdG{\varGamma_{\!\circ}} $  sits inside  $ \uhgd \, $,  and repeating the previous analysis we find that
\begin{equation}  \label{UpsiqgdGpsi-subHopf-Upsihgd}
   \text{\it $ \uqgdG{\varGamma_{\!\circ}} $  is a  {\sl Hopf subalgebra}  (over\/  $ \k_q \, $)  inside  $ \upsihgd \, $.}
\end{equation}
 \vskip4pt
   At the end of the day, we are allowed to give the following definition:

\vskip11pt

\begin{definition}  \label{def: UpsiqgdG}
 We denote by  $ \upsiqgdG{\varGamma_{\!\circ}} $  the Hopf algebra given in  \eqref{UpsiqgdGpsi-subHopf-Upsihgd},
 whose Hopf structure is given by restriction from  $ \upsihgd \, $.
 We call any such Hopf algebra  {\sl twisted polynomial QUEA}   --- or  {\sl TwQUEA},  in short ---
 saying it is obtained from  $ \uqgdG{\varGamma_{\!\circ}} $  by twisting (although, strictly speaking,
 this is not entirely correct).
\end{definition}

\vskip9pt

\begin{rmk}  \label{rem: UpsiqgdG ---> UpsiqgG}
 It follows by construction that the epimorphism of twisted formal QUEA's
 $ \; \pi^{\scriptscriptstyle \Psi}_\lieg \! := \pi_\lieg : \upsihgd \relbar\joinrel\relbar\joinrel\twoheadrightarrow \upsihg \; $
 (cf.\  \S \ref{twist-Uhgd})  restricts to an epimorphism
%
%
 $ \; \pi^{\scriptscriptstyle \Psi}_\lieg \! := \pi_\lieg : \upsiqgdG{\varGamma_{\!\circ}}
 \relbar\joinrel\relbar\joinrel\twoheadrightarrow \upsiqgG{\varGamma} \; $
 of twisted polynomial QUEA's, where  $ \varGamma $  is the image of  $ \varGamma_{\circ} $  for the
%
%
 projection of  $ \QQ(Q \times Q) $  onto  $ \QQ{}Q $  mapping
 $ \, \alpha_\pm \in Q_\pm \, $  onto  $ \, \alpha \in Q \, $.
\end{rmk}
\end{free text}

\vskip11pt

\begin{free text}  \label{twist-Uqbpm}
 {\bf Comultiplication twistings of  $ U_q(\lieb_\pm) \, $.}
 We still work with
 $ \, \Psi \! := \! {\big( \psi_{ij} \big)}_{i,j \in I} \! \in \! M_n(\QQ) \, $  as in  \S \ref{twist-Uqg}.
 Using it, we define suitable
 ``comultiplication twistings'' of the quantum Borel algebras
 $ U_q(\lieb_\pm) $  {\sl as Hopf subalgebras inside  $ U_q^{\scriptscriptstyle \Psi}(\lieg) $
 and  inside  $ U_q^{\scriptscriptstyle \Psi}(\liegd) \, $.}
 \vskip5pt
   First we pick  in  $ \QQ{}Q \, $  the sublattices
   $ \, Q^{\scriptscriptstyle \Psi}_\pm := (\text{id} + \psi_\pm)(Q) + \psi_\mp(Q) \, $,
   and for any lattice  $ \varGamma_\pm $  in  $ \QQ{}Q $  containing
   $ Q^{\scriptscriptstyle \Psi}_\pm $  we consider  $ \, \varGamma_* := \varGamma_+ +  \varGamma_- \, $
   and the corresponding  $ \uqgG{\varGamma_{*\!}} \, $,  as in  \S \ref{twist-Uqg};
   inside the latter, we consider the (polynomial) quantum Borel (or ``Borel-like'')  subalgebra
   $ U_{q,\varGamma_\pm}(\lieb_\pm) \, $.  Now, the formulas for the twisted Hopf structure of
   $ \uqgG{\varGamma_*\!} $  show that the subalgebra  $ U_{q,\varGamma_\pm}(\lieb_\pm) $
   is also a  {\sl Hopf subalgebra\/}  in  $ \upsiqgG{\varGamma_*} \, $.
   Therefore  $ U_{q,\varGamma_\pm}(\lieb_\pm) $  {\sl with the twisted coproduct,
   antipode and counit is a new Hopf algebra, that we denote by}
   $ U^{\scriptscriptstyle \Psi}_{q,\varGamma_\pm}\!(\lieb_\pm) \, $,
   and call (polynomial) ``twisted'' quantum Borel subalgebra.
                                                  \par
   Note also that both  $ U^{\scriptscriptstyle \Psi}_{q,\varGamma_+}\!(\lieb_+) $
   and  $ U^{\scriptscriptstyle \Psi}_{q,\varGamma_-}\!(\lieb_-) $  are Hopf subalgebras in
   $ \upsiqgG{\varGamma_*\!} \, $.
 \vskip5pt
   Second, we pick  in  $ \QQ{}(Q \times Q) \, $  the sublattices
   $ \, Q^{\scriptscriptstyle \Psi}_{(\pm)} := (\text{id} + \psi_\pm)(Q_\pm) + \psi_\mp(Q_\mp) \, $,
   and for any lattice  $ \varGamma_{(\pm)} $  in  $ \QQ{}(Q \times Q) $  containing
   $ Q^{\scriptscriptstyle \Psi}_{(\pm)} $  we consider
   $ \, \varGamma_{(*)} := \varGamma_{(+)} +  \varGamma_{(-)} \, $
   and the corresponding (double) quantum group  $ \upsiqgdG{\varGamma_{(*)\!}} \, $;
   inside the latter, we fix the (polynomial) quantum Borel subalgebra
 $ \, U^{\scriptscriptstyle \Psi}_{q,\varGamma_{(\pm)}}\!(\lieb_\pm) $.
%
%
 Then the explicit formulas for the (twisted) Hopf structure of
 $ \upsiqgdG{\varGamma_{(*)\!}} $ show that
 \vskip5pt
 \centerline{\it  $ U^{\scriptscriptstyle \Psi}_{q,\varGamma_{(\pm)\!}}(\lieb_\pm) $
 is a Hopf subalgebra of  $ \, \upsiqgdG{\varGamma_{(*)\!}} \, $.}
 \vskip5pt
   However, a major drawback of both the subalgebras
   $ U^{\scriptscriptstyle \Psi}_{q,\varGamma_\pm}\!(\lieb_\pm) $
   --- inside  $ \upsiqgG{\varGamma_{*\!}} $  ---   and
   $ U^{\scriptscriptstyle \Psi}_{q,\varGamma_{(\pm)\!}}(\lieb_\pm) $
   --- inside  $ \upsiqgdG{\varGamma_{(*)\!}} $  ---   is that  {\sl they have too large a toral part\/}
   to be rightfully called ``(polynomial) quantum Borel (sub)algebras''.
   We tackle and settle this problem in  \S \ref{twist-gen's_x_twist-QUEA's}  hereafter.
\end{free text}

\vskip11pt

\begin{free text}  \label{twist-gen's_x_twist-QUEA's}
 {\bf Twisted generators for polynomial twisted QUEA's.}
 Let us consider in  $ \QQ(Q\times Q) $  the elements
 $ \; \tau_i^\pm := (\text{id} + \psi_\pm)(\alpha_i^\pm) - \psi_\mp(\alpha_i^\mp) = \alpha_i^\pm +
 \zeta_{i,\pm}^\pm - \zeta_{i,\mp}^\mp \; (i \in I) \; $   --- cf.\ \S \ref{root-twisting}  ---
 and the sublattices  $ Q^{\scriptscriptstyle \Psi}_{[\pm]} $  with
 $ \ZZ $--basis  $ \, \big\{\, \tau_i^\pm \,\big|\; i \in I \,\big\} \, $.
                                                       \par
   Let  $ \, Q^{\scriptscriptstyle \Psi}_{(\pm)} := (\text{id} + \psi_\pm)(Q_\pm) + \psi_\mp(Q_\mp) \, $
   and  $ \varGamma_{(\pm)} $  be as in  \S \ref{twist-Uqbpm}  above; then
   $ Q^{\scriptscriptstyle \Psi}_{[\pm]} $  is a sublattice in the lattice
   $ Q^{\scriptscriptstyle \Psi}_{(\pm)} \, $,  hence in  $ \varGamma_{(\pm)} $  as well.
   Inside the (twisted) Borel QUEA
$ \, U^{\,\scriptscriptstyle \Psi}_{q,\varGamma_{(\pm)}}\!(\lieb_\pm) $  we consider the elements
  $$  \displaylines{
   E^{\scriptscriptstyle \Psi}_i := K_{-\psi_-(\alpha^-_i)} E_i = K_{-\zeta^-_{i,-}} E_i  \;\; , \quad
   K^{\scriptscriptstyle \Psi}_{i,+} := K_{( \id + \, \psi_+)(\alpha^+_i) - \psi_-(\alpha^-_i)} =
   K_{\alpha^+_i + \, \zeta^+_{i,+} - \,\zeta^-_{i,-}}  \cr
   F^{\scriptscriptstyle \Psi}_i := K_{+\psi_+(\alpha^+_i)} F_i = K_{+\zeta^+_{i,+}} F_i  \;\; , \quad
   K^{\scriptscriptstyle \Psi}_{i,-} := K_{( \id \! + \, \psi_-)(\alpha^-_i) - \psi_+(\alpha^+_i)} =
   K_{\alpha^-_i + \, \zeta^-_{i,-} \! - \,\zeta^+_{i,+}}  }  $$
 (for all  $ \, i \in I \, $)   --- hereafter called ``twisted generators'' ---   and then
 {\sl define  $ \, {\hat{U}}^{\scriptscriptstyle \Psi}_q(\lieb_+) \, $,  resp.\
 $ \, {\hat{U}}^{\scriptscriptstyle \Psi}_q(\lieb_-) \, $,  as being the unital\/  $ \k_q $--subalgebra  of
 $ U^{\,\scriptscriptstyle \Psi}_{q,\varGamma_{(\pm)\!}}(\lieb_\pm) $  generated by the
 $ E^{\scriptscriptstyle \Psi}_i $'s  and the  $ {\big( K^{\scriptscriptstyle \Psi}_{i,+} \big)}^{\pm 1} \, $'s,
 resp.\ by the  $ F^{\scriptscriptstyle \Psi}_i $'s  and the  $ {\big( K^{\scriptscriptstyle \Psi}_{i,-} \big)}^{\pm 1} \, $'s}.
                                                                         \par
   More in general, for any other sublattice  $ M_\pm $  containing
   $ Q^{\scriptscriptstyle \Psi}_{[\pm]} $  {\sl we define  $ \, \uhatpsiqbpG{M_+\!} \, $,
   resp.\  $ \, \uhatpsiqbmG{M_-\!} \, $,  as the unital  $ \, \k_q $--subalgebra  of
   $ \, U_{q,M_\pm + Q^\Psi_{(\pm)}}\!(\lieb_\pm) $  generated  by the  $ E^{\scriptscriptstyle \Psi}_i $'s
   and the  $ K_{y_+} $'s,  resp.\  by the $ F^{\scriptscriptstyle \Psi}_i $'s  and the  $ K_{y_-} $'s,
   with  $ \, i \in I \, $  and  $ \, y_\pm \! \in \! M_\pm \, $}.
 \vskip5pt
   The key fact is that the (twisted) Hopf structure of
   $ \, U^{\scriptscriptstyle \Psi}_{q,M_\pm + Q^\Psi_{(\pm)}}\!(\lieb_\pm) \, $  yields
  $$  \displaylines{
   \Delta^{\scriptscriptstyle \!(\Psi)}\! \big(E^{\scriptscriptstyle \Psi}_\ell\big)
\, = \,  E^{\scriptscriptstyle \Psi}_\ell \otimes 1 + K_{+\tau^+_\ell} \otimes E^{\scriptscriptstyle \Psi}_\ell  \;\; ,
  \;\quad  \SS^{\scriptscriptstyle (\Psi)}\big(E^{\scriptscriptstyle \Psi}_\ell\big)
  =  - K_{-\tau^+_\ell} E^{\scriptscriptstyle \Psi}_\ell  \;\; ,
  \;\quad  \epsilon\big(E^{\scriptscriptstyle \Psi}_\ell\big) \, = \,  0  \cr
   \Delta^{\scriptscriptstyle \!(\Psi)}\big(K_{y_\pm}\big)  \, = \,  K_{y_\pm} \otimes K_{y_\pm}  \;\; ,  \;\quad
 \SS^{\scriptscriptstyle (\Psi)}\big(K_{y_\pm}\big)  \, = \,  K_{y_\pm}^{\,-1}  \, = \,  K_{-y_\pm}  \;\; ,
 \;\quad  \epsilon\big(K_{y_\pm}\big)  \, = \,  1  \cr
   \Delta^{\scriptscriptstyle \!(\Psi)}\! \big(F^{\scriptscriptstyle \Psi}_\ell\big)
\, = \,  F^{\scriptscriptstyle \Psi}_\ell \otimes K_{-\tau^-_\ell} + 1 \otimes F^{\scriptscriptstyle \Psi}_\ell  \;\; ,
 \;\quad  \SS^{\scriptscriptstyle (\Psi)}\big(F^{\scriptscriptstyle \Psi}_\ell\big)  =
- F^{\scriptscriptstyle \Psi}_\ell K_{+\tau^-_\ell}  \;\; ,  \;\quad
\epsilon\big(F^{\scriptscriptstyle \Psi}_\ell\big) \, = \,  0  }  $$
 for all  $ \, \ell \in I \, $.  Altogether, these formulas show that
 \vskip7pt
   \centerline{\it  $ \, \uhatpsiqbpmG{M_\pm\!} \, $  is a Hopf subalgebra of
   $ \, U^{\scriptscriptstyle \Psi}_{q,M_\pm + Q^\Psi_{(\pm)}}\!(\lieb_\pm) \, $.}
 \vskip9pt
   It is clear from definitions that if  $ \, M_\pm \supseteq Q^{\scriptscriptstyle \Psi}_{(\pm)} \, $
   then we have  $ \, \uhatpsiqbpmG{M_\pm} = \upsiqbpG{M_\pm} \, $   ---  cf.\ \S \ref{twist-Uqbpm}.
   Thus  $ \, \uhatpsiqbpG{M_+} = \upsiqbpG{M_+} \, $  can be generated by either set of generators
   $ \, {\big\{ E_i \, , K_{y_+} \big\}}_{i \in I \, , \, y_+ \in M_+} \, $  or
   $ \, {\big\{ E^{\scriptscriptstyle \Psi}_i , K_{y_+} \big\}}_{i \in I \, , \, y_+ \in M_+} \, $,  while
   $ \, \uhatpsiqbmG{M_-} = \upsiqbmG{M_-} \, $  can be generated by either
   $ \, {\big\{ F_i \, , K_{y_-} \big\}}_{i \in I \, , \, y_- \in M_-} \, $  or
   $ \, {\big\{ F^{\scriptscriptstyle \Psi}_i , K_{y_-} \big\}}_{i \in I \, , \, y_- \in M_-} \, $.
                                                            \par
   An entirely similar remark applies if we use the  $ E^{\scriptscriptstyle \Psi}_i $'s,  $ K_{y_\pm} \! $'s
   and  $ F^{\scriptscriptstyle \Psi}_i $'s  altogether as algebra generators of
$ \, \upsiqgdG{\varGamma_{(\ast)\!}} \, $  for
$ \, \varGamma_{(\ast)} := \varGamma_{(+)} + \varGamma_{(-)} \, $  with
$ \, \varGamma_{(\pm)} \supseteq Q^{\scriptscriptstyle \Psi}_{(\pm)} \; $.
                                                             \par
   Of course we can repeat this analysis for quantum Borel subalgebras inside
   $ \, \upsiqgG{\varGamma} \, $   --- for suitable lattices  $ \varGamma $  in  $ \QQ{}Q \, $
   ---   and introduce ``twisted'' algebra generators for them, that will also be, altogether,
   (twisted) algebra generators of  $ \, \upsiqgG{\varGamma} \, $  itself.
 \vskip5pt
   Now, when we use the ``twisted'' generators, the above shows that the formulas for
   the (twisted) coproduct, antipode and counit look exactly like those for the generators of the
   {\sl untwisted\/}  Borel subalgebras.  In other words, for these quantum Borel subalgebras
   {\sl twisting the algebra generators (as above) we end up with untwisted formulas for the
   coalgebra and antipodal structure}.  This remark applies to  {\sl both\/}  cases: twisted Borel
   subalgebras (and their generators) inside a double quantum group
   $ \, \upsiqgdG{\varGamma_{\!\circ}} \, $  {\sl and\/}  embedded inside a (``single'') quantum group
   $ \, \upsiqgG{\varGamma} \, $.
                                                             \par
   In all these new presentations, the new (twisted) generators of our Borel algebras
   (inside  $ \, \upsiqgdG{\varGamma_{\!\circ}} \, $  and inside  $ \, \upsiqgG{\varGamma} \, $)
   enjoy new ``twisted'' relations.  The striking fact is that in these presentations the ``twisted''
   relations happen to present a well precise form, commonly formalized via the notion of
   ``multiparameter quantum group'': we shall investigate this in detail in the forthcoming sections.
\end{free text}

\vskip11pt

\begin{free text}  \label{twist-deform's_x_TwQUEA's}
{\bf Comultiplication twistings for polynomial TwQUEA's.}
With notation as above, let  $ \, \Psi \in M_n(\QQ) \, $  be as in  \S \ref{twist-Uqg},
 and let  $ U_{q,\varGamma}^{\scriptscriptstyle \,\Psi}(\lieg) $  be the corresponding polynomial TwQUEA
 (for some lattice  $ \varGamma \, $)  over  $ \lieg'(A) \, $.  From its very construction, it follows that
the comultiplication twisting procedure
 can be iterated: namely, for  $ \, \Psi' \in M_n(\QQ) \, $  we can define the Hopf algebra
 $ {\big( U_{q,\varGamma}^{\scriptscriptstyle \,\Psi}(\lieg) \big)}^{\scriptscriptstyle \!\Psi'} $
 just reproducing the construction of  $ U_{q,\varGamma}^{\scriptscriptstyle \,\Psi}(\lieg) $
 but with  $ \Psi' $  replacing  $ \Psi $  and  $ U_{q,\varGamma}^{\scriptscriptstyle \,\Psi}(\lieg) $
 replacing  $ U_{q,\varGamma}(\lieg) \, $.
                                                                     \par
   {\sl It is then clear that
   $ \; {\big( U_{q,\varGamma}^{\scriptscriptstyle \,\Psi}(\lieg) \big)}^{\scriptscriptstyle \!\Psi'} \!\! =
   U_{q,\varGamma}^{\scriptscriptstyle \,\Psi+\Psi'}\!(\lieg) \, $},
   as this equality holds for the TwQUEA $ \upsihg \, $.  {\sl Therefore, all these new Hopf algebras will still
   be called TwQUEA's.}
 \vskip3pt
   The same construction is possible, and similar remarks apply, for TwQUEA's associated with
   $ \liegd'(A) \, $,  $ \lieb_+'(A) $ and  $ \lieb_- '(A)\, $.
\end{free text}

\bigskip

\section{Multiparameter quantum groups}
 We introduce now the multiparameter QUEA  $ U_\bq(\hskip0,8pt\liegd) \, $,  or
 {\sl MpQUEA\/}  for short, associated with a suitable matrix of parameters.
 Hereafter,  $ \FF $  will be a field of characteristic zero, and  $ \, \FF^\times := \FF \setminus \{0\} \, $.

\vskip13pt

\subsection{Defining multiparameter QUEA's (=MpQUEA's)}  \label{def-MpQUEA}  \
 \vskip7pt
   We fix a multiparameter matrix  $ \, \bq := \! {\big( q_{ij} \big)}_{i, j \in I} \in M_n\big(\FF\big) \, $,  with
   $ \, i, j \! \in \! I \, $,  $ \, i \! \not= \! j \, $,  and a generalized, symmetrisable Cartan matrix
   $ \, A := {\big( a_{ij} \big)}_{i, j \in I} \, $,  \,{\it with the additional compatibility assumption that  $ \, q_{ii}^k \not= 1 \, $}
   for  $ \, k = 1, \dots, \max(1\!-\!a_{ij}\,,2) \, $.  Moreover, when  $ \, \bq \, $  is of Cartan type
   we assume that its associated
   (generalized, symmetrisable) Cartan matrix   --- cf.\  \S \ref{mult-multiparameters}  ---   coincides with the matrix  $ A $
   mentioned above, and we fix scalars  $ \, q_i \in \FF \, $  ($ \, i \! \in I \, $)  as in  \S \ref{mult-multiparameters}.
   Finally, like in  \S \ref{tool-case},  we denote by  $ \, \lieg = \lieg(A) \, $
   the unique Kac-Moody algebra associated with $ A \, $,
by  $ \lieg' $  its derived subalgebra, etc.

\vskip11pt

\begin{definition}  \label{def:multiqgroup_ang}
 (cf.\ \cite{HPR})  We denote by  $ U_\bq(\hskip0,8pt\liegd) $  the unital associative algebra over  $ \FF $  with generators
 $ \, E_i \, , \, F_i \, , \, K_i^{\pm 1} \, , \, L^{\pm 1}_i \, $  (for all  $ \, i \in I \, $)  and relations
\begin{align*}
  (a\,)  \qquad  &  K_i^{\pm 1} L^{\pm 1}_j  \, = \,  L^{\pm 1}_j K^{\pm 1}_i \;\; ,
\qquad K_i^{\pm 1} K^{\mp 1}_i  \, = \,  1  \, = \,  L^{\pm 1}_i L^{\mp 1}_i  \\
  (b\,)  \qquad  &  K_i^{\pm 1} K^{\pm 1}_j  \, = \,  K^{\pm 1}_j K^{\pm 1}_i  \;\; ,
\qquad  L_i^{\pm 1} L_j^{\pm 1}  \, = \,  L_j^{\pm 1} L_i^{\pm 1}  \\
  (c\,)  \qquad  &  K_i \, E_j \, K_i^{-1}  \, = \,  q_{ij} \, E_j  \;\; ,  \qquad
L_i \, E_j \, L_i^{-1}  \, = \,  q_{ji}^{-1} \, E_j  \\
  (d\,)  \qquad  &  K_i \, F_j \, K_i^{-1}  \, = \,  q_{ij}^{-1} \, F_j  \;\; ,  \qquad
L_i \, F_j \, L_i^{-1}  \, = \,  q_{ji} \, F_j
\end{align*}
\begin{align*}
  (e\,)  \qquad  &  [E_i , F_j]  \, = \,  \delta_{i,j} \, q_{ii} \, \frac{\, K_i - L_i \,}{\, q_{ii} - 1 \,}  \\
  (f\,)  \qquad  &  \sum_{k=0}^{1-a_{ij}} (-1)^k \,
  {\bigg( {1-a_{ij} \atop k} \bigg)}_{\!\!q_{ii}} q_{ii}^{{k \choose 2}} \, q_{ij}^k \,
E_i ^{\,1-a_{ij}-k} E_j \, E_i^{\,k}  \; = \;  0   \;\; \qquad (\, i \neq j \,)  \\
  (g\,)  \qquad  &  \sum_{k=0}^{1-a_{ij}} (-1)^k \,
  {\bigg( {1-a_{ij} \atop k} \bigg)}_{\!\!q_{ii}} q_{ii}^{{k \choose 2}} \, q_{ij}^k \,
F_i^{\,k} F_j \, F_i^{\,1-a_{ij}-k}  \; = \;  0   \;\; \qquad (\, i \neq j \,)  \\
\end{align*}

   \indent   Moreover,  $ U_\bq(\hskip0,8pt\liegd) $  is a Hopf algebra with coproduct,
   counit and antipode determined for all  $ \, i, j \in I \, $  by
\begin{align*}
  \com(E_i) \,  &  = \,  E_i \otimes 1 + K_i \otimes E_i  \;\; ,  &
 \eps(E_i) \,  &  = \,  0  \;\; ,  &  \SS(E_i) \,  &  = \,  -K_i^{-1} E_i  \\
  \com(F_i) \,  &  = \,  F_i \otimes L_i + 1 \otimes F_i  \;\; ,  &
 \eps(F_i) \,  &  = \,  0  \;\; ,  &  \SS(F_i) \,  &  = \,  - F_i{L_i}^{-1}  \\
  \com\big(K_i^{\pm 1}\big)  &  = \,  K_i^{\pm 1} \otimes K_i^{\pm 1}  \;\;  ,  &
 \eps\big(K_i^{\pm 1}\big)  &  = \,  1  \;\; ,  &  \SS\big(K_i^{\pm 1}\big)  &  = \,  K_i^{\mp 1}  \\
  \com\big(L_i^{\pm 1}\big)  &  = \,  L_i^{\pm 1} \otimes L_i^{\pm 1}  \;\; ,  &
 \eps\big(L_i^{\pm 1}\big)  &  = \,  1  \;\; ,  &  \SS\big(L_i^{\pm 1}\big)  &  = \,  L_i^{\mp 1}
\end{align*}
 \vskip5pt
   Finally, for later use we introduce also,  for every  $ \, \lambda = \sum_{i \in I} \lambda_i \, \alpha_i \, \in \, Q \, $,
   the notation  $ \; K_\lambda := \prod_{i \in I} K_i^{\lambda_i} \; $  and  $ \; L_\lambda := \prod_{i \in I} {L_i}^{\lambda_i} \; $.
\end{definition}

\medskip

\begin{rmk}  \label{link_QEq-QE & symm-case}
 Assume that  $ \, q \in \FF^\times \, $  is not a root of unity and fix the  \hbox{$ \, q $--canonical}
 multiparameter  $ \check{\bq} := {\big(\, \check{q}_{ij} = q^{d_i a_{ij}} \big)}_{i, j \in I} \, $  like in
 (\ref{qij-canon})  above; then we can define the corresponding MpQUEA as above, now denoted  $ \QEqcheck \, $.
 The celebrated one-parameter quantum group  $ U_q(\lieg) $  defined by Jimbo and Lusztig is (up to a minimal
 change of generators in its presentation, irrelevant for what follows) nothing but the quotient of our  $ \QEqcheck $
 by the (Hopf) ideal generated by  $ \, \big\{ L_i - K_i^{-1} \,\big|\; i = 1, \dots, n \, \big\} \, $.
                                                                  \par
   As a matter of fact,  {\sl most constructions usually carried on for  $ \, U_q(\lieg) \, $  actually makes sense
   and apply the same to  $ \QEqcheck $  as well}.
\end{rmk}

\smallskip

   We introduce then the so-called ``quantum Borel / nilpotent / Cartan subalgebras'' of any MpQUEA, say
   $ U_\bq(\hskip0,8pt\liegd) \, $,  as follows:

\medskip

\begin{definition}  \label{def:q-Bor-sbgr}
 Given  $ \, \bq := {\big(\, q_{ij} \big)}_{i,j \in I} \, $  and  $ \, U_\bq(\hskip0,8pt\liegd) \, $  as in  \S \ref{def-MpQUEA},
 we define  $ \, U_\bq^{\,0} := U_\bq(\liehd) \, $,  $ \, U_\bq^{+,0} \, $,  $ \, U_\bq^{-,0} \, $,
 $ \; U_\bq^- := U_\bq(\lien_-) \, $,  $ \; U_\bq^+ := U_\bq(\lien_+) \, $,  $ \; U_\bq^\leq :=  U_\bq(\lieb_-) \; $  and
 $ \; U_\bq^\geq := U_\bq(\lieb_+) \; $  to be the  $ \k $--subalgebras  of  $ \QEq $  respectively generated as
  $$  \displaylines{
   U_\bq^{\,0}  \; := \;  \Big\langle\, K_i^{\pm 1} \, , \, L_i^{\pm 1} \,\Big\rangle_{i \in I}  \quad ,
\qquad   U_\bq^{+,0}  \; := \;  \Big\langle\, K_i^{\pm 1} \,\Big\rangle_{i \in I}  \quad ,
\qquad   U_\bq^{-,0}  \; := \;  \Big\langle\, L_i^{\pm 1} \,\Big\rangle_{i \in I}  \cr
%
%
   U_\bq^-  := \,  \big\langle F_i \big\rangle_{\! i \in I}  \; ,
\quad   U_\bq^\leq  := \,  \Big\langle F_i \, , \, L_i^{\pm 1} \Big\rangle_{\! i \in I}  \; ,
\quad   U_\bq^\geq  := \,  \Big\langle K_i^{\pm 1} \, , \, E_i \Big\rangle_{\! i \in I}  \; ,
\quad   U_\bq^+  := \,  \big\langle E_i \big\rangle_{\! i \in I}  }  $$
   \indent   We shall refer to  $ U_\bq^\leq $  and  $ U_\bq^\geq $
   as to the  {\sl positive\/}  and  {\sl negative\/}  multiparameter quantum Borel
   (sub)algebras, and  $ U_\bq^{\,0} \, $,  $ U_\bq^{+,0} $  and  $ U_\bq^{-,0} $
   as to the  {\sl global},  {\sl positive\/}  and  {\sl negative\/}  multiparameter Cartan (sub)algebras.
                                                                       \par
   For later use, we also define  $ \G_n $  to be the free Abelian group of rank  $ \, n := |I| \, $,
   that we will write in multiplicative notation, and we denote by  $ \, U_\bq(\lieh) \, $  the group
   algebra of  $ \G_n $  over  $ {\mathbb F} $   --- which, in spite of notation, is independent of  $ \bq \, $.
   Letting  $ {\{G_i\}}_{i \in I} $  be a  $ \ZZ $--basis  of  $ \G_n \, $,  we have natural Hopf algebra
   isomorphisms of  $ U_\bq(\lieh) $  with  $ U_\bq^{+,0} $  and with  $ U_\bq^{-,0} $  given by
   $ \, G_i^{\pm 1} \mapsto K_i^{\pm 1} \, $  and  $ \, G_i^{\pm 1} \mapsto L_i^{\pm 1} \, $,  respectively.
\end{definition}

\vskip5pt

\begin{rmks}  \label{rmks:mpqg's-vs-Nichols}  {\ }
 \vskip3pt
   {\it (a)}\;  Let  $ \, \bq = {\big( q_{ij} \big)}_{1\leq i,j \leq n} \, $  be a multiparameter matrix of Cartan type.
   Following  \cite[Section 4]{He},  we define the following braided vector spaces:
 \vskip3pt
  {\it (i)} \quad  $ V_E $ with  $ \FF $--basis  $\{E_1, \dots, E_n\}$ and
braiding given by
  $$  c(E_{i}\ot E_{j}) \, := \, q_{ij} \, E_j \ot E_i   \qquad \text{ for all } 1 \leq i,j \leq n \; ,  $$
 \vskip3pt
  {\it (ii)} \quad  $ V_F $ with  $ \FF $--basis  $ \{F'_1=L^{-1}_{1}F_{1}, \dots, F'_n=L_{n}^{-1}F_{n}\} $ and
braiding given by
  $$  c(F'_{i}\ot F'_{j}) \, := \, q_{ji} \, F'_j \ot F'_i   \qquad \text{ for all } 1 \leq i,j \leq n \; .  $$
   \indent    Then we have the corresponding Nichols algebras   --- of diagonal type ---
   $ \B\big(V_E\big) $  and  $ \B\big(V_F\big) $,  as well as their bosonizations over the group
   $ \, \varGamma := \ZZ^n \, $:
  $$  \Hc\big(V_E\big) \, := \, \B\big(V_E\big) \,\#\, \FF[\varGamma]   \qquad \text{and} \qquad
      \Hc\big(V_F\big) \, := \, \B\big(V_F\big) \,\#\, \FF[\varGamma]  $$
                                              \par
   Directly from definitions, one has canonical identifications  $ \, U_\bq^- \cong \B\big(V_F\big) \, $  and
   $ \, U_\bq^+ \cong \B\big(V_E\big) \, $  (as  $ U_\bq^{\pm,0} $--comodule  algebras), and
   $ \, U_\bq^\leq \cong \Hc\big(V_F\big) \, $  and  $ \, U_\bq^\geq \cong \Hc\big(V_E\big) \, $
   as Hopf algebras,  $ \, U_\bq^0 \cong \FF[\varGamma \times \varGamma] \, $  as Hopf algebras, etc.
   For more details on the relation with Nichols algebras and bosonization see  \cite{An}, \cite{Gar},  \cite{He}.
 \vskip3pt
   {\it (b)}\; It is also known, see for example  \cite{AA},  that the multiparameter quantum group  $ U_\bq(\hskip0,8pt\liegd) $
   can be realized as a Drinfeld double of  $ \, U_\bq^\leq \cong \Hc\big(V_F\big) \, $  and
   $ \, U_\bq^\geq \cong \Hc\big(V_E\big) \, $
   using the Hopf pairing given in  Proposition \ref{sk-H_pair}  below.  Thus, in the end,  $ U_\bq(\hskip0,8pt\liegd) $
   is a Drinfeld double of bosonizations of Nichols algebras of diagonal type.
 \vskip3pt
%
%
%
\end{rmks}

\vskip5pt

   From  \cite[Proposition 4.3]{He}   --- see also  \cite[Theorem 20]{HPR}  and  \cite[Proposition 2.4]{AY}   ---
   we recall the following:

\vskip9pt

\begin{prop}  \label{sk-H_pair}  {\ }
 \vskip1pt
   With the assumptions above, assume in addition that  $ \, q_{ii} \not= 1 \, $  for all indices $ \, i \in I \, $.
   Then there exists a unique skew-Hopf pairing
 $ \; \eta : U_\bq^\geq \mathop{\otimes}\limits_\FF \, U_\bq^{\leq} \relbar\joinrel\longrightarrow \FF \; $  such that
 \vskip-11pt
  $$  \eta(K_i,L_j) \, = \, q_{ij} \;\; ,  \qquad  \eta(E_i,F_j) \, = \, \delta_{i,j} \, {{\, - \, q_{ii} \,} \over {\, q_{ii} - 1 \,}} \;\; ,
  \qquad  \eta(E_i,L_j) \, = \, 0 \, = \, \eta(K_i,F_j)  $$
for all  $ \, 1 \leq i, j \leq n \, $.  It enjoys the following property: for every  $ \, E \in U_\bq^+ \, $,  $ \, F \in U_\bq^-  \, $,
and every Laurent monomials  $ K $  in the  $ K_i $'s  and  $ L $  in the  $ L_i $'s,  we have
  $$  \eta\big( E \, K , F \, L \big)  \; = \;  \eta(E\,,F) \, \eta(K,L)   \eqno  \square  $$
\end{prop}

\vskip9pt

   For later use, we also single out the following result, whose proof is trivial:

\vskip9pt

\begin{lema}  \label{lemma: proj_Uqgd->Uqh}
 Keep notation as above.  There exists a unique Hopf\/  $ {\mathbb F} $--algebra
 epimorphism  $ \; p : U_\bq(\hskip0,8pt\liegd) \relbar\joinrel\relbar\joinrel\twoheadrightarrow U_\bq(\lieh) \; $
 given by
  $$  p(E_i) \, = \, 0 \, = \, p(F_i) \quad ,  \qquad  p\big(K_i^{\pm 1}\big) \, = \, G_i^{\pm 1} \,
  = \, p\big(L_i^{\pm 1}\big)   \eqno  \forall \;\; i \in I \; .  \quad  \qed$$

\end{lema}

\vskip15pt

 \subsection{Multiplication twistings of MpQUEA's}  \label{coc-deform-MpQUEA}  {\ }
 \vskip7pt
   We want to perform on the Hopf algebras  $ U_\bq(\hskip0,8pt\liegd) $  a
multiplication twisting
 process, via special types of  $ 2 $--cocycles,  like in  \S \ref{cocyc-defs},
 following  \cite{AST},  \cite{DT}  and  \cite{Mo}.

\vskip5pt

\begin{free text}  \label{characters --> 2-cocycles}
 {\bf Special  $ 2 $--cocycles  of  $ U_\bq(\hskip0,8pt\liegd) \, $.}
 Let  $ \, \bq :=  {\big(\, q_{ij} \big)}_{i,j \in I} \in M_n(\FF) \, $  be a multiparameter, and let
 $ \, \chi : \G_n \times \G_n \!\relbar\joinrel\relbar\joinrel\longrightarrow \FF^\times \, $
 be any bicharacter of the Abelian group  $ \G_n \, $.  Then  $ \chi $  is automatically a
 $ 2 $--cocycle  of  $ \G_n \, $;  as such, it induces (by  $ \FF $--linear  extension) a normalized
 $ 2 $--cocycle  of the group algebra of  $ \G_n $  over  $ \FF $,  that is  $ U_\bq(\lieh) \, $,
 cf.\ Definition \ref{def:q-Bor-sbgr}:  we denote this  $ 2 $--cocycle  of  $ U_\bq(\lieh) $  by  $ \chi $  again.
 Now, composing  $ \, \chi : U_\bq(\lieh) \otimes U_\bq(\lieh) \!\relbar\joinrel\relbar\joinrel\longrightarrow \FF \, $
 with  $ p^{\otimes 2} $   --- where
 $ \; p : U_\bq(\hskip0,8pt\liegd) \relbar\joinrel\relbar\joinrel\twoheadrightarrow U_\bq(\lieh) \; $  is the Hopf
 $ {\mathbb F} $--algebra  epimorphism in  Lemma \ref{lemma: proj_Uqgd->Uqh}  ---
 we get a (normalized) bimultiplicative Hopf  $ 2 $--cocycle  $ \, \sigma_\chi := \chi \circ p^{\otimes 2} \, $  of
 $ U_\bq(\hskip0,8pt\liegd) \, $.  This leads us to the following definition:

\vskip5pt

\begin{definition}\label{def-sigma}
 We denote by  $ \, \widetilde{\Z}_2\big( U_\bq(\hskip0,8pt\liegd) \, , \FF \,\big) \, $
 the set of all (normalized, bimultiplicative)
 Hopf  $ 2 $--cocycles  of  $ U_\bq(\hskip0,8pt\liegd) $  of the form  $ \, \sigma_\chi \, $
 (as defined above) for some bicharacter
 $ \chi $  of the Abelian group  $ \G_n \, $.
 We call these ``toral''  $ 2 $--cocycles  of  $ U_\bq(\hskip0,8pt\liegd) \, $.
                                                       \par
   Note, in particular, that
%
%
 $ \, \widetilde{\Z}_2\big( U_\bq(\hskip0,8pt\liegd) \, , \FF \,\big) \, $  is
%
%
 independent of the multiparameter  $ \bq \, $.
\end{definition}

\vskip5pt

   By construction, for every  $ \, \sigma_\chi \in \widetilde{\Z}_{2}\big( U_\bq(\hskip0,8pt\liegd) \, , \FF \,\big) \, $  one has
  $$  \displaylines{
   \sigma_\chi(E_i\,,Y\,) \, = \, 0 \quad ,   \qquad  \sigma_\chi(X,F_i) \, = \, 0   \qquad  \forall \;\;
   Y , X \in U_\bq(\hskip0,8pt\liegd) \; , \;\; i, j \in I  \cr
   \sigma_\chi\big(Z'_i\,,Z''_j\big) \, = \, \chi(G_i\,,G_j)   \qquad
   \forall \;\; Z'_{i} , Z''_{j} \in \{K_{i},L_{i}\}_{i\in I} \; , \;\; i, j \in I  }  $$
 in particular, the  $ 2 $--cocycle  $ \sigma_\chi $  is determined uniquely by the values of  $ \chi $  on the Abelian group
 $ \G_n \, $.  Conversely, each (normalized) bimultiplicative Hopf  $ 2 $--cocycle  $ \sigma $  of
 $ U_\bq(\hskip0,8pt\liegd) $  such that
\begin{equation}  \label{propts-cocyc_from-char}
  \begin{aligned}
     \sigma(E_i\,,Y\,) \, = \, 0 \quad ,   \qquad  \sigma(X,F_i) \, = \, 0   \qquad  \forall \;\;
     Y , X \in U_\bq(\hskip0,8pt\liegd) \; , \;\; i, j \in I  \\
     \sigma(K_i\,,K_j) \, = \, \sigma(K_i\,,L_j) \, = \, \sigma(L_i\,,K_j) \, = \, \sigma(L_i\,,L_j)   \qquad  \forall \;\; i, j \in I  \;\;
  \end{aligned}
\end{equation}
is necessarily of the form  $ \, \sigma = \sigma_\chi \, $,  as it is uniquely determined by its values on the torus, hence
$ \sigma $  belongs to  $ \, \widetilde{\Z}_2\big( U_\bq(\hskip0,8pt\liegd) \, , \FF \,\big) \, $.
The corresponding bicharacter  $ \chi $  is then defined by the conditions (for all  $ \, i, j \in I \, $)
  $$  \chi(G_i\,,G_j) \, = \, \sigma(K_i\,,K_j)  \quad \big(\, = \, \sigma(K_i\,,L_j) \, = \, \sigma(L_i\,,K_j) \,
  = \, \sigma(L_i\,,L_j) \,\big)  $$
Therefore,  \eqref{propts-cocyc_from-char}  are the conditions that characterize intrinsically
those (normalized, bimultiplicative) Hopf  $ 2 $--cocycles of  $ U_\bq(\hskip0,8pt\liegd) $
that belong to  $ \, \widetilde{\Z}_{2}\big( U_\bq(\hskip0,8pt\liegd) \, , \FF \,\big) \, $.

\end{free text}

\vskip7pt

\begin{free text}  \label{deforming-MpQUEAs}
 {\bf Deforming MpQUEAs via  $ 2 $--cocycles.}  For any multiparameter
 $ \, \bq := {\big(\, q_{ij} \big)}_{i,j \in I} \in M_n(\FF) \, $  we fix the notation
 $ \, K_\lambda := \prod_{i \in I} K_i^{\lambda_i} \, $  and  $ \, L_\lambda := \prod_{i \in I} {L_i}^{\lambda_i} \, $
 for every  $ \, \lambda = \sum_{i \in I} \lambda_i \alpha_i \, \in \, Q \, $;  similarly, we shall also write
 $ \; q_{\mu\,\nu} \, := \! {\textstyle \prod\limits_{i,j \in I}} \, q_{ij}^{\, \mu_i \nu_j} \; $ for each
 $\mu=\sum_{i\in I} \mu_{i}\alpha_{i}$ and $\nu=\sum_{i\in I} \nu_{i}\alpha_{i}$ in $Q$.
                                                          \par
   If  $ \bq $  is of Cartan type, we fix a special element
   $ \, q:= q_{j_{\raise-2pt\hbox{$ \scriptscriptstyle 0 $}}} \in \FF^\times \, $
   as explained in  \S \ref{mult-multiparameters}  (so that  $ \, q_{ii} = q^{\,2\,d_i} \, $  for all  $ \, i \in I \, $);
   accordingly, we have also a well-defined multiparameter of canonical type, namely
   $ \, \check{\bq} := {\big(\, q^{\,d_i{}a_{ij}} \big)}_{i,j \in I} \, $.
                                                          \par
   Finally,  $ U_\bq(\hskip0,8pt\liegd) $  will be the MpQUEA associated with  $ \bq $  as in
   Definition \ref{def:multiqgroup_ang};  similarly, we have also  $ U_{\check{\bq}}(\hskip0,8pt\liegd) $
   when  $ \bq $  is of Cartan type.

\medskip

   The  {\sl key result\/}  that we shall rely  upon in the sequel is the following:

\medskip

\begin{theorem}  \label{thm:sigma_2-cocy}
 (cf.\ \cite[Theorem 4.5]{LHR}  and  \cite[Theorem 28]{HPR})
 \vskip3pt
   (a)\,  If any two multiparameters are\/  $ \sim $--equivalent  (notation as in  \S \ref{equiv-act_x_mprmts}),
   then their associated MpQUEA's are multiplicative twistings
 of each other.
 In detail, if  $ \, \bq' := {\big(\, q'_{ij} \,\big)}_{i, j \in I} \, , \bq'' := {\big(\, q''_{ij} \,\big)}_{i, j \in I} \in M_n(\FF) \, $
 are multiparameters such that  $ \, \bq' \!\sim \bq'' \, $,  then there exists
 $ \; \sigma \in \widetilde{\Z}_2\big(\, U_{\bq'}(\hskip0,8pt\liegd) \, , \FF \,\big) \; $
 such that  $ \,\; U_{\bq''}(\hskip0,8pt\liegd) \cong {\big( U_{\bq'}(\hskip0,8pt\liegd) \big)}_\sigma \;\, $.
 \vskip3pt
   (b)\,  Any MpQUEA with multiparameter of Cartan type is a
multiplicative twisting
 of a MpQUEA with canonical multiparameter.  In detail, if  $ \, \bq := {\big(\, q_{ij} \big)}_{i, j \in I} \in M_n(\FF) \, $
 is a multiparameter of Cartan type, with associated  $ \, q \in \FF^\times \, $  and Cartan matrix
 $ \, A = {\big( a_{ij} \big)}_{i, j \in I} \, $,  and corresponding multiparameter of canonical type
 $ \, \check{\bq} := {\big(\, q^{\,d_i{}a_{ij}} \big)}_{i, j \in I} \, $,  then there exists
 $ \; \sigma \in \widetilde{\Z}_{2}\big(\, U_{\check{\bq}}(\hskip0,8pt\liegd) \, , \FF \,\big) \; $  such that
 $ \,\; U_\bq(\hskip0,8pt\liegd) \cong {\big( U_{\check{\bq}}(\hskip0,8pt\liegd) \big)}_\sigma \;\, $.
 \vskip3pt
   (c)\,  Similar results hold true for quantum Borel (sub)MpQUEA's as well.
\end{theorem}

\pf
 For claim  {\it (a)},  by  Lemma \ref{lemma: mprmtr-action=>twist-equiv}{\it (a)\/}  there exists
 $ \, \bnu = {\big( \nu_{ij} \big)}_{i, j \in I} \in M_n(\FF) \, $  such that  $ \, \bq'' = \bnu.\,\bq' \, $,
 that is, $ \, q_{ij}'' = \nu_{ij} \, q_{ij}' \, \nu_{ji}^{-1} \, $  for all  $ \, i , j \in I \, $.
 Clearly there exists a unique bicharacter  $ \chi_\bnu $  of the Abelian group  $ \G_n $
 defined by  $ \, \chi_\bnu(G_i\,,G_j) := \nu_{ij} \, $  (for  $ \, i, j \in I \, $): then, like in
 \S \ref{characters --> 2-cocycles},  this  $ \chi_\bnu $  defines a unique
 $ \, \sigma= \sigma_{\chi_\bnu} \in \widetilde{\Z}_{2}\big(\, U_{\bq'}(\hskip0,8pt\liegd) \, , \FF \,\big) \, $.
 Through a rather straightforward calculation   --- much like in  \cite{LHR}  and  \cite{HPR}  ---
 one may see that  $ \; U_{\bq''}(\hskip0,8pt\liegd) \cong {\big( U_{\bq'}(\hskip0,8pt\liegd) \big)}_\sigma \; $,
 \, as claimed; for the sake of completeness, we include below most of the computations.
 As in  \S \ref{cocyc-defs},  we write  $ \, a \cdot_{\sigma} b := m_{\sigma}(a,b) \, $.
 \vskip5pt
%
  First, we show that the generators  $ \, E_i \, , \, F_i \, , \, K_i^{\pm 1} \, , \, L^{\pm 1}_i \, $
  (for all  $ \, i \in I \, $)  of  $ U_{\bq'}(\hskip0,8pt\liegd) $  do satisfy the relations that define
  $ U_{\bq''}(\hskip0,8pt\liegd) $  when computed with this new product.
                                                                                \par
   In general, the product of any two group-like elements remains unchanged under a multiplicative twist by a
   $ 2 $--cocycle,  for
  $$  g \cdot_{\sigma} h  \, = \,  \sigma(g,h) \, g \, h \, \sigma^{-1}(g,h) \, = \,
  \eps(g) \, \eps(h) \, g \, h \, = \, g \, h   \qquad  \text{\ for all\ } g, h \in G\big(U_{\bq'}(\hskip0,8pt\liegd)\big) \;\; .  $$
 Since the elements  $ \, K_i^{\pm 1} \, , \, L^{\pm 1}_i \, $  are group-like, for all  $ \, i \in I \, $,  we have
\begin{align*}
  (a\,)  \qquad  &  \  K_i^{\pm 1}\cdot_{\sigma} L^{\pm 1}_j  \, = \,
  L^{\pm 1}_j \cdot_{\sigma} K^{\pm 1}_i  \quad ,   \qquad  K_i^{\pm 1} \cdot_{\sigma} K^{\mp 1}_i  \, = \,  1  \, = \,
  L^{\pm 1}_i \cdot_{\sigma} L^{\mp 1}_i  \\
  (b\,)  \qquad  &  K_i^{\pm 1} \cdot_{\sigma} K^{\pm 1}_j  \, = \,
  K^{\pm 1}_j \cdot_{\sigma} K^{\pm 1}_i  \quad ,   \qquad \ \ \   L_i^{\pm 1} \cdot_{\sigma} L_j^{\pm 1}  \, = \,
  L_j^{\pm 1}  \cdot_{\sigma} L_i^{\pm 1}  \\
\end{align*}
   \indent   In order to compute the remaining relations, one has to keep in mind formulas involving
   the coproduct and powers of the generators, for instance
\begin{small}
  \begin{align*}
     \com\big(E_i^k\big) &  \, = \,  \sum_{\ell=0}^k {k \choose \ell}_{\!\!q_{ii}'} \! E_i^{k-\ell} K_i^\ell \otimes E_i^\ell  \quad ,
     \qquad   \com\big(F_i^k\big)  \, = \,  \sum_{\ell=0}^k {k \choose \ell}_{\!\!q_{ii}'} \! F_i^{k-\ell} \otimes F_i^\ell \, L_i^{k-\ell}  \\
     \com^{(2)}\big(E_i^k\big)  &  \, = \,
     \sum_{\ell=0}^k \sum_{j=0}^\ell {k \choose \ell}_{\!\!q_{ii}'} {\ell \choose j}_{\!\!q_{ii}'} E_i^{k-\ell} K_i^\ell \otimes E_i^{\ell -j} K_i^j \otimes E_i^j  \\
     \com^{(2)}\big(F_i^k\big)  &  \, = \,  \sum_{\ell=0}^k \sum_{j=0}^\ell {k \choose \ell}_{\!\!q_{ii}'} {\ell \choose j}_{\!\!q_{ii}'} F_i^{k-\ell-j} \otimes F_i^j L_i^{k-\ell-j} \otimes F_i^\ell \, L_i^{k-\ell}  \\
     \com^{(2)}\big(E_i^k E_s\big)  &  \, = \,  \sum_{\ell=0}^k \sum_{j=0}^\ell {k \choose \ell}_{\!\!q_{ii}'} {\ell \choose j}_{\!\!q_{ii}'} \bigg( E_i^{k-\ell} \, K_i^\ell \, E_s \otimes E_i^{\ell-j} K_i^j \otimes E_i^j \; + \\
        &  \qquad   + \; E_i^{k-\ell} \, K_i^\ell \, K_s \otimes E_i^{\ell-j} K_i^j E_s \otimes E_i^j \, + \, E_i^{k-\ell} \, K_i^\ell \, K_s \otimes E_i^{\ell-j} K_i^j K_s \otimes E_i^j E_s \bigg)
   \end{align*}
\end{small}
\hskip-3,5pt
 for all  $ \, i , s \in I \, $  and all  $ \, k \geq 0 \, $.
 Here  $ \, \com^{(2)} := (\com \otimes \id) \circ \com = (\id \otimes \com) \circ \com \, $.
 So, for the relations in  $ (c\,) $  we have
\begin{align*}
   K_i \,\cdot_\sigma E_j  &  \; = \;  \sigma(K_i,E_j) \, K_i \, \sigma^{-1}(K_i,1) \, + \,
   \sigma(K_i,K_j) \, K_i \, E_j \, \sigma^{-1}(K_i,1) \, +  \\
                           &  \quad   + \, \sigma(K_i,K_j) \, K_i \, K_j \, \sigma^{-1}(K_i,E_j)  \; = \;
                           \sigma(K_i,K_j) \, K_i \, E_j  \; = \;  \nu_{ij} \, K_i \, E_j  \\
   \big(K_i \, E_j\big) \,\cdot_\sigma K_i^{-1}  &  \; = \;
   \sigma\big( K_i \, E_j , K_i^{-1} \big) \, K_i K_i^{-1} \, \sigma^{-1}\big(K_i,K_i^{-1}\big) \, +  \\
                                                 &  \quad   + \, \sigma\big(K_i \, K_j , K_i^{-1} \big) \, K_i \, E_j \, K_i^{-1} \, \sigma^{-1}\big( K_i , K_i^{-1} \big) \, +  \\
                                                 &  \quad   + \, \sigma\big( K_i \, K_j , K_i^{-1} \big) \, K_i \, K_j \, K_i^{-1} \, \sigma^{-1}\big( K_i \, E_j , K_i^{-1} \big)  \; =  \\
                                                 &  = \;  \sigma\big( K_i \, K_j , K_i^{-1} \big) \, K_i \, E_j \, K_i^{-1} \, \sigma^{-1}\big( K_i , K_i^{-1} \big)  \; =  \\
                                                 &  = \;  \sigma\big( K_j , K_i^{-1} \big) \, K_i \, E_j \, K_i^{-1}  \; = \;  \nu_{ji}^{-1} \, K_i \, E_j \, K_i^{-1}  \; = \; \nu_{ji}^{-1} \, q'_{ij} \, E_j
\end{align*}
 for all  $ \, i , j \in I \, $.  Therefore,
  $$  K_i \,\cdot_\sigma E_j \,\cdot_\sigma K_i^{-1}  \; = \;  \nu_{ij} \, \big( K_i \, E_j \big) \cdot_\sigma K_i^{-1}  \;
  = \;  \nu_{ij} \, q'_{ij} \, \nu_{ji}^{-1} \, E_j  \; = \;  q''_{ij} \, E_j  $$
 Similarly, we have
\begin{align*}
   L_i \,\cdot_\sigma E_j \,\cdot_\sigma L_i^{-1}  &  \; = \;  \big( L_i \,\cdot_\sigma E_j \big) \,\cdot_\sigma L_i^{-1}  \; =
   \;  \nu_{ij} \, \big( L_i \, E_j \big) \cdot L_i^{-1}  \; = \;  \nu_{ij} \, \nu_{ji}^{-1} \, L_i \, E_j \, L_i^{-1}  \; =  \\
                                                   &  \; = \;  \nu_{ij} \, {\big( q'_{ji} \big)}^{-1} \, \nu_{ji}^{-1} \, E_j  \; = \;
                                                   {\big( \nu_{ji} \, q'_{ji} \, \nu_{ij}^{-1} \big)}^{-1} \, E_j  \; = \; {(q''_{ji})}^{-1} \, E_j
\end{align*}
 The computation for relation  $ (d\,) $  is entirely similar: for all  $ \, i, j \in I \, $  we have
\begin{align*}
   K_i \,\cdot_\sigma F_j  &  \; = \;  \sigma(K_i,F_j) \, K_i \, L_j \, \sigma^{-1}(K_i,L_j) \, + \,
   \sigma(K_i,1) \, K_i \, F_j \, \sigma^{-1}(K_i,L_j) \, +  \\
                           &  \quad   + \, \sigma(K_i,1) \, K_i \, \sigma^{-1}(K_i,F_j)  \; = \;  \nu_{ij}^{-1} \, K_i \, F_j  \\
   \big( K_i \, F_j \big) \,\cdot_\sigma K_i^{-1}  &  = \;  \sigma\big( K_i \, F_j , K_i^{-1} \big) \, K_i \, L_j \, K_i^{-1} \, \sigma^{-1}\big( K_i \, L_j , K_i^{-1} \big) \, +  \\
                                                   &  \quad   + \, \sigma\big( K_i , K_i^{-1} \big) \, K_i \, F_j \, K_i^{-1} \, \sigma^{-1}\big( K_i \, L_j , K_i^{-1} \big) \, +  \\
                                                   &  \quad   + \, \sigma\big( K_i , K_i^{-1} \big) \, K_i \, K_i^{-1} \, \sigma^{-1}\big( K_i \, F_j , K_i^{-1} \big)  \\
                                                   &  = \;  \sigma\big( K_i , K_i^{-1} \big) \, K_i \, F_j \, K_i^{-1} \, \sigma^{-1}\big( K_i \, L_j , K_i^{-1} \big)  \; =  \\
                                                   &  = \;  \sigma\big( L_j , K_i \big) \, K_i \, F_j \, K_i^{-1}  \; = \;  \nu_{ji} \, K_i \, F_j \, K_i^{-1}  \; = \; \nu_{ji} \, {\big( q'_{ij} \big)}^{-1} \, F_j
\end{align*}
 whence, we obtain
  $$  K_i \,\cdot_\sigma F_j \,\cdot_\sigma K_i^{-1}  =  \nu_{ij}^{-1} \big( K_i F_j \big) \cdot_\sigma K_i^{-1}
  =  \nu_{ij}^{-1} {\big(q'_{ij}\big)}^{-1} \nu_{ji} \, F_j  =  {\big( \nu_{ij} q'_{ij} \nu_{ji}^{-1} \big)}^{-1} F_j  =
  {\big(q''_{ij}\big)}^{-1} F_j
$$
 Also, we have that  $ \; L_i \,\cdot_\sigma F_j \,\cdot_\sigma L_i^{-1} \, = \, q''_{ji} \, F_j \; $.
                                                              \par
   As to relation  $ (e\,) \, $,  it is enough to note that
   $ \; E_i \,\cdot_\sigma F_j \, = \, E_i \, F_j \, $,  $ \; F_j \,\cdot_\sigma E_i \, = \, F_j \, E_i \; $  and
   $ \, q'_{ii} = q''_{ii} \, $  for all  $ \, i, j \in I \, $.  Thus, for instance,
\begin{align*}
   E_i \,\cdot_\sigma F_j  &  \; = \;  \sigma(E_i , F_j) \, L_j \, \sigma^{-1}(1,L_j) \, + \,
   \sigma(K_i,F_j) \, E_i \, L_j \, \sigma^{-1}(1,L_j) \, +  \\
                           &  \quad   + \, \sigma(K_i,F_j) \, K_i \, L_j \, \sigma^{-1}(E_i,L_j) \, + \,
                           \sigma(E_i,1) \, F_j \, \sigma^{-1}(1,L_j) \, +  \\
                           &  \quad   + \, \sigma(K_i,1) \, E_i \, F_j \, \sigma^{-1}(1,L_j) \, + \,
                           \sigma(K_i,1) \, K_i \, F_j \, \sigma^{-1}(E_i,L_j) \, +  \\
                           &  \quad   + \, \sigma(E_i,1) \, \sigma^{-1}(1,F_j) \, + \, \sigma(K_i,1) \, E_i \, \sigma^{-1}(1,F_j) \, +  \\
                           &  \quad   + \, \sigma(K_i,1) \, K_i \, \sigma^{-1}(E_i,F_j)  \; = \;  E_i \, F_j \\
\end{align*}
hence  $ \; E_i \,\cdot_\sigma F_j - F_j \,\cdot_\sigma E_i \, = \, [E_i,F_j] \, = \, \delta_{i,j} \, q'_{ii} \,
\frac{\, K_i - L_i \,}{\, q'_{ii} - 1 \,} \, = \, \delta_{i,j} \, q''_{ii} \, \frac{\, K_i - L_i \,}{\, q''_{ii} - 1 \,} \; $  for all  $ \, i , j \in I \, $.
                                                                  \par
   Finally, we check that the assertion holds for the quantum Serre relations; we prove only  $ (f\,) $,
   since  $ (\,g) $  is completely analogous.
   We have to verify the equality
\begin{align*}
  (f\,)  \qquad  &  \sum_{k=0}^{1-a_{ij}} {(-1)}^k \, {\bigg( {1-a_{ij} \atop k} \bigg)}_{\!\!q''_{ii}} {(q''_{ii})}^{k \choose 2}
  \, {(q''_{ij})}^k \, E_i^{\,\cdot_{\sigma}(1-a_{ij}-k)} \cdot_\sigma E_j \,\cdot_\sigma E_i^{\,\cdot_{\sigma} k}  \; = \;
  0  \quad  (\, i \neq j \,)
\end{align*}
where  $ \, E_i^{\,\cdot_\sigma k} \, $  denotes the  $ k $--th  power with respect to the twisted product.
A routine check yields that  $ \; E_i^{\,\cdot_\sigma k} \, = \, \sigma(K_i,K_i) \cdots \sigma\big( K_i^{k-1} , K_i \big) \,
E_i^k \, = \, \nu_{ii}^{{k-1}\choose 2} E_i \; $
for all  $ \, k \geq 0 \, $.  Thus for all $ \, i \neq j \, $
the left-hand side of  $ (f\,) $  above, that we denote by  $ \clubsuit \, $,  now reads
\begin{small}
   \begin{align*}
      &  \clubsuit  \,\; = \;  \sum_{k=0}^{1-a_{ij}} {(-1)}^k \, {\bigg( {1-a_{ij} \atop k} \bigg)}_{\!\!q''_{ii}} {(q''_{ii})}^{k \choose 2}
      \, {(q''_{ij})}^k \, E_i^{\,\cdot_{\sigma}(1-a_{ij}-k)} \cdot_\sigma E_j \,\cdot_\sigma E_i^{\,\cdot_{\sigma} k}  \,\; =  \\
      &  = \;  \sum_{k=0}^{1-a_{ij}} {(-1)}^k \, {\bigg( {1-a_{ij} \atop k} \bigg)}_{\!\!q'_{ii}} {(q'_{ii})}^{k \choose 2} \,
      \nu_{ij}^k \, {(q'_{ij})}^k \, \nu_{ji}^{-k} \, \nu_{ii}^{{-a_{ij}-k} \choose 2} \nu_{ii}^{{k-1} \choose 2}
      \, E_i^{\, (1-a_{ij}-k)} \cdot_\sigma E_j \,\cdot_\sigma E_i^{\,k}  \,\; =  \\
      &  = \;  \nu_{ii}^{\frac{a_{ij}(a_{ij}-1)}{2}} \sum_{k=0}^{1-a_{ij}} \! {(-1)}^k \,
      {\bigg( {1\!-\!a_{ij} \atop k} \bigg)}_{\!\!q'_{ii}} {(q'_{ii})}^{k \choose 2} \, {(q'_{ij})}^k \, \nu_{ij}^k \, \nu_{ji}^{-k} \,
      \nu_{ii}^{a_{ij} k + k^2 - k} \, E_i^{\, (1-a_{ij}-k)} \cdot_\sigma E_j \,\cdot_\sigma E_i^{\,k}
   \end{align*}
\end{small}
 But since  $ \; E_i^{\,(1-a_{ij}-k)} \cdot_\sigma E_j \, = \, \nu_{ij}^{(1-a_{ij}-k)} \, E_i^{(1-a_{ij}-k)} \, E_j \; $  and
\begin{small}
   \begin{align*}
      &  \big( E_i^{\,(1-a_{ij}-k)} \, E_j \big) \,\cdot_\sigma E_i^{\,k}  \; =  \\
      &  \; =  \sum_{\ell=0}^{1-a_{ij}-k} \! \sum_{r=0}^\ell \sum_{s=0}^k \sum_{t=0}^s
      \sigma\big( E_i^{(1-a_{ij}-k)-\ell} \, K_i^\ell \, K_j \, , E_i^{k-s} K_i^s \big) \, E_i^{\ell-r\,} K_i^r \, E_j \, E_i^{s-t\,} K_i^t \,
      \sigma^{-1}\!\big( E_i^r , E_i^t \big)  \, =  \\
      &  \quad   = \;  \sigma\big( K_i^{(1-a_{ij}-k)} \, K_j \, , K_i^k \big) \, E_i^{(1-a_{ij}-k)} \, E_j \, E_i^k  \; = \;
      \nu_{ii}^{(1-a_{ij}-k)k} \, \nu_{ji}^k \, E_i^{(1-a_{ij}-k)} \, E_j \, E_i^k
   \end{align*}
\end{small}
\hskip-3,5pt
 we have  $ \; \nu_{ij}^k \, \nu_{ji}^{-k} \, \nu_{ii}^{a_{ij}k + k^2 - k} \, E_i^{\,(1-a_{ij}-k)} \cdot_\sigma E_j \,\cdot_\sigma
 E_i^{\,k} \, = \, \nu_{ij}^{1-a_{ij}} \, E_i^{\,(1-a_{ij}-k)} \, E_j \, E_i^k \; $  and so  %
  $$  \clubsuit  \,\; = \;\,  \nu_{ii}^{\frac{a_{ij}(a_{ij}-1)}{2}} \, \nu_{ij}^{1-a_{ij}} \sum_{k=0}^{1-a_{ij}} {(-1)}^k \,
  {\bigg( {1-a_{ij} \atop k} \bigg)}_{\!\!q'_{ii}} {(q'_{ii})}^{k \choose 2} \, {(q'_{ij})}^k \, E_i^{\,(1-a_{ij}-k)} \, E_j \, E_i^k  \,\; =
  \;\,  0  $$
 This proves that there is a Hopf algebra epimorphism
 $ \, p : U_{\bq''}(\hskip0,8pt\liegd)   \relbar\joinrel\longrightarrow  {\big( U_{\bq'}(\hskip0,8pt\liegd) \big)}_\sigma \, $
 which is defined as the identity on the generators.  Using the same relation among the multiparameters  $ \bq' $
 and  $ \bq'' $  but written as  $ \, q'_{ij} = \nu_{ij}^{-1} \, q''_{ij} \, \nu_{ji} \, $,  \,one may define a multiplicative
 $ 2 $--cocycle  $ \sigma_{\bnu^{-1}} $  on  $ U_{\bq''}(\hskip0,8pt\liegd) $  and, via the same computations,
 one gets an epimorphism
 $ \; t :   U_{\bq'}(\hskip0,8pt\liegd) \relbar\joinrel\relbar\joinrel\longrightarrow{\big(U_{\bq''}(\hskip0,8pt\liegd)\big)}_{\sigma_{\bnu^{-1}}} \; $.
 Also, one may consider the multiplicative  $ 2 $--cocycle  $ \, \sigma_\bnu \, $  on
 $ {\big(U_{\bq''}(\hskip0,8pt\liegd)\big)}_{\sigma_{\bnu^{-1}}} $  and performing the
twisting procedure one again yields
$ \, \Big( {\big(U_{\bq''}(\hskip0,8pt\liegd)\big)}_{\sigma_{\bnu^{-1}}} \Big)_{\sigma_\bnu} =
{\big( U_{\bq''}(\hskip0,8pt\liegd) \big)}_{\sigma_{\bnu^{-1}} \ast\, \sigma_{\bnu}} = U_{\bq''}(\hskip0,8pt\liegd) \, $.
On the other hand, since  $ t $  is an epimorphism one may push forward the cocycle deformation by  $ \sigma $
and obtain an epimorphism
$ \, t_{\sigma} : {\big( U_{\bq'}(\hskip0,8pt\liegd) \big)}_\sigma \relbar\joinrel\relbar\joinrel\relbar\joinrel\longrightarrow
\Big( {\big(U_{\bq''}(\hskip0,8pt\liegd)\big)}_{\sigma_{\bnu^{-1}}} \Big)_{\sigma_\bnu} = U_{\bq''}(\hskip0,8pt\liegd) \, $.
Composing the latter with  $ p \, $,  one obtains a sequence of epimorphisms
  $$  U_{\bq''}(\hskip0,8pt\liegd) \,\;{\buildrel p \over {\relbar\joinrel\relbar\joinrel\relbar\joinrel\longrightarrow}}\;\,
  {\big( U_{\bq'}(\hskip0,8pt\liegd) \big)}_\sigma \,{\buildrel t_{\sigma} \over
  {\relbar\joinrel\relbar\joinrel\relbar\joinrel\longrightarrow}}\;\, U_{\bq''}(\hskip0,8pt\liegd)  $$
 whose composition acts as the identity on the generators, so that  $ \, t_{\sigma} \circ p  = \id \, $.
 Similarly, one proves that  $ \, p \circ t_{\sigma}  = \id \, $,  so that  $ p $  is actually an isomorphism.
 \vskip4pt
   Claim  {\it (b)\/}  is a special case of  {\it (a)},  with  $ \, \bq' := \check{\bq} \, $  and  $ \, \bq'' := \bq \; $.
 \vskip4pt
   Claim  {\it (c)\/}  is treated like the previous ones.
\epf

\medskip

\begin{rmk}
 Our (sketch of) proof of  Theorem \ref{thm:sigma_2-cocy}  above mimics that of Theorem 4.5 in \cite{LHR}:
 however, in the latter the very existence of a  $ 2 $--cocycle  as required is, to the best of our understanding,
 not well proved.  The proof of Theorem 28 in  \cite{HPR}  is along the same lines, but the  $ 2 $--cocycle
 used there is that of  Lemma \ref{lemma: mprmtr-action=>twist-equiv}{\it (c)}
 --- hence needs the existence of suitable square roots.
\end{rmk}

\medskip

\begin{cor}  \label{cor: sigma_mprmts -> sigma_MpQUEAs}
 If  $ \, \bq = {\big(\, q_{ij} \big)}_{i, j \in I} \in M_n(\FF) \, $  is a multiparameter,
 $ \, \bnu = {\big(\, \nu_{ij} \big)}_{i, j \in I} \in M_n(\FF) \, $  and
 $ \; \sigma_{\chi_\bnu} \in \widetilde{\Z}_{2}\big(\, U_{\bq}(\hskip0,8pt\liegd) \, , \FF \,\big) \; $
 is the unique element of  $ \, \widetilde{\Z}_{2}\big(\, U_{\bq}(\hskip0,8pt\liegd) \, , \FF \,\big) \, $  such that
 $ \, \sigma_{\chi_\bnu}(K_i\,,K_j) = \nu_{ij} \, $  for  $ \, i, j \in I \, $  (as in the proof of
 Theorem \ref{thm:sigma_2-cocy}  above),  then
  $$  U_{\bnu.\bq}(\hskip0,8pt\liegd)  \; \cong \;  {\big( U_\bq(\hskip0,8pt\liegd) \big)}_{\sigma_{\chi_\bnu}}  $$
   \indent   An entirely similar claim holds true for Borel (sub)MpQUEA's as well.
\end{cor}

\pf
 By  Lemma \ref{lemma: mprmtr-action=>twist-equiv}{\it (a)\/}  we know that  $ \, \bq \sim \bnu.\,\bq \, $;
 then  Theorem \ref{thm:sigma_2-cocy}{\it (a)\/}  above applies, and the claim follows at once.
\epf

\vskip9pt

   We still need some extra notation:

\vskip11pt

\begin{definition}
 For any  $ \; \bq = {\big(\, q_{ij} \big)}_{i, j \in I} \, $,  $ \, \bq' = {\big(\, q_{ij} \big)}_{i, j \in I} \in M_n(\FF) \, $,  \,
 we say that  $ \; \bq \approx \bq' \; $  if and only if there exists a permutation  $ \gamma $  of  $ I $  such that
  $$  q'_{ij}  \, = \,  q_{\gamma(i)\gamma(j)}   \quad  \forall \ i, j \in I  \;\qquad  \text{ or }  \;\qquad  q'_{ij}  \, = \,
  q_{\gamma(j)\gamma(i)}^{-1}   \quad \forall \ i,j \in I \; .  $$
\end{definition}

\vskip9pt

\begin{prop}\label{prop:isom-multip}
 Let  $ \; \bq = {\big(\, q_{ij} \big)}_{i, j \in I} \, $,  $ \, \bq' = {\big(\, q_{ij} \big)}_{i, j \in I} \in M_n(\FF) \, $
 be multiparameters.  Then
 \quad  $ U_{\bq}(\hskip0,8pt\liegd) \; \cong \; U_{\bq'}(\hskip0,8pt\liegd) \; \iff \; \bq \approx\bq' \;\; $.
\end{prop}

\pf
 We denote by  $ E_i \, $,  $ F_i \, $,  $ K_i^{\pm 1} $ and  $ L_i^{\pm 1} $  (with  $ \, i \in I \, $)
 the generators of  $ U_\bq(\hskip0,8pt\liegd) $  and by
 $ E'_j \, $,  $ F'_j \, $,  $ K^{\prime \, \pm 1}_j $  and  $ L^{\prime \, \pm 1}_j $  (with  $ \, j \in I \, $)
 the generators of  $ U_{\bq'}(\hskip0,8pt\liegd) \, $.
 \vskip5pt
   Assume first that there is a Hopf algebra isomorphism
   $ \, \varphi: U_\bq(\hskip0,8pt\liegd) \relbar\joinrel\relbar\joinrel\longrightarrow U_{\bq'}(\hskip0,8pt\liegd) \, $.
   Note that both Hopf algebras  $ \, U_\bq(\hskip0,8pt\liegd) \, $  and $ \, U_{\bq'}(\hskip0,8pt\liegd) \, $
   are pointed, as they are generated by group-likes and skew-primitives: then, in particular,
   their coradicals coincide with the subalgebras  $ \, U_\bq(\hskip0,8pt\liehd) \, $  and
   $ \, U_{\bq'}(\hskip0,8pt\liehd) \, $,  respectively, which are actually isomorphic.
   Moreover, by \cite[Theorem 3.8]{Gar}, they belong to a special family of pointed Hopf algebras
   introduced in  \cite{ARS};  then, using the coradical filtration one can write
   $ \, \gr\big( U_\bq(\hskip0,8pt\liegd) \big) = R \,\#\, U_\bq(\hskip0,8pt\liehd) \, $  and
   $ \, \gr\big( U_{\bq'}(\hskip0,8pt\liegd) \big) = R' \,\#\, U_{\bq'}(\hskip0,8pt\liehd) \, $,
   where  $ \, R \, $  and  $ \, R' \, $  are braided Hopf algebras in the category of Yetter-Drinfeld modules over
   $ \, U_\bq(\hskip0,8pt\liehd) \, $  and  $ \, U_{\bq'}(\hskip0,8pt\liehd) \, $, respectively.
   These are given by the subalgebra of coinvariants, i.e.,
   $ \, R = \big\{\, x \in \gr\big( U_\bq(\hskip0,8pt\liegd) \big) \,\big|\, (\id\ot \pi) \com(x) = x\ot 1 \,\big\} \, $,  where
   $ \, \pi : \gr\big( U_\bq(\hskip0,8pt\liegd) \big) \longrightarrow U_\bq(\hskip0,8pt\liehd) \, $  is the canonical projection.
   As  $ \, \varphi\big( U_\bq(\hskip0,8pt\liehd) \big) = U_{\bq'}(\hskip0,8pt\liehd) \, $,  the map  $ \varphi $
   induces a braided Hopf algebra isomorphism
   $ \, \tilde{\varphi} : \gr\big( U_\bq(\hskip0,8pt\liegd) \big) \!\relbar\joinrel\longrightarrow \gr\big( U_{\bq'}(\hskip0,8pt\liegd)
   \big) \, $  such that  $ \, \tilde{\varphi}(R) = R' \, $.  In particular,  $ \tilde{\varphi} $
   maps primitive element in  $ R $  to primitive elements in $ R' \, $.
                                                               \par
   Now take  $ \, i \in I \, $.  As (the image of the element in the graded algebra)  $ \, E_i \, $
   is  primitive in  $ R \, $,  we have that  $ \tilde{\varphi}(E_i) $  is primitive in  $ R' $  as well.
   As the latter is linearly spanned by the elements  $ \, E'_i \, $  and  $ \, L_j^{\prime\,-1} F_j \, $  with  $ i, j \in I \, $
   --- these are the linear generators of  $ \, V_E \, $  and  $ \, V_F \, $  in Remarks \ref{rmks:mpqg's-vs-Nichols}
   above ---   this implies that  $ \, \varphi(K_i) = K'_j \, $  or  $ \, \varphi(K_i) = L_j^{\prime \, -1} \, $  for some
   $ \, j \in I \, $.  In particular, one has that  $ \, \varphi(E_i) = c_{i,j} E'_j \, $  or
   $ \, \varphi(E_i) = d_{i,j} L_j^{\prime \, -1} F'_j \, $,  respectively, for some  $ \, c_{i,j} , d_{i,j} \in \FF \, $.
   Thus, in the end,  $ \varphi $  induces a bijective map  $ \, \gamma : I \longrightarrow I \, $  such that
   $ \, \varphi(K_i) = K'_{\gamma(i)} \, $  and  $ \varphi(E_i) = c_{i,\gamma(i)} E'_{\gamma(i)} \, $  or
   $ \, \varphi(K_i) = L_{\gamma(i)}^{\prime \, -1} \, $  and
   $ \, \varphi(E_i) = d_{i,\gamma(i)} L_{\gamma(i)}^{\prime \, -1} F'_{\gamma(i)} \, $.
 \vskip7pt
  {\sl  $ \underline{\hbox{Claim 1}} $:}  $ \; \varphi(E_i) =c_{i,\gamma(i)}  E'_{\gamma(i)} \; $  for all
  $ \, i \in I \, $  or  $ \; \varphi(E_i) =  d_{i,\gamma(i)} L_{\gamma(i)}^{\prime \, -1} F'_{\gamma(i)} \; $  for all $ \, i \in I \, $.
 \vskip5pt
   Indeed, suppose that  $ \, \varphi(E_i) = c_{i,\gamma(i)} E'_{\gamma(i)} \, $  and
   $ \, \varphi(E_j) =  d_{j,\gamma(j)} L_{\gamma(j)}^{\prime \, -1} F'_{\gamma(j)} \, $
   for some  $ \, i \neq j \in I \, $;  as the Cartan matrix  $ A $  is indecomposable
   --- hence the Dynkin diagram is connected ---   we can assume in addition that  $ i $  and  $ j $
   are adjacent in the Dynkin diagram, i.e.,  $ \, a_{ij} \not= 0 \, $.  Now,  $ \, \varphi(K_i) = K'_{\gamma(i)} \, $
   and  $ \, \varphi(K_j) = L_{\gamma(j)}^{\prime \, -1} \; $;  since
  $$  q_{ij} \, \varphi(E_j)  \,\; = \;\,  \varphi\big( K_i \, E_j \, K_i^{-1} \big)  \,\; = \;\,
%
%
  d_{j,\gamma(j)}  K'_{\gamma(i)} \, L_{\gamma(j)}^{\prime \, -1} \, F'_{\gamma(j)} \, K_{\gamma(i)}^{\prime \, -1}  \,\;
  = \;\,  q_{\gamma(i)\gamma(j)}^{\prime \, -1} \, \varphi(E_j)  $$
 it follows that  $ \; q_{ij} = q_{\gamma(i)\gamma(j)}^{\prime \, -1} \; $.  On the other hand,
  $$  q_{ji} \, \varphi(E_i)  \,\; = \;\,  \varphi\big( K_j \, E_i \, K_j^{-1} \big)  \,\; = \;\, %
%
 c_{i,\gamma(i)} L_{\gamma(j)}^{\prime \, -1} \, E'_{\gamma(i)} \, L'_{\gamma(j)}  \,\; =
 \;\, q'_{\gamma(i)\gamma(j)} \, \varphi(E_{i})  $$
 which implies that  $ \; q_{ji} = q'_{\gamma(i)\gamma(j)} \; $.
 Therefore  $ \; q_{ii}^{a_{ij}} \, = \, q_{ij} \, q_{ji} \, = \, 1 \; $  and  $ \; a_{ij} = 0 \; $,  as the  $ q_{ii} $'s
 are not  $ a_{ij} $--th  roots of unity:  but
 $ \, a_{ij} \not= 0 \, $  by assumption, a contradiction.
 \vskip7pt
  {\sl  $ \underline{\hbox{Claim 2}} $:}  \;  Let  $ \, i \in I \, $.  If  $ \, \varphi(E_i) = c_{i,\gamma(i)} E'_{\gamma(i)} \; $,
  respectively
  $ \, \varphi(E_i) = d_{i,\gamma(i)} L_{\gamma(i)}^{\prime \, -1} \, F'_{\gamma(i)}$,  then
  $ \, \varphi(F_i) = r_{i,\gamma(i)}F'_{\gamma(i)} \; $,  respectively
  $ \, \varphi(F_i) = s_{i,\gamma(i)}K_{\gamma(i)}^{\prime \, -1} \, E'_{\gamma(i)} \; $,
  \,for some $r_{i,\gamma(i)}$,
  $s_{i,\gamma(i)} \in \FF^{\times}$.
 \vskip7pt
   Indeed, since  $ \; [E_i,F_i] \, = \, q_{ii} \, {\displaystyle \frac{\, K_i - L_i \,}{\, q_{ii} - 1 \,}} \, \in \, U_\bq(\liehd) \; $,
   and  $ \, \varphi(U_\bq(\liehd)) = U_{\bq'}(\liehd) \, $, we must have
   $ \varphi([E_i,F_i]) = \; \big[ \varphi(E_i) , \varphi(F_i) \big] \in U_{\bq'}(\liehd) \; $.
   Hence, if  $ \, \varphi(E_i) = c_{i,\gamma(i)} E'_{\gamma(i)} \, $  and
   $ \, \varphi(F_i) = s_{i,\gamma(i)} K_{\gamma(i)}^{\prime \, -1} \, E'_{\gamma(i)} \, $,
   we would have that  $ \, \big[ \varphi(E_i) , \varphi(F_i) \big] \in  U_{\bq'}^{\geq}\cap  U_{\bq'}(\liehd)\, $
   is zero or the degree w.r.t.\ the PBW-element  $ \, E'_{\gamma(i)} \, $  is 2.  Thus
   $ \, \big[ \varphi(E_i) , \varphi(F_i) \big] = 0 \, $,  which in turn implies that  $ \, \varphi(K_i) = \varphi(L_i) \, $,
   against the injectivity of  $ \varphi \, $.  Therefore we must have that
   $ \, \varphi(F_i) = r_{i,\gamma(i)}F'_{\gamma(i)} \, $  and $ \, \varphi(L_i) = L'_{\gamma(i)} \, $.
                                                         \par
   The remaining case is analogous.
 \vskip7pt
   The claims above imply that  $ \varphi $  is given by a Dynkin diagram automorphism possibly composed with a
   ``Chevalley involution''.  Namely, if  $ \; \varphi(E_i) = c_{i,\gamma(i)} E'_{\gamma(i)} \; $  for all  $ \, i \in I \, $,
   then  $ \, q'_{ij} = q_{\gamma(i)\gamma(j)} \, $  for all  $ \, i , j \in I \, $  and  $ \gamma $  is an automorphism
   of the generalized Dynkin diagram  $ D(\bq) \, $.  If instead
   $ \, \varphi(E_i) = d_{i,\gamma(i)}L_{\gamma(i)}^{\prime \, -1} \, F'_{\gamma(i)} \; $  for all  $ \, i \in I \, $,
   then  $ \, q'_{ij} = q_{\gamma(j)\gamma(i)}^{\, -1} \, $  for all  $ \, i , j \in I \, $,  and  $ \gamma $
   is an automorphism of  $ D(\bq) $  composed with a Chevalley involution.  Altogether, this implies that
   $ \, \bq \approx \bq' \, $,  \, q.e.d.
 \vskip7pt
   Conversely, assume now that  $ \; \bq \approx \bq' \; $.  If  $ \, q'_{ij} = q_{\gamma(i)\gamma(j)} \, $
   for all  $ \, i , j \in I \, $,  then the algebra map
   $ \; \varphi : U_\bq(\hskip0,8pt\liegd) \relbar\joinrel\longrightarrow  U_{\bq'}(\hskip0,8pt\liegd) \; $  given for all
   $ \, i \in I \, $   by
  $$  \varphi(E_i) := E'_{\gamma(i)} \;\, ,  \;\quad  \varphi(F_i) := F'_{\gamma(i)} \;\, ,
  \;\quad \varphi\big(K_i^{\pm 1}\big) := K_{\gamma(i)}^{\prime \, \pm 1} \;\, ,  \;\quad  \varphi\big(L_i^{\pm 1}\big) :=
  L_{\gamma(i)}^{\prime \, \pm 1}  $$
 is well-defined and yields a Hopf algebra isomorphism.  Similarly, if  $ \; q'_{ij} = q_{\gamma(j)\gamma(i)}^{-1} \; $
 for all  $ \, i , j \in I \, $,  then the algebra map
 $ \; \varphi : U_\bq(\hskip0,8pt\liegd) \relbar\joinrel\longrightarrow U_{\bq'}(\hskip0,8pt\liegd) \; $
 given for  $ \, i \in I \, $  by
  $$  \varphi(E_i) := L_{\gamma(i)}^{\prime \, -1} \, F'_{\gamma(i)} \;\, ,  \;\;  \varphi(F_i) :=
  K_{\gamma(i)}^{\prime \, -1} \, E'_{\gamma(i)} \;\, ,
  \;\;  \varphi\big(K_i^{\pm 1}\big) := L_{\gamma(i)}^{\prime \, \mp 1} \;\, ,  \;\;  \varphi\big(L_i^{\pm 1}\big) :=
  K_{\gamma(i)}^{\prime \, \mp 1}  $$
 is well-defined and yields a Hopf algebra isomorphism.
\epf

\end{free text}

\vskip17pt

\subsection{Multiparameter quantum groups with larger torus} \label{MpQUEAs-larger-torus}  \
 \vskip7pt
   The MpQUEA's  $ \QEq $  we considered so far have a toral part (i.e., the subalgebra  $ U_\bq^0 $
   generated by the  $ K_i^{\pm 1} $'s
   and the  $ L_j^{\pm 1} $'s)  that is nothing but the group algebra of a double copy of the root lattice  $ Q $  of
   $ \lieg \, $,
   much like in the one-parameter case (but for the duplication of  $ Q \, $,  say).
   Now, in that (uniparameter) case, one also
   considers MpQUEA's with a larger toral part, namely the group algebra of any intermediate lattice between
   $ Q $  and  $ P \, $;
   similarly, we can introduce MpQUEA's whose toral part is the group algebra of any lattice
   $ \, \varGamma_\ell \times \varGamma_r \, $
   with  $ \, Q \subseteq \varGamma_\ell \, $  and  $ \, Q \subseteq \varGamma_r \, $.

\vskip13pt

\begin{free text}  \label{tori_in_MpQUEA's}
 {\bf Larger tori for MpQUEA's.}
 The definition of the ``toral parts'' of our MpQUEA's  $ \, U_\bq(\hskip0,8pt\liegd) \, $
 --- cf.\  Definition \ref{def:q-Bor-sbgr}  ---
 is actually independent of the multiparameter  $ \bq \, $.  We will use this fact to define ``larger toral MpQUEA's''
 as toral parts
 of some larger MpQUEA's, as we did in \S \ref{pol-qgroups_larg-tor}.  This requires some compatibility constraints on
 $ \bq \, $;
 for later use, we fix now some more preliminary facts.
                                                      \par
   Let  $ \varGamma $  be any sublattice of  $ \QQ{}Q $  of rank  $ n $  with  $ \, Q \leq \varGamma \, $.
   For any basis  $ \big\{ \gamma_1 , \dots \gamma_n \big\} $  of  $ \varGamma \, $,  let
   $ \, C := {\big( c_{ij} \big)}_{i=1,\dots,n;}^{j=1,\dots,n;} \in \ZZ^{n\times n}\, $  be
 defined by
 $ \, \alpha_i = \sum_{j=1}^n c_{ij}\,\gamma_j \, $  for every  $ \, i \in I = \{1,\dots,n\} \, $.  Write
 $ \, c := \big| \text{\sl det}(C) \big| \in \NN_+ \, $  and
 $ \, C^{-1} = {\big( c'_{ij} \big)}_{i=1, \dots, n;}^{j=1, \dots, n;} \, $  for the inverse matrix to  $ C \, $;
 in particular, we have that  $ \, c''_{ij} := c \cdot c'_{ij} \in \ZZ \, $  for all  $ \, i, j \in I \, $.
                                                      \par
   For any such  $ \varGamma \, $,  let  $ \, U_{\bq\,,\varGamma}^{+,0} = \FF\varGamma \, $
   be the group algebra with generators  $ K_{\gamma_i}^{\pm 1} $  (for  $ \, i \in I \, $).
   If we set  $ \, K_{\alpha_i} := \prod_{j \in I} K_{\gamma_j}^{\,c_{ij}} \, $  for all  $ \, i \in I \, $,  then the
   $ \FF $--subalgebra  of  $ U_{\bq\,,\varGamma}^{+,0} $  generated by the  $ K_{\alpha_i}^{\pm 1} $'s
   is an isomorphic copy of  $ U_{\bq\,,Q}^{+,0} \, $.  In the obvious symmetric way we define also the
   ``negative counterpart''  $ U_{\bq\,,\varGamma}^{-,0} \, $  generated by the $ L_{\gamma_i}^{\pm 1} $'s.
                                                      \par
   Finally, given any two sublattices  $ \varGamma_\pm $  of rank  $ n $  in  $ \QQ{}Q $  containing  $ Q \, $,
   letting  $ \, \varGamma_\bullet := \varGamma_+ \times \varGamma_- \, $  we define
   $ \; U_{\bq\,,\varGamma_\bullet}^0 := \, U_{\bq\,,\varGamma_+}^{+,0} \!\mathop{\otimes}\limits_\FF \,
   U_{\bq\,,\varGamma_-}^{-,0} \; $;  in this case, the basis elements for  $ \varGamma_\pm $  will be denoted by
   $ \, \gamma_i^\pm \, $  ($ \, i \in I \, $).
\end{free text}

\vskip5pt

\begin{free text}  \label{larger-MpQUEA's}
 {\bf MpQUEA's with larger torus.}  Let  $ \varGamma_+ $  and  $ \varGamma_- $
 be any two lattices in  $ \QQ{}Q $ such that  $ \, Q \leq \varGamma_\pm \, $;
 then set  $ \, \varGamma_{\!\bullet} := \varGamma_+ \times \varGamma_- \, $.  For these lattices
 $ \varGamma_\pm \, $,  we have matrices of integers  $ \, C_\pm = {\big( c_{ij}^\pm \big)}_{i,j \in I} \, $
 and  $ \, C_\pm^{\,-1} = {\big( c_{ij}^{\pm,\prime\,} \big)}_{i,j \in I} \, $,  and also
 $ \, c_\pm := \big| \text{\sl det}(C_\pm) \big| \, $.
 \vskip7pt
   {\sl For the rest of this subsection, we assume that for every  $ \, i, j \in I \, $,
   the field  $ \FF $  contains a  $ c_\pm $--th  root of  $ \, q_{ij} \, $,  hereafter denoted by
   $ q_{ij}^{\,1/c_{{}_\pm}} \, $;  moreover, if the multiparameter  $ \, \bq = {\big(\, q_{ij} \big)}_{i,j \in I} \, $
   is of Cartan type we assume that the multiparameter
   $ \, \bq^{\,1/c_{{}_\pm}} := {\Big(\, q_{ij}^{\,1/c_{{}_\pm}} \Big)}_{i,j \in I} \, $  is of Cartan type as well.}
   In the sequel,  {\sl these assumptions will always tacitly be taken whenever we consider MpQUEA's with larger tori}.
 \vskip7pt
   The natural (adjoint) action of  $ U_\bq^0 $  onto  $ U_\bq $  extends (uniquely) to a
   $ U_{\bq\,,\varGamma_{\!\bullet}}^0 $--action
   $ \, \cdot :\, U_{\bq\,,\varGamma_{\!\bullet}}^0 \times U_\bq \longrightarrow U_\bq \, $, given by
\begin{align*}
  K_{\gamma^+_i} \,\cdot \, E_j  \, = \,  q_{ij}^{\,\varGamma_{\!+}} \, E_j  \;\; ,  \qquad
      L_{\gamma^-_i} \,\cdot\, E_j  \, = \,  {\big( q_{ji}^{\,\varGamma_{\!-}} \big)}^{-1} \, E_j  \\
  \hskip5pt  K_{\gamma^+_i} \,\cdot \, F_j  \, = \,  {\big( q_{ij}^{\,\varGamma_{\!+}} \big)}^{-1} \, F_j  \;\; ,  \qquad \hskip7pt
      L_{\gamma^-_i} \,\cdot \, F_j  \, = \,  q_{ji}^{\,\varGamma_{\!-}} \, F_j  \\
  K_{\gamma^+_i} \,\cdot \, K_{\alpha_{j}}  \, = \,   K_{\alpha_{j}}  \;\; ,  \qquad
      L_{\gamma^-_i} \,\cdot\,  K_{\alpha_{j}}  \, = \,   K_{\alpha_{j}}
\end{align*}
--- where  $ \; q_{rs}^{\,\varGamma_{\!+}} := \prod_{k \in I} {\big( q_{ks}^{\,1/c_{{}_+}} \big)}^{\!c_{rk}^{+,\prime\prime}} \; $
and  $ \,\; q_{ae}^{\,\varGamma_{\!-}} := \prod_{k \in I} {\big( q_{ak}^{\,1/c_{{}_-}} \big)}^{\!c_{ek}^{-,\prime\prime}} \; $
---   that makes  $ U_\bq $  into a  $ U_{\bq\,,\varGamma_{\!\bullet}}^0 $--module algebra.
This allows us to consider the  {\it smash product\/} Hopf algebra
$ \, U_{\bq\,,\varGamma_{\!\bullet}}^0 \ltimes U_\bq \, $;
its product is given by the formula \eqref{eq:smashprod}.
As  $ U_{\bq\,,\varGamma_{\!\bullet}}^0 $  is a right  $ U_\bq^0 $--module
Hopf algebra with respect to the adjoint action, one may consider the vector space
$ \, U_{\bq\,,\varGamma_{\!\bullet}}^0 \mathop{\otimes}\limits_{U_\bq^0} U_\bq \, $.
Moreover, the smash product  $ \, U_{\bq\,,\varGamma_{\!\bullet}}^0 \ltimes U_\bq \, $
maps onto a Hopf algebra structure on it, which hereafter we denote by
$ \, U_{\bq\,,\varGamma_{\!\bullet}}^0 \mathop{\ltimes}\limits_{U_\bq^0} U_\bq \, $
--- see, e.g.,  \cite[Theorem 2.8]{Len}.  We define then
\begin{equation}  \label{def-MpQUEA_wt_lrg-torus}
  U_{\bq\,,\varGamma_{\!\bullet}}(\hskip0,8pt\liegd)  \, \equiv \,  U_{\bq\,,\varGamma_{\!\bullet}}  \, :=
  \,  U_{\bq\,,\varGamma_{\!\bullet}}^0 \mathop{\ltimes}\limits_{U_\bq^0} U_\bq  \, =
  \, U_{\bq\,,\varGamma_{\!\bullet}}^0 \mathop{\ltimes}\limits_{U_\bq^0} \QEq
\end{equation}
   \indent   Since the coalgebra structure is the one given by the tensor product, to give a
   presentation by generators and relations like that for  $ \QEq \, $,  one has to describe
   only the algebra structure.  For this, one has to replace the generators
   $ \, K_i^{\pm 1} = K_{\pm \alpha_i} \, $  and  $ \, L_i^{\pm 1} = L_{\pm \alpha_i} \, $
   with the generators  $ \, \Kc_i^{\pm 1} = K_{\pm \gamma^+_i} \, $  and
   $ \, \Lc_i^{\pm 1} = L_{\pm \gamma^-_i} \, $;  replace relations  {\it (c)\/}  and  {\it (d)\/}  of
   Definition \ref{def:multiqgroup_ang}  with the following, generalized relations:
\begin{align*}
  \big(c'\big)  \qquad  &  \hskip3pt  K_{\gamma^+_i\,} E_j \, K_{\gamma^+_i}^{-1}  \,
  = \,  q_{ij}^{\,\varGamma_{\!+}} \, E_j  \;\; ,  \qquad
      L_{\gamma^-_i\,} E_j \, L_{\gamma^-_i}^{-1}  \, = \,  {\big( q_{ji}^{\,\varGamma_{\!-}} \big)}^{-1} \, E_j  \\
  \big(d'\big)  \qquad  &  \hskip5pt  K_{\gamma^+_i\,} F_j \, K_{\gamma^+_i}^{-1}  \, =
  \,  {\big( q_{ij}^{\,\varGamma_{\!+}} \big)}^{-1} \, F_j  \;\; ,  \qquad \hskip1pt
      L_{\gamma^-_i\,} F_j \, L_{\gamma^-_i}^{-1}  \, = \,  q_{ji}^{\,\varGamma_{\!-}} \, F_j
\end{align*}
finally, in relation  {\it (e)\/}  replace the elements  $ K_i $  and  $ L_i $  by  $ K_{\alpha_i} $
and $ L_{\alpha_i} \, $,  respectively, and leave the quantum Serre relations  {\it (f)\/}  and {\it (g)\/} unchanged.
                                 \par
   With much the same approach, one defines also the ``(multiparameter) quantum subgroups'' of
   $ U_{\bq\,,\varGamma_{\!\bullet}}(\hskip0,8pt\liegd) $  akin to those of  $ \QEq $  (cf.\  Definition \ref{def:q-Bor-sbgr}),
   that we denote by adding a subscript  $ \varGamma_{\!\bullet} \, $,  namely
   $ U_{\bq\,,\varGamma_{\!\bullet}}^+ \, $,  $ U_{\bq\,,\varGamma_\bullet}^- \, $,
   $ U_{\bq\,,\varGamma_{\!\bullet}}^\geq \, $,  $ U_{\bq\,,\varGamma_{\!\bullet}}^\leq \, $,
   $ U_{\bq\,,\varGamma_{\!\bullet}}^{+,0} $  and  $ U_{\bq\,,\varGamma_{\!\bullet}}^{-,0} \, $.
\end{free text}

\vskip7pt

\begin{free text}  \label{duality x larger MpQUEA's}
 {\bf Duality among MpQUEA's with larger torus.}  Let again  $ \, \varGamma_\pm \, $
 be two lattices of rank  $ n $  in  $ \QQ{}Q $
 containing  $ Q \, $,  and set  $ \, \varGamma_\bullet := \varGamma_+ \times \varGamma_- \, $.
 We repeat our assumptions  (cf.\ \S \ref{larger-MpQUEA's}).  We fix bases  $ {\big\{\, \gamma_s^\pm \,\big\}}_{s \in I} $
 of  $ \varGamma_\pm \, $,
 the matrices
 $ \, C_\pm = {\big( c_{ij}^\pm \big)}_{i,j \in I} \, $,  $ \, C_\pm^{\,-1} = {\big( c_{ij}^{\pm,\prime\,} \big)}_{i,j \in I} \, $,
 and write  $ \, c_\pm := \big| \text{\sl det}(C_\pm) \big| \, $  and
 $ \, c_{ij}^{\pm,\prime\prime} := c_\pm \cdot c_{ij}^{\pm,\prime} \, $  ($ \, i, j \in I \, $).
 In addition, we assume that  $ \FF $  contain a  $ (c_+\,c_-) $--th  root of  $ \, q_{ij} \, $,  say
 $ q_{ij}^{\,1/(c_{{}_+}c_{{}_-})} $,  and that overall the multiparameter
 $ \, \bq^{\,1/(c_{{}_+}c_{{}_-})} := {\Big( q_{ij}^{\,1/(c_{{}_+}c_{{}_-})} \Big)}_{i,j \in I}\, $  be of Cartan type.
                                                               \par
   It is straightforward to check that the Hopf skew-pairing
   $ \; \eta : U_\bq^\geq \mathop{\otimes}\limits_\FF U_\bq^{\leq} \!\relbar\joinrel\relbar\joinrel\longrightarrow \FF \; $
   in  Proposition \ref{sk-H_pair}   --- cf.\  Remarks \ref{rmks:mpqg's-vs-Nichols}{\it (a)\/}  too
   --- actually extends to a similar skew-pairing
   $ \; U_{\bq\,,\varGamma_+}^\geq \mathop{\otimes}\limits_\FF
   U_{\bq\,,\varGamma_-}^{\leq} \!\! \relbar\joinrel\relbar\joinrel\longrightarrow \FF \; $  given for all  $ \, i, j \in I \, $ by
  $$  \begin{matrix}
   \eta( E_i \otimes\! L_{\gamma^-}) \; = \; 0  \qquad   &   \qquad  \eta(K_{\gamma^+} \!\otimes F_j) \; = \; 0  \\
  \eta(E_i \otimes F_j) \; = \; -\delta_{ij} {{\,q_{ii}\,}\over {\,q_{ii} - 1\,}}  \qquad   &
  \qquad  \eta( K_{\gamma_i^+\!} \otimes L_{\gamma_j^-}) \; = \;
  {\textstyle \prod\limits_{h, k \in I}} {\Big( q_{hk}^{\,1/(c_{{}_+}c_{{}_-})} \Big)}^{\! c_{ih}^{+,\prime\prime}
  c_{jk}^{-,\prime\prime}}
      \end{matrix}  $$
 \vskip5pt
 One easily sees that, using such a Hopf pairing  $ \eta $  between (suitably chosen) quantum Borel subgroups
 $ U_{\bq\,,\varGamma_+}^{\,\geq} $  and  $ U_{\bq\,,\varGamma_-}^{\,\leq} \, $,
 {\it every MpQUEA with larger torus  $ \, U_{\bq\,,\varGamma_{\!\bullet}}(\hskip0,8pt\liegd) \, $
 can be realized as a Drinfeld double (of those quantum Borel subgroups)},
 thus generalizing what happens for MpQUEA's with ``standard'' torus   --- see
 Remarks \ref{rmks:mpqg's-vs-Nichols}{\it (b)}.
\end{free text}

\vskip5pt

\begin{rmk}
 From the description above, it follows that one clearly can suitably extend both  Theorem \ref{thm:sigma_2-cocy}  and
 Corollary \ref{cor: sigma_mprmts -> sigma_MpQUEAs}  to the case of MpQUEA's with larger torus:
 Namely, let  $ \, \varGamma_\pm \, $  be two lattices of rank  $ n $  in  $ \QQ{}Q $  containing  $ Q \, $,  and set
 $ \, \varGamma_\bullet := \varGamma_+ \times \varGamma_- \, $.  In addition, assume that  $ \FF $  contain a
 $ (c_1\,c_2) $--th  root of  $ \, q_{ij} \, $  with  $ \, c_i \in \{c_+ , c_-\} \, $,  say  $ q_{ij}^{\,1/(c_{{}_1}c_{{}_2})} $,
 and that overall the multiparameter  $ \, \bq^{\,1/(c_{{}_1}c_{{}_2})} := {\Big( q_{ij}^{\,1/(c_{{}_1}c_{{}_2})} \Big)}_{i,j \in I}\, $
 be of Cartan type.  Pick the standard MpQUEA $ \, U_{\check{\bq},\varGamma_{\!\bullet}}(\hskip0,8pt\liegd) \, $
 with larger torus  $ \varGamma_\bullet $  and canonical multiparameter  $ \check{\bq} \, $.
 Then there exist a normalized  $ 2 $--cocycle  of the Hopf algebra
 $ \, U_{ \check{\bq},\varGamma_{\!\bullet}}(\hskip0,8pt\liegd) \, $
 and a Hopf algebra isomorphism such that
  $$  U_{\bq\,,\varGamma_{\!\bullet}}(\hskip0,8pt\liegd)  \; \cong \;
  {\big(\, U_{ \check{\bq},\varGamma_{\!\bullet}}(\hskip0,8pt\liegd) \big)}_\sigma  $$
\end{rmk}

\bigskip

\section{Twisted QUEA's vs.\ multiparameter QUEA's}  \label{TwQUEA's-MpQUEA's}

\smallskip

   In this section we show
 that,  \textsl{assuming the Cartan matrix  $ A $  to be of  \textit{finite type}},
 all  TwQUEA's   --- as considered in  \S \ref{twist_polyn-QUEAs}  ---   are actually (isomorphic to) MpQUEA's
 --- for multiparameters in a special, yet quite general, subclass of integral type.
 Conversely, any MpQUEA for such a multiparameter is (isomorphic to) a TwQUEA for a single, specific twist element.
 In short   --- and up to sticking to twists of type  \eqref{Resh-twist_F-uhg}  or  \eqref{Resh-twist_F-uhgd}
 and to multiparameters of integral type ---   we have
 \vskip4pt
   \centerline{\it every TwQUEA is (isomorphic to) a MpQUEA, and viceversa.}
 \vskip4pt
   In all this section, we fix as ground field
   $ \, \FF := \k_q := \lim\limits_{\buildrel {\leftarrow\joinrel\relbar\joinrel\relbar} \over {m \in \NN}} \k\big( q^{1/m} \big) \, $,
   with  $ \k $ being a field of characteristic zero and  $ q $  an indeterminate.
   As before, by  $ q^r $  for any  $ \, r \in \QQ \, $  we always mean
   $ \, q^r = q^{a/d} := {\big( q^{1/d} \big)}^a \, $  if  $ \, r = a/d \, $  with  $ \, a \in \ZZ \, $  and  $ \, d \in \NN_+ \, $.
 In addition,  \textsl{we assume that our (indecomposable) Cartan matrix  $ A $  is of  \textit{finite type}}.

\vskip19pt

\subsection{Twists vs.\ ``rational'' multiparameters}  \label{Tw-vs-ratMp}  \

\vskip11pt

   We introduce now a special subclass of integral multiparameters
   (cf.\  \S \ref{mult-multiparameters})  in  $ \, \FF := \k_q \, $;  then we link them with twist elements
   (as in  \eqref{Resh-twist_F-uhg}  or  \eqref{Resh-twist_F-uhgd}  alike).

\vskip11pt

\begin{free text}  \label{q-rat_Mp}
 {\bf  $ q $--rational  multiparameters.}
 We fix now some more notation.
                                                            \par
   First of all, we remark that for any matrix  $ \, C \in M_n(\QQ) \, $  it is uniquely defined
   $ \, d_C := \min{\big\{\, d \in \NN \,\big|\, d\,C \in M_n(\ZZ) \,\big\}} \, $.  Second, for any given
   $ \, R \in M_n(\QQ) \, $  we adopt notation $ \, \underline{q^R} := {\big( q^{r_{ij}} \big)}_{i,j \in I} \, $.
   We say that
 {\it a multiparameter  $ \, \bq := {\big(\, q_{ij} \big)}_{i,j \in I} \, $  in  $ \, \FF := \k_q \, $
 is of  $ q $--rational  type  $ \, R := {\big( r_{ij} \big)}_{i,j \in I} \, $}
 if  $ \, R \in M_n(\QQ) \, $  and  $ \bq $  is of Cartan type with  $ \, \bq = \underline{q^R} \, $,
 i.e.,  $ \, q_{ij} = q^{r_{ij}} \, $  for all  $ \, i, j \in I \, $.  In other words, writing
 $ \, r_{ij} = b_{ij} \big/ d_R \, $  for some  $ \, b_{ij} \in \ZZ \, $  ($ \, i, j \in I \, $),
 we have that
 {\sl  $ \, \bq := {\big(\, q_{ij} \big)}_{i,j \in I} \, $  is of  $ q $--rational  type
 $ \, R := {\big( r_{ij} \big)}_{i,j \in I} \, $  if and only if it is of integral type  $ \, \big(\,
 q^{1/d_R} , \, B := {\big( b_{ij} \big)}_{i,j \in I} \,\big) \, $}.
                                                                      \par
   In the sequel, we call  $ \, \text{\sl $ q $--Mp}_\QQ \, $  the set of all multiparameters of
   $ q $--rational  type (or simply  ``$ q $--rational  multiparameters'')   ---
   for any possible  $ \, R := {\big( r_{ij} \big)}_{i,j \in I} $  ---   in  $ \k_q \, $.
\end{free text}

\vskip11pt

\begin{free text}  \label{Twists<->Mp's<->2-coc's}
 {\bf The links  $ \, \boldsymbol{\{} $twists$ \boldsymbol{\}} \leftrightarrows \boldsymbol{\{}
 $multiparameters$ \boldsymbol{\}} \leftrightarrows \boldsymbol{\{} $2-cocycles$ \boldsymbol{\}} \, $}.
 We begin defining two maps from  $ M_n(\QQ) $  to itself given by
\begin{equation}  \label{eq: def-theta(Psi)}
   \Psi  \; \mapsto \;  \vartheta(\Psi)  \, := \,
 d_\Psi^{-1}
 D A + A^{\scriptscriptstyle T} \big( \Psi^{\scriptscriptstyle T} \! - \Psi \big) \, A
\end{equation}
 and
\begin{equation}  \label{eq: def-xi(R)}
   R  \; \mapsto \;  \xi(R)  \, := \,  2^{-1} A^{\scriptscriptstyle -T} \big(
 d_R^{-1} D A - R \,\big) \, A^{-1}
\end{equation}
   \indent   A moment's check shows that the following hold:
 \vskip5pt
   \quad  {\it (a)} \quad  $ \; \text{\sl Im}(\vartheta) \, = \, \big\{\, R \in M_n(\QQ ) \,\big|\, R + \! R^{\scriptscriptstyle T} =
 d_R^{-1} \, 2 \, D A \,\big\} \,   =: \, \mathfrak{o}_{{}_{2DA}}(\QQ) \; $
 \vskip3pt
   \quad  {\it (b)} \quad  $ \; \text{\sl Im}(\xi) \, =
   \, \big\{\, \Psi \in M_n(\QQ ) \,\big|\, \Psi + \Psi^{\scriptscriptstyle T} = 0 \,\big\} \,
   =: \, \lieso_n(\QQ) \; $
 \vskip3pt
   \quad  {\it (c)} \qquad  $ \;\; \xi \circ \vartheta \, = \, \id_{\lieso_n(\QQ)} \;\;\; ,
   \;\quad  \vartheta \circ \xi \, = \, \id_{\mathfrak{o}_{{}_{2DA}}(\QQ)} \; $
 \vskip5pt
%
%
   As a consequence,  $ \, \vartheta' := \vartheta\big|_{\lieso_n(\QQ)} \, $  and
   $ \, \xi' := \xi\big|_{\mathfrak{o}_{{}_{2DA}}(\QQ)} \, $  yield mutually inverse bijections
 $ \; \lieso_n(\QQ) \;{\buildrel \vartheta' \over
 {\lhook\joinrel\relbar\joinrel\relbar\joinrel\twoheadrightarrow}}\; \mathfrak{o}_{{}_{2DA}}(\QQ) \; $
 and
 $ \; \lieso_n(\QQ) \;{\buildrel \xi' \over {\twoheadleftarrow\joinrel\relbar\joinrel\relbar\joinrel\rhook}}\;
 \mathfrak{o}_{{}_{2DA}}(\QQ) \; $.
 Note also that, being defined over  {\sl antisymmetric\/}  matrices,  $ \vartheta' $
 is described by the following modified version of  \eqref{eq: def-theta(Psi)}:
\begin{equation}  \label{eq: def-theta'(Psi)}
   \Psi  \; \mapsto \;  \vartheta'(\Psi)  \, := \,
 d_\Psi^{-1} D A + A^{\scriptscriptstyle T}
   \big( \Psi^{\scriptscriptstyle T} \! - \Psi \big) \, A  \, = \,  \big( d_\Psi^{-1}
 D - 2 \, A^{\scriptscriptstyle T} \Psi \big) \, A
\end{equation}
 \vskip3pt
   In addition, there exists also a natural bijection
   $ \; \mathfrak{o}_{{}_{2DA}}(\QQ) \;{\buildrel \chi \over {\lhook\joinrel\relbar\joinrel\relbar\joinrel\relbar\joinrel\twoheadrightarrow}}\;
   \text{\sl $ q $--Mp}_\QQ \; $  given by  $ \, R \mapsto \chi(R) := \underline{q^R} \, $.  Using it, we can define the maps
\begin{equation}  \label{eq: twist->mp}
  \lieso_n(\QQ) \;{\buildrel \chi \circ \vartheta' \over {\lhook\joinrel\relbar\joinrel\relbar\joinrel\relbar\joinrel\relbar\joinrel\twoheadrightarrow}}\;
  \text{\sl $ q $--Mp}_\QQ  \quad ,  \qquad  \Psi \, \mapsto \, \bq_\Psi := \underline{q^{\vartheta'(\Psi)}}
\end{equation}
 and
\begin{equation}  \label{eq: mp->twist}
  \text{\sl $ q $--Mp}_\QQ \;{\buildrel \xi' \circ \chi^{-1} \over
  {\lhook\joinrel\relbar\joinrel\relbar\joinrel\relbar\joinrel\relbar\joinrel\twoheadrightarrow}}\; \lieso_n(\QQ)  \quad ,
  \qquad  \bq = \underline{q^R} \, \mapsto \, \Psi_\bq := \xi'(R)
\end{equation}
 \vskip5pt
\noindent
 that are inverse to each other, hence are bijections.
 \vskip5pt
   {\sl As a matter of notation, in the following when  $ \, \Psi \in \lieso_n(\QQ) \, $  and
   $ \, \bq \in \text{\sl $ q $--Mp}_\QQ \, $
   are such that  $ \, \bq = \big( \chi \circ \vartheta' \big)(\Psi) \, $  and  $ \, \Psi = \big( \xi' \circ \chi^{-1} \big)(\bq) \, $
   we shall write in short  $ \; \Psi \leftrightsquigarrow \bq \; $}.
 \vskip5pt
   Finally, by  Definition \ref{def-sigma},  every multiparameter  $ \, \bq \in \text{\sl $ q $--Mp}_\QQ \, $  uniquely defines a
   corresponding  $ 2 $--cocycle  $ \, \sigma = \sigma_\bq \, $
on  $ \QEqcheck \, $:
this yields a map  $ \, \bq \mapsto \sigma_\bq \, $  which is injective, hence  $ \text{\sl $ q $--Mp}_\QQ $
is in bijection with
the subset  $ {\widehat{\mathcal Z}}^{\,2}_\QQ $  of all  $ \sigma_\bq $'s  inside the set
 $ \, \mathcal{Z}^2\big(\QEqcheck,\FF\big) \, $  of all  $ 2 $--cocycles.
  Composing this bijection with the bijection  $ \; \Psi \leftrightsquigarrow \bq \; $
  we eventually get a third bijection between  $ \lieso_n(\QQ) $   --- encoding ``rational twists'' ---
  and  $ {\widehat{\mathcal Z}}^{\,2}_\QQ $   --- encoding ``rational  $ 2 $--cocycles''.
  We shall shortly denote this bijection by  $ \; \Psi \leftrightsquigarrow \sigma \; $:  explicitly, it is described by
\begin{equation}  \label{eq: twist->cocycle_I}
  \lieso_n(\QQ)
  \;{\lhook\joinrel\relbar\joinrel\relbar\joinrel\relbar\joinrel\twoheadrightarrow}\;
  {\widehat{\mathcal Z}}^{\,2}_\QQ   \;\; ,
  \quad  \Psi \, \mapsto \, \sigma  \qquad  \text{with}  \quad  \sigma\big( K_{\alpha_i} , K_{\alpha_j} \big)
  := q^{-(DA + A^T \Psi A)_{ij}}
\end{equation}
 On the other hand, we shall also consider yet another bijection, namely
\begin{equation}  \label{eq: twist->cocycle_II}
  \lieso_n(\QQ)
  \;{\lhook\joinrel\relbar\joinrel\relbar\joinrel\relbar\joinrel\relbar\joinrel\twoheadrightarrow}\;
  {\widetilde{\mathcal Z}}^{\,2}_\QQ  \quad ,
  \qquad  \Psi \, \mapsto \, \sigma \; ,  \quad  \text{with}  \quad  \sigma\big( K_{\alpha_i} ,
  K_{\alpha_j} \big) := q^{-{(A^T \Psi A)}_{ij}}
\end{equation}
where  $ \, {\widetilde{\mathcal Z}}^{\,2}_\QQ \, $  is nothing but the subset of all  $2 $--cocycles  in
$ \, \mathcal{Z}^2\big(\QEqcheck,\FF\,\big) \, $  of the form given in  \eqref{eq: twist->cocycle_II};
the inverse of this map is clearly given by
\begin{equation}  \label{eq: cocycle->twist_II}
  {\widetilde{\mathcal Z}}^{\,2}_\QQ
  \;{\lhook\joinrel\relbar\joinrel\relbar\joinrel\relbar\joinrel\twoheadrightarrow}\;  \lieso_n(\QQ)   \;\; ,
  \quad  \sigma \, \mapsto \, \Psi_\sigma := -A^{-T} S_\sigma \, A^{-1}
\end{equation}
where  $ \; S_\sigma := {\big( s_{ij} \big)}_{i, j \in I} \; $  is uniquely defined by
$ \, \sigma\big( K_{\alpha_i} , K_{\alpha_j} \big) := q^{s_{ij}} \, $.
                                                                    \par
   In the sequel, we shall denote this last bijective correspondence by  $ \; \Psi \longleftrightarrow \sigma \; $.
\end{free text}

\vskip15pt

\subsection{TwQUEA's vs.\ MpQUEA's: duality}  \label{TwQUEA-MpQUEA: dual}  \

\vskip9pt

   Roughly speaking   --- that is, up to technicalities such as dealing with finite-dimensional objects,
   or dealing with Hopf algebras in categories with a well-behaving notion of ``dual Hopf algebra'', etc.\ ---
   the two notions of ``twist element'' and of ``2-cocycle'' are, by definition, dual to each other (in Hopf-theoretical sense).
   As a consequence the two procedures of
``comultiplication twisting'' and of ``multiplication twisting''
 are dual to each other as well (see  Proposition \ref{prop: duality-deforms}  for a formalization).
                                                                      \par
  Beyond this, we can prove the following: when the Hopf algebras  $ H $  and  $ K $  are opposite (polynomial) quantum
  Borel subgroups in duality, the link between  $ \Psi $  and  $ \sigma $  is ruled precisely by the bijection
  $ \; \Psi \longleftrightarrow \sigma \; $.  In short, we can claim that
 \vskip3pt
   \centerline{\it TwQUEA's and MpQUEA's of (opposite) Borel type are dual to each other}
 \vskip3pt
\noindent
 and in this duality the correspondence  \; {\sl twists}  $ \, \rightleftarrows \, $  {\sl  $ 2 $--cocycles}  \;
 is given by the bi\-jection  $ \; \Psi \longleftrightarrow \sigma \; $.  The precise statement is the following:

\vskip11pt

\begin{theorem}  \label{dual Borel TwQUEA-MpQUEA}
 Let  $ \, U_{q,\varGamma_\pm\!}(\lieb_\pm) \! $  be opposite Borel quantum subgroups
and let also  $ \; \eta : U_{q,\varGamma_+\!}(\lieb_+)
  \mathop{\otimes}\limits_{\;\k_q} {U_{q,\varGamma_-\!}(\lieb_-)} \!\! \relbar\joinrel\relbar\joinrel\longrightarrow \k_q \; $
  be a skew-Hopf pairing
 as in  \S \ref{pol-qgroups_larg-tor}.  Given  $ \, \Psi \in \lieso_n(\QQ) $,  let
 $ \, \varGamma'_\pm := (\id + \psi_+)(\varGamma_\pm) - \psi_-(\varGamma_\mp) + Q_\pm \, $.
 For every  $ \, \sigma \in {\mathcal Z}^2 (U_{q,\varGamma'_-\!}(\lieb_-),\k_{q}) \, $,
 consider the corresponding twistings  $ \, U_{q,\varGamma'_+\!}^\Psi(\lieb_+) $
 and  $ \, {\big( U_{q,\varGamma'_-\!}(\lieb_-) \big)}_\sigma \, $.  Then the extended linear map
 $ \; \eta : U_{q,\varGamma'_+\!}^\Psi(\lieb_+) \mathop{\otimes}\limits_{\;\k_q}
 {\big( U_{q,\varGamma'_-\!}(\lieb_-) \big)}_\sigma \!\! \relbar\joinrel\longrightarrow \k_q \; $
 given in  \S \ref{pol-qgroups_larg-tor}  (see also  \S \ref{duality x larger MpQUEA's})
 with respect to  $ \, \varGamma_\bullet = \varGamma'_+ \times \varGamma'_- \, $
 is again a skew-Hopf pairing with respect to the new, deformed Hopf structures if and only if
 $ \; \Psi \longleftrightarrow \sigma \; $.  In other words, the deformed coproduct
 $ \Delta^{\scriptscriptstyle (\Psi)} \! $  on  $ U_{q,\varGamma'_+\!}(\lieb_+) $  and the deformed product
 $ \, \cdot_\sigma \, $  on  $ U_{q,\varGamma'_-\!}(\lieb_-) $  are dual to each other (via  $ \eta $)
 if and only if  $ \,\; \Psi \longleftrightarrow \sigma \; $.
                                                                           \par
   A symmetric, parallel result holds true when switching the roles of  $ \, \Psi $  and  $ \sigma $
   from left to right and viceversa.
\end{theorem}

\begin{proof}
 This is a sheer matter of computation.  Indeed, let us consider for instance the element
 $ \, \Delta^{\scriptscriptstyle (\Psi)}(E_j) \, $:  by construction, if we consider the standard
 $ Q $--grading  on  $ U_{q,\varGamma'_-\!}(\lieb_-) $  then for
 $ \, Y \in {U_{q,\varGamma'_-\!}(\lieb_-)}^{\times 2} \, $  we see that  $ \, \eta\big( E_j \, , Y \,\big) \not= 0 \, $
 can only occur with elements which actually belong to the  $ \k_q $--span  of elements of the form
 $ \, K_{\gamma_-} \cdot_\sigma F_j \, $  or  $ \, F_j\cdot_\sigma K_{\gamma_-} \, $  with
 $ \, \gamma_- \in \varGamma'_- \, $;  in fact, to simplify the notation we can assume
 $ \, \gamma_- = \alpha_i^- \, $.  Now write
 $ \, \Delta^{\scriptscriptstyle (\Psi)}(E_j) = (E_j)_{(1)}^\Psi \otimes (E_j)_{(2)}^{\Psi} =
 E_j \otimes K_{+\psi_-(\alpha_j^-)} + K_{+(\text{id} + \psi_+)(\alpha_j^+)} \otimes E_j \; $.
 Then, a direct computation gives
  $$  \displaylines{
%
%
   \qquad   \big( \eta \otimes \eta \big) \big( \Delta^{\scriptscriptstyle (\Psi)}(E_j) , K_{\alpha_i^-} \otimes F_j \big)
      \; = \;  \eta\big((E_j)_{(1)}^{\Psi} , K_{\alpha_i^-}\big) \eta\big((E_j)_{(2)}^{\Psi}, F_j \big)  \; =   \hfill  \cr
   \qquad \qquad   = \;  \eta\big( E_j,  K_{\alpha_i^-}\big) \, \eta\big( K_{+\psi_-(\alpha_j^-)},  F_j\big) \, + \,
   \eta\big(K_{+(\text{id} + \psi_+)(\alpha_j^+)}, K_{\alpha_i^-}\big) \, \eta\big(E_j \, ,  F_j \big)  \; =   \hfill  \cr
   \qquad \qquad \qquad   = \;  \eta\big( K_{+(\text{id} + \psi_+)(\alpha_j^+)} \, ,
   K_{\alpha_i^-} \big) \, \eta\big( E_j \, , F_j \big)  \;
   = \;  q^{-( (\text{id} + \psi_+)(\alpha_j^+) \,, \, {\alpha_i^-})} \, \eta\big( E_j \, , F_j \big)   \hfill  }  $$
   \indent   On the other hand, for the deformed product  $ \, \cdot_\sigma \, $  in
   $ {\big( U_{q,\varGamma'_-\!}(\lieb_-) \big)}_\sigma $  we have
  $$  \displaylines{
   \quad   K_{\alpha_i^-} \cdot_\sigma F_j  \, =
   \,  \sigma\big( {(K_{\alpha_i^-})}_{(1)} \, , {(F_j)}_{(1)} \big) \, {K_{\alpha_i^-}}_{(2)} \, {(F_j)}_{(2)} \,
   \sigma^{-1}\big( {(K_{\alpha_i^-})}_{(3)} \, , {(F_j)}_{(3)} \big)
   \, =   \hfill  \cr
   \hfill   = \,  \sigma\big( K_{\alpha_i^-} \, , {(F_j)}_{(1)} \big) \, \sigma^{-1}\big( K_{\alpha_i^-}
   \, , {(F_j)}_{(3)} \big) \,  K_{\alpha_i^-} \, {(F_j)}_{(2)}   \,
   = \,  \sigma^{-1}\big( K_{\alpha_i^-} \, , K_{\alpha_j^+}^{-1} \big) \, K_{\alpha_i^-} \, F_j   \quad  }  $$
so that
  $$  \displaylines{
   \qquad   \eta\big( E_j \, , K_{\alpha_i^-} \cdot_\sigma F_j \big)  \; = \;
   \sigma^{-1}\big( K_{\alpha_i^-} \, , K_{\alpha_j^+}^{-1} \big) \, \eta\big( E_j \, , K_{\alpha_i^-} \, F_j \big)  \; =   \hfill  \cr
   \qquad \qquad   = \;  \sigma^{-1}\big( K_{\alpha_i^-} \, , K_{\alpha_j^+}^{-1} \big) \,
   \eta\big( K_{\alpha_j^+} \, , K_{\alpha_i^-} \big) \, \eta\big( E_j \, , F_j \big)  \; =   \hfill  \cr
   \qquad \qquad \qquad   = \;  q^{-d_i a_{ij}}\, \sigma\big( K_{\alpha_i^-} \, , K_{\alpha_j^+} \big) \,
   \eta\big( E_j \, , F_j \big)  }  $$
   \indent   Comparing with the above, this means that we have
\begin{equation}  \label{eq: cond-Hopf_I}
  \big( \eta \otimes \eta \big) \big( \Delta^{\scriptscriptstyle (\Psi)}(E_j) ,
  K_{\alpha_i^-} \otimes F_j \big)  \,\; = \;\,  \eta\big( E_j \, , K_{\alpha_i^-} \cdot_\sigma F_j \big)
\end{equation}
if and only if
\begin{equation}  \label{eq: Psi<--->sigma}
  \sigma\big( K_{\alpha_i^-} \, , K_{\alpha_j^+} \big)  \; = \;  q^{\,d_i a_{ij}-( (\text{id} +
  \psi_+)(\alpha_j^+) \,, \, \alpha_i^-)}  \,\; = \;\,  q^{-(A^T \Psi A)_{ij}}
\end{equation}
and  {\sl the last condition means that}  $ \, \Psi \longleftrightarrow \sigma \, $.
Similar computations show that
\begin{equation}  \label{eq: cond-Hopf_II}
  \eta\big( \Delta^{\scriptscriptstyle (\Psi)}(E_j) \, , F_j \otimes K_{\alpha_i^-} \big)  \,\; =
  \;\,  \eta\big( E_j \, , F_j \cdot_\sigma K_{\alpha_i^-} \big)
\end{equation}
if and only if  \eqref{eq: Psi<--->sigma}  holds, again, that is if and only if
$ \, \Psi \longleftrightarrow \sigma \, $.
                                                          \par
   Furthermore, notice that  $ \; \Delta^{\scriptscriptstyle (\Psi)}\big( K_{\gamma_+} \big) =
   \Delta\big( K_{\gamma_+} \big) \; $  and
   $ \; K_{\gamma'_-} \cdot_\sigma K_{\gamma''_-} = K_{\gamma'_-} \, K_{\gamma''_-} \; $
   for all  $ \, \gamma'_\pm, \gamma''_\pm \in \varGamma_\pm \, $,
   so that we automatically have
\begin{equation}  \label{eq: cond-Hopf_III}
  \eta\big( \Delta^{\scriptscriptstyle (\Psi)}\big(K_{\gamma_+}\big) \, , K_{\gamma'_-} \otimes K_{\gamma''_-} \big)  \,\;
  = \;\,  \eta\big( K_{\gamma_+} \, , K_{\gamma'_-} \cdot_\sigma K_{\gamma''_-} \big)
\end{equation}
   \indent   Now, conditions  \eqref{eq: cond-Hopf_I},  \eqref{eq: cond-Hopf_II}  and
   \eqref{eq: cond-Hopf_III}  altogether are the conditions for
   $ \Delta^{\scriptscriptstyle (\Psi)} \! $  and  $ \, \cdot_\sigma \, $  to be
   {\sl dual to each other\/}  via  $ \eta $   --- so that the pairing  $ \eta $  itself be a
   {\sl Hopf skew-pairing\/}  w.r.t.\ the new, deformed structures.
   The above proves that the sole necessary and sufficient condition for all this is
   \eqref{eq: Psi<--->sigma},  i.e., that  $ \; \Psi \longleftrightarrow \sigma \; $,  \, as claimed.
                                                    \par
   The same argument proves also the last part of the claim, when the roles of  $ \Psi $  and  $ \sigma $  are interchanged.
\end{proof}

\vskip15pt

\subsection{TwQUEA's vs.\ MpQUEA's: correspondence}  \label{TwQUEA-MpQUEA: corr}  \

\vskip9pt

   We shall presently prove the following striking fact:  {\it the classes of TwQUEA's and of MpQUEA's associated with
   ``rational'' data actually coincide}.  More precisely, if we consider a (rational) antisymmetric twist  $ \Psi $  and a
 $ q $--rational  multiparameter  $ \bq $  such that  $ \, \Psi \leftrightsquigarrow \bq \, $,
 then any TwQUEA (over  $ \lieb_\pm \, $,  $ \lieg $  or  $ \liegd $)  with twist  $ \Psi $
 and any MpQUEA (over  $ \lieb_\pm \, $,  $ \lieg $  or  $ \liegd $) with multiparameter  $ \bq $
 are isomorphic to each other.
                                                           \par
   In other words, as each MpQUEA is a  $ 2 $--cocycle  deformation of the canonical quantum group,
   this result can be read as follows:  {\it every comultiplication twisting by a (rational)  {\sl twist}  of a canonical
   quantum group is a multiplication twisting by a (rational)  {\sl  $ 2 $--cocycle},  and viceversa},
   with the correspondence  \, {\sl twist}  $ \rightleftarrows $  {\sl  $ 2 $--\,cocycle} \;  ruled by
   $ \, \Psi \leftrightsquigarrow \sigma \; $.

\vskip13pt

   The precise statement of our main result   --- formulated here for ``double'' quantum groups ---   reads as follows:

\vskip13pt

\begin{theorem}  \label{thm: Tw-Uqgd=Mp-Uqgd}
 Let  $ \, \Psi \in \lieso_n(\QQ) \, $  and  $ \, \bq \in \text{\sl $ q $--Mp}_\QQ \, $  be such that
 $ \; \Psi \leftrightsquigarrow \bq \; $.  Let  $ M_\pm $  be any lattice in  $ \QQ{}Q $  containing  $ Q \, $,
 with  $ {\big\{ \mu^\pm_i \big\}}_{i \in I} $  any  $ \ZZ $--basis  of it, let  $ M_\pm^{\scriptscriptstyle \Psi} $
 be the sublattice of  $ \, {\QQ{}Q}^{\times 2} $  with  $ \ZZ $--basis
 $ \; {\big\{\, \varpi^\pm_i := \mu^\pm_i + \psi_\pm\big( \mu^\pm_i \big) - \psi_\mp\big( \mu^\mp_i \big) \big\}}_{i \in I} \; $
 and consider in  $ \, {\QQ{}Q}^{\times 2} $  also the lattices  $ \, M_\bullet := M_+ \times M_- \, $  and
 $ \, M_\ast^{\scriptscriptstyle \Psi} := M_+^{\scriptscriptstyle \Psi} + M_-^{\scriptscriptstyle \Psi} \, $.
                                                                      \par
   Let  $ \, U_{\bq\,,M_{\bullet}}\big(\hskip0,8pt\liegd\big) $  be the MpQUEA associated with the lattice
   $ M_{\bullet} \, $  (inside  $ \, {\QQ{}Q}^{\times 2} $),  and
   $ \, {\hat{U}}_{q\,,M_\ast^{\,\scriptscriptstyle \Psi}}^{\scriptscriptstyle \Psi}\big(\hskip0,8pt\liegd\big) \, $
   be the TwQUEA associated with  $ M_\ast^{\scriptscriptstyle \Psi} $  (cf.\ \S \ref{twist-gen's_x_twist-QUEA's}).
   Then there exists a Hopf algebra isomorphism %
 $ \,\; U_{\bq\,,M_{\bullet}}\big(\hskip0,8pt\liegd\big) \, \cong \,
 {\hat{U}}_{q\,,M_\ast^{\scriptscriptstyle \Psi}}^{\scriptscriptstyle \Psi}\big(\hskip0,8pt\liegd\big) \;\, $
 given (for  $ \, i \!\in\! I $)  by
  $$  \displaylines{
   E_i \, \mapsto \, q_i E^{\scriptscriptstyle \Psi}_i := \, q_i K_{-\psi_-(\alpha^-_i)} \, E_i  \;\; ,
   \quad  K_{\mu_i^+} \, \mapsto \, K_{+(\mu^+_i + \psi_+(\mu^+_i) - \psi_-(\mu^-_i))} \, = \, K_{+\varpi^+_i}  \cr
   L_{\mu_i^{-}} \; \mapsto \;  \, K_{-(\mu^-_i + \psi_-(\mu^-_i) - \psi_+(\mu^+_i))} \, = \, K_{-\varpi^-_i}  \;\; ,
   \;\quad  F_i \; \mapsto \; F^{\scriptscriptstyle \Psi}_i := \, K_{+\psi_+(\alpha^+_i)} \, F_i  }  $$
 In other words, letting  $ \, \sigma = \sigma_\bq \, $  be the (rational)  $ 2 $--cocycle  corresponding to
 $ \bq $  so that  $ \, \Psi \leftrightsquigarrow \sigma \, $  (as in  \S \ref{Twists<->Mp's<->2-coc's}),
 we have a Hopf algebra isomorphism
  $$  {\big( U_{\check{\bq}\,,M_{\bullet}}\big(\hskip0,8pt\liegd\big) \big)}_\sigma  \; \cong \;
  {\hat{U}}_{q\,,M_\ast^{\scriptscriptstyle \Psi}}^{\scriptscriptstyle \Psi}\big(\hskip0,8pt\liegd\big)  $$
given by the same formulas as above.
\end{theorem}

\pf
   Define an algebra map  $ \, \Phi : U_{\bq\,,M_{\bullet}}\big(\hskip0,8pt\liegd\big)
\relbar\joinrel\relbar\joinrel\longrightarrow{\hat{U}}_{q\,,M_\ast^{\scriptscriptstyle \Psi}}^{\scriptscriptstyle \Psi}\big(\hskip0,8pt\liegd\big) \, $
on the generators of  $ \, U_{\bq\,,M_{\bullet}}\big(\hskip0,8pt\liegd\big) \, $  as above.
We only have to prove that such a  $ \Phi $  is well--defined, for then it is clearly surjective. Actually,
$ \Phi $  is well-defined indeed, since the defining relations of  $ U_{\bq\,,M_{\bullet}}\big(\liegd\big) $
--- see  Definition \ref{def:multiqgroup_ang}  and  \S \ref{larger-MpQUEA's}  ---
are all mapped to zero: this follows  straightforward calculations, so we provide only some of them as guidelines.
                                                                   \par
   It is clear that  $ \Phi $  ``respects'' the commutation relations  {\it (a)\/}  and  {\it (b)\/}  in
   Definition \ref{def:multiqgroup_ang},  so now we go for the other ones.  As in  \S \ref{twist-gen's_x_twist-QUEA's},
   write
  $$  K^{\scriptscriptstyle \Psi}_{i,+}  \, := \,
  K_{(\id + \,\psi_+)(\alpha^+_i) - \psi_-(\alpha^-_i)} = K_{\alpha^+_i + \,\zeta^+_{i,+} - \,\zeta^-_{i,-}}  $$
 and
  $$  K^{\scriptscriptstyle \Psi}_{i,-}  \, := \,  K_{-( \id \! + \,\psi_-)(\alpha^-_i) +
  \psi_+(\alpha^+_i)} = K_{-\alpha^-_i - \,\zeta^-_{i,-} \! + \,\zeta^+_{i,+}}  $$
 Recall that, through the correspondence  $ \, \Psi \leftrightsquigarrow \bq \, $,  we have
  $$  \qquad   q_{ij}  \, = \,  q^{{\vartheta'(\Psi)}_{ij}}  \, = \,
  q^{\,(\alpha_i + \,\psi_+(\alpha_i) \, - \, \psi_-(\alpha_i) \, , \, \alpha_{j})}   \qquad \quad  \forall \;\; i, j \in I \;\; .  $$
 Now, by definition we have  $ \, K_{+\varpi^+_i}^{\pm 1} K_{-\varpi^-_j}^{\pm 1} =
 K_{-\varpi^-_j}^{\pm 1} K_{+\varpi^+_i}^{\pm 1} \, $  for all  $ \, i, j \in I \, $.  Moreover,
  $$  \displaylines{
   \quad   \Phi\big( K_i E_j K_i^{-1} \big)  \; = \;  q_j \, \big(K^{\scriptscriptstyle \Psi}_{i,+}\big) \,
   E^{\scriptscriptstyle \Psi}_j \big( K^{\scriptscriptstyle \Psi}_{i,+}\big)^{-1}  \; =   \hfill  \cr
   \quad \qquad   = \;  q_j \, K_{\alpha^+_i + \psi_+(\alpha^+_i) - \psi_-(\alpha^-_i)}
   K_{-\psi_-(\alpha^-_j)} E_j K^{\,-1}_{\alpha^+_i + \psi_+(\alpha^+_i) - \psi_-(\alpha^-_i)}  \; =   \hfill  \cr
   \quad \qquad \qquad   = \;  q^{(\alpha_i + \psi_+(\alpha_i) - \psi_-(\alpha_i) \, , \, \alpha_j)}
   \, q_j \, K_{-\psi_-(\alpha^-_j)} E_j  \; =
   \hfill  \cr
   \quad \qquad \qquad \qquad   = \;\;  q^{\vartheta'(\Psi)_{ij}} \, q_j \, E^{\scriptscriptstyle \Psi}_j  \;
   = \;  q_{ij} \, q_j \, E^{\scriptscriptstyle \Psi}_j  \; = \;  q_{ij} \, \Phi(E_j)
   \hfill
 }  $$
  $$  \displaylines{
   \quad   \Phi\big( L_i E_j L_i^{-1} \big)  \; = \;  q_j \, \big(K^{\scriptscriptstyle \Psi}_{i,-}\big) \,
   E^{\scriptscriptstyle \Psi}_j \big( K^{\scriptscriptstyle \Psi}_{i,-} \big)^{-1}  \; =   \hfill  \cr
   \quad \qquad  = \;  q_j \, K_{-\alpha^-_i - \psi_-(\alpha^-_i) + \psi_+(\alpha^+_i)}
   K_{-\psi_-(\alpha^-_j)} E_j K^{\,-1}_{-\alpha^-_i - \psi_-(\alpha^-_i) + \psi_+(\alpha^+_i)}  \; =   \hfill  \cr
   \quad \qquad \qquad   = \;  q^{(-\alpha_i - \psi_-(\alpha_i) + \psi_+(\alpha_i) \, , \, \alpha_j)} \,
   q_j \, K_{-\psi_-(\alpha^-_j)} E_j  \; \,{\buildrel \circledast \over =}   \hfill  \cr
   \quad \qquad \qquad \qquad   {\buildrel \circledast \over =}\, \;
   q^{-(\alpha_j + \psi_+(\alpha_j) - \psi_-(\alpha_j) \, , \, \alpha_i)} \, q_j \, K_{-\psi_-(\alpha^-_j)} E_j  \; =   \hfill  \cr
   \quad \qquad \qquad \qquad \qquad   = \;  q^{-\vartheta'(\Psi)_{ji}} \, q_j \, E^{\scriptscriptstyle \Psi}_j  \; =
   \;  q_{ji}^{-1} \, q_j \, E^{\scriptscriptstyle \Psi}_j  \; = \;  q_{ji}^{-1} \, \Phi(E_j)   \hfill  }  $$
where
 the equality  $ \, {\buildrel \circledast \over =} \, $  follows from  Lemma \ref{psi_pm-psi_mp-antisym}.
 This proves that  $ \Phi $  ``respects'' the commutation relations  {\it (c)\/}  in  Definition \ref{def:multiqgroup_ang};
 similar computations prove the same for relations  {\it (d)\/}  as well.  To check the relations  {\it (e)},
 we first observe that
  $$ \big( \psi_+(\alpha_j) , \alpha_i \big)  \,\; = \;\,  {\textstyle \sum_{k,\ell}} \, \psi_{k,\ell} \, a_{\ell,j} \, a_{k,i}  \,\; =
  \;\,  \big( \alpha_j , \psi_-(\alpha_i) \big)   \qquad   \text{\ for all \ } i, j \in I\ .  $$
 Then, for all  $ \, i, j \in I \, $  we have
\begin{align*}
   \big[ \Phi(E_i) \, , \Phi(F_j) \big]  &  \; = \;  q_i \big[ E^{\scriptscriptstyle \Psi}_i ,
   F^{\scriptscriptstyle \Psi}_j \big]  \; = \;  q_i \big[ K_{-\psi_-(\alpha^-_i)} \, E_i \, ,
   K_{+\psi_+(\alpha^+_j)} \, F_j \big]  \\
     &  \; = \;  q_i \, \big( K_{-\psi_-(\alpha^-_i)} \, E_i \, K_{+\psi_+(\alpha^+_j)} \, F_j  -
     K_{+\psi_+(\alpha^+_j)} \, F_j \, K_{-\psi_-(\alpha^-_i)} \, E_i \big)  \\
     &  \; = \;  q_i \, K_{-\psi_-(\alpha^-_i)} \, K_{+\psi_+(\alpha^+_j)} \, \big( q^{-(\alpha_i , \psi_+(\alpha_j))} E_i \, F_j -
     q^{-(\alpha_j , \psi_-(\alpha_i))} F_j \, E_i \big)  \\
\end{align*}
\begin{align*}
   &  \; = \;  q_i \, q^{-(\alpha_i , \psi_+(\alpha_j))} \, K_{-\psi_-(\alpha^-_i)} \,
   K_{+\psi_+(\alpha^+_j)} \; \delta_{ij} \, {{\, K_{\alpha^+_i} - K_{\alpha^-_i}^{-1} \,}
\over {\, q_i - q_i^{-1} \,}}  \\
   &  \; = \;  \delta_{ij} \, q_i \, q^{-(A^T \Psi A)_{ii}} \, {{\,\; K_{\alpha^+_i + \psi_+(\alpha^+_i) -
   \psi_-(\alpha^-_i)} - K_{-\alpha^-_i + \psi_+(\alpha^+_i) - \psi_-(\alpha^-_i)} \;\,} \over{\,\; q_i - q_i^{-1} \;\,}}  \\
   &  \; = \;  \delta_{ij} \, q_i^2 \, {{\,\; K^{\scriptscriptstyle \Psi}_{i,+} -
   K^{\scriptscriptstyle \Psi}_{i,-} \;\,} \over {\,\; q_i^2 -1 \;\,}}  \; = \;  \Phi\big( [E_i , F_j] \big)
\end{align*}
 Finally, for the quantum Serre relations  {\it (f)\/}  and  {\it (g)},  we use that for all $ \, m \in \NN \, $
 we have formal identities involving  $ q $--numbers  and  $ q^{\frac{1}{2}} $--numbers,  namely
  $$  {(m)}_q  \, = \,  \frac{\, q^m - 1 \,}{\, q - 1 \,}  \, = \,  q^{\frac{m-1}{2}} \,
\frac{\, q^{\frac{m}{2}} - q^{-\frac{m}{2}} \,}{\, q^{\frac{1}{2}} - q^{-\frac{1}{2}} \,}  \, =
\, q^{\frac{m-1}{2}} \, {[m]}_{q^{\frac{1}{2}}} \;\; ,   \qquad   {\bigg( {m \atop k} \bigg)}_{\!\!q} =
\, q^{\frac{-k^{2}+km}{2}} {\bigg[ {m \atop k} \bigg]}_{\!q^{\frac{1}{2}}}  $$
 and the identities  $ \; \big( (\psi_+ - \psi_-)(\alpha_i) , \alpha_j \big) = -\big( \alpha_i \, ,
 (\psi_+ - \psi_-)(\alpha_j) \big) \; $,  $ \; \big( \psi_+(\alpha_i) , \alpha_j \big) =
 \big( \alpha_i \, , \psi_-(\alpha_j) \big) \; $  and  $ \; \big( \psi_-(\alpha_i) , \alpha_i \big) = 0 \; $  for all  $ \; i, j \in I \; $.
 \vskip3pt
   It is not hard to see that  $ \Phi $  is  also a  {\sl Hopf algebra map\/}  too: for example, we have
\begin{align*}
  \com^{\!(\scriptscriptstyle \Psi)}\!\big( \Phi(E_i) \big)  &  \; = \;  q_i \,
  \com^{\!(\scriptscriptstyle \Psi)}\big( E^{\scriptscriptstyle \Psi}_i \big)  \, =
  \,  q_i \,  \com^{\!(\scriptscriptstyle \Psi)}\big( K_{-\psi_-(\alpha_i^-)} E_i \big)  \, =
  \,  q_i \,  \com^{\!(\scriptscriptstyle \Psi)}\big( K_{-\psi_-(\alpha_i^-)} \big) \, \com^{(\scriptscriptstyle \Psi)}(E_i)  \\
  &  \; = \;  q_i \, K_{-\psi_-(\alpha_i^-)} E_i \otimes 1 \, + K_{\alpha_i^+ + \psi_+(\alpha_i^+) - \psi_-(\alpha_i^-)}
  \otimes q_i \, K_{-\psi_-(\alpha_i^-)} E_i  \; =  \\
  &  \; = \;  q_i \, E^{\scriptscriptstyle \Psi}_i \otimes 1 \, + \, K^{\scriptscriptstyle \Psi}_{i,+}
  \otimes q_i \, E^{\scriptscriptstyle \Psi}_i  \; = \;  (\Phi \otimes \Phi)\big(\com(E_i)\big)
\end{align*}
 \noindent
 and  $ \; \eps^{(\scriptscriptstyle \Psi)}\big( \Phi(E_i) \big) \, = \, q_i \,
 \eps^{(\scriptscriptstyle \Psi)}\big( E^{\scriptscriptstyle \Psi}_i \big) \, = \,
 \eps\big( q_i \, K_{-\psi_-(\alpha_i^-)} E_i \big) \, = \, 0 \, = \, \eps(E_i) \; $  for all  $ \, 1 \! \leq \! i \! \leq \! n \, $.
 \vskip5pt
   Now define an algebra map
   $ \; \Phi' : {\hat{U}}_{q\,,M_\ast^{\scriptscriptstyle \Psi}}^{\scriptscriptstyle \Psi}\big(\hskip0,8pt\liegd\big)
   \relbar\joinrel\relbar\joinrel\longrightarrow U_{\bq\,,M_{\bullet}}\big(\hskip0,8pt\liegd\big) \; $  by
  $ \; \Phi' \big( E^{\scriptscriptstyle \Psi}_i \big) := q_i^{-1} \, E_i \, $,
  $ \; \Phi'\big(F_i^{\scriptscriptstyle \Psi}\big) := F_i \, $,
  $ \; \Phi'\Big(\! \big( K_{+\varpi^+_i}\big)^{\pm 1} \Big) := K_{\mu_i^+}^{\pm 1} \; $  and
  $ \; \Phi'\Big(\! \big( K_{-\varpi^-_i}\big)^{\pm 1} \Big) := L_{\mu_i^-}^{\pm 1} \, $,  \, for all $ \, 1 \leq i \leq n \, $.
 By means of calculations quite similar to the previous ones, one proves that such a  $ \Phi' $  is well-defined,
 and definitions clearly yield
 $ \, \Phi' \circ \Phi = \id_{U_{\bq\,,M_{\bullet}}(\hskip0,8pt\liegd)} \, $  and
 $ \, \Phi \circ \Phi' = \id_{ {\hat{U}}_{q\,,M_\ast^{\scriptscriptstyle \Psi}}^{\scriptscriptstyle \Psi}(\hskip0,8pt\liegd)} \, $.
 So in the end  $ \,  {\hat{U}}_{q\,,M_\ast^{\scriptscriptstyle \Psi}}^{\scriptscriptstyle \Psi}\big(\hskip0,8pt\liegd\big)
 \cong U_{\bq\,,M_{\bullet}}\big(\hskip0,8pt\liegd\big) \, $  as Hopf algebras.
\epf

\vskip3pt

   Similar arguments as those used in the proof above lead to a simpler version of
   Theorem \ref{thm: Tw-Uqgd=Mp-Uqgd}  for Borel MpQUEA's and Borel TwQUEA's:

\vskip17pt

\begin{prop}  \label{prop: Tw-Uqb=Mp-Uqb}
  Let  $ \, \Psi \in \lieso_n(\QQ) \, $  and  $ \, \bq \in \text{\sl $ q $--Mp}_\QQ \, $  be such that
  $ \; \Psi \leftrightsquigarrow \bq \; $.  Let  $ M $  be any lattice in  $ \QQ{}Q $  containing  $ Q \, $,
  with  $ {\big\{ \mu_i \big\}}_{i \in I} $  any  $ \ZZ $--basis  of it, and let  $ M_\pm^{\scriptscriptstyle (\Psi)} $
  be the sublattice of  $ \, \QQ{}Q \times \QQ{}Q $ with  $ \ZZ $--basis
  $ \; {\big\{\, \varpi^\pm_i := \mu^\pm_i + \psi_\pm\big( \mu^\pm_i \big) - \psi_\mp\big( \mu^\mp_i \big) \big\}}_{i \in I} \; $.
                                                                \par
   Let  $ \, U_{\bq\,,M}(\hskip0,8pt\lieb_\pm) $  be the (positive/negative) Borel MpQUEA
   associated with  $ M $,  and
   $ \, {\hat{U}}_{q\,,M_\pm^{\scriptscriptstyle (\Psi)}}\hskip-27pt{}^{\scriptscriptstyle \Psi}\hskip21pt(\hskip0,8pt\lieb_\pm) \, $
   be the (positive/negative) Borel TwQUEA associated with  $ M_\pm^{\scriptscriptstyle (\Psi)} $
 (cf.\ \S \ref{twist-gen's_x_twist-QUEA's}).
 Then there exist Hopf algebra isomorphisms
  $$  U_{\bq\,,M}(\hskip0,8pt\lieb_+)  \; \cong \;
  {\hat{U}}_{q\,,M_+^{\scriptscriptstyle (\Psi)}} \hskip-27pt{}^{\scriptscriptstyle \Psi}\hskip21pt(\hskip0,8pt\lieb_+)
  \quad ,  \qquad
      U_{\bq\,,M}(\hskip0,8pt\lieb_-)  \; \cong \;
      {\hat{U}}_{q\,,M_-^{\scriptscriptstyle (\Psi)}} \hskip-27pt{}^{\scriptscriptstyle \Psi}\hskip21pt(\hskip0,8pt\lieb_-)  $$
\,respectively given by
  $$  \displaylines{
   E_i \; \mapsto \; E^{\scriptscriptstyle \Psi}_i \, := \, K_{-\psi_-(\alpha^-_i)} \, E_i  \;\; ,  \;\quad
   K_{\mu_i} \; \mapsto \; K^{\scriptscriptstyle \Psi}_{+\mu_{i} } \, := \, K_{+(\mu^+_i + \psi_+(\mu^+_i) - \psi_-(\mu^-_i))} \,
   = \, K_{+\varpi^+_i}  \cr
   L_{\mu_i} \; \mapsto \; K^{\scriptscriptstyle \Psi}_{-\mu_i} \, := \, K_{-(\mu^-_i + \psi_-(\mu^-_i) - \psi_+(\mu^+_i))} \,
   = \, K_{-\varpi^-_i}  \;\; ,  \;\quad  F_i \; \mapsto \; F^{\scriptscriptstyle \Psi}_i \, := \, K_{+\psi_+(\alpha^+_i)} \, F_i  }  $$
 In other words, letting  $ \, \sigma = \sigma_\bq \, $  be the (rational)  $ 2 $--cocycle
 corresponding to  $ \bq $  as in  \S \ref{Twists<->Mp's<->2-coc's},  so that
$ \, \Psi \leftrightsquigarrow \sigma \, $,  we have Hopf algebra isomorphisms
  $$  {\big( U_{q\,,M}(\hskip0,8pt\lieb_+) \big)}_\sigma  \; \cong \;
   {\hat{U}}_{q\,,M_+^{\scriptscriptstyle (\Psi)}}\hskip-27pt{}^{\scriptscriptstyle \Psi}\hskip21pt(\hskip0,8pt\lieb_+)  \quad ,
  \qquad
      {\big( U_{q\,,M}(\hskip0,8pt\lieb_-) \big)}_\sigma  \; \cong \;
    {\hat{U}}_{q\,,M_-^{\scriptscriptstyle (\Psi)}}\hskip-27pt{}^{\scriptscriptstyle \Psi}\hskip21pt(\hskip0,8pt\lieb_-)  $$
given by the same formulas as above.  In particular, when
$ \, M \supseteq Q_\pm + \psi_\pm(Q_\pm) + \psi_\mp(Q_\mp) \, $  we have
$ \, \hat{U}_{\!\!q\,,M_\pm^{\scriptscriptstyle (\Psi)}}^{\;\scriptscriptstyle \Psi}(\hskip0,8pt\lieb_\pm) =
U_{q\,,M}^{\;\scriptscriptstyle \Psi}(\hskip0,8pt\lieb_\pm) \, $,  hence the isomorphisms above read
  $$  {\big( U_{q\,,M}(\hskip0,8pt\lieb_\pm) \big)}_\sigma  \; \cong \;\, U_{\bq\,,M}(\hskip0,8pt\lieb_\pm)
  \,\; \cong \;\,  U_{q\,,M}^{\;\scriptscriptstyle \Psi}(\hskip0,8pt\lieb_\pm)  $$
                     \par
   A similar claim holds true as well for  {\sl twisted}  quantum Borel subalgebras and
   {\sl multiparameter}  quantum Borel subalgebras inside the MpQUEA's  $ \, \upsiqgdG{M} \, $.   \qed
\end{prop}
\vskip11pt

   Finally, we shall presently see that any TwQUEA of type  $ \upsiqgG{\varGamma} $
   can be seen as a (still to define)  ``$ q $--rational  MpQUEA over  $ \lieg \, $''.
   We begin by defining the latter.

\vskip13pt

\begin{definition}  \label{def: Ubq(g)}
 Let  $ \, \Psi \in \lieso_n(\QQ) \, $  and  $ \, \bq \in \text{\sl $ q $--Mp}_\QQ \, $
 be such that  $ \; \Psi \leftrightsquigarrow \bq \; $,  and let  $ \, \psi_+ \, $
 be the map associated with  $ \Psi $  (cf\  \S \ref{root-twisting}).
 Let  $ M_\pm $  be lattices in  $ \QQ{}Q $  containing  $ Q $  such that
 $ \, \psi_+(M_\pm) \subseteq M_\pm \, $;  denote by  $ {\big\{ \mu^\pm_i \big\}}_{i \in I} $
 any  $ \ZZ $--basis  of it.  Let  $ \, \II_{\scriptscriptstyle \Psi} $  be the two-sided ideal of
 $ \, U_{\bq\,,M_\bullet}\big(\hskip0,8pt\liegd\big) \, $  generated by the elements
 \vskip4pt
   \centerline{ \qquad \qquad   $ K_{\mu_i^+} \, L_{\mu_i^-} \, - \,  K_{2\,\psi_+(\mu_i^+)} \,
   L_{2\,\psi_-(\mu_i^-)}   \qquad \qquad   \big(\, i \in I \,\big) \; $.}
\vskip4pt
   Then we denote by  $ \, U_{\bq\,,M_\bullet}(\hskip0,8pt\lieg) \, $
   --- that we loosely call {\it ``MpQUEA over  $ \lieg $''}  ---   the quotient Hopf algebra
  $$ U_{\bq\,,M_\bullet}(\hskip0,8pt\lieg) \; := \;
  U_{\bq\,,M_{\bullet}}\big(\hskip0,8pt\liegd\big) \Big/ \II_{\scriptscriptstyle \Psi}  $$
\end{definition}

\vskip7pt

\begin{rmk}  \label{rmk:Ug=Uhatg}
 If  $ M $  is any lattice in  $ \QQ{}Q $ containing  $ Q $
 and such that  $ \psi_+(M) \subseteq M \, $,  then for  $ \, M_+ := M =: M_- \, $
 we write  $ \, M_\bullet := M_+ \times M_- \, $.  In particular,
 $ \, M =  M_+ + M_- = M^{\scriptscriptstyle \Psi} = M + \psi_+(M) + \psi_-(M) \, $.
                                                              \par
   Now, under the assumption  $ \, \psi_+(M) \subseteq M \, $  for the lattice in
   $ \QQ{}Q \, $,  we have that  $ \; \hat{U}_{q,M}^{\scriptscriptstyle \,\Psi}(\hskip0,8pt\lieg) =
   U_{q,M}^{\scriptscriptstyle \,\Psi}(\hskip0,8pt\lieg) \; $.  In fact,
   $ \hat{U}_{q,M}^{\scriptscriptstyle \,\Psi}(\hskip0,8pt\liegd) $  is the Hopf subalgebra of
   $ U_{q,M}^{\scriptscriptstyle \,\Psi}(\hskip0,8pt\liegd) $  generated by the elements
   $ E_i^{\scriptscriptstyle \Psi} := K_{-\psi_-(\alpha_i^-)}E_i \, $,  $ F_i^{\scriptscriptstyle \Psi} :=
   K_{+\psi_+(\alpha_i^+)}F_i \, $,  $ K_{\varpi_i^+} := K_{\mu_i^+ + \psi_+(\mu_i^+) - \psi_-(\mu_i^-)} \, $
   and  $ \, K_{\varpi_i^-} := K_{\mu_i^- + \psi_-(\mu_i^-) - \psi_+(\mu_i^+)} \, $  for all  $ i \in I \, $.
   Besides, by Remark \ref{rem: UpsiqgdG ---> UpsiqgG},  $ \hat{U}_{q,M}^{\scriptscriptstyle \,\Psi}(\hskip0,8pt\lieg) $
   is the image of the Hopf algebra epimorphism
   $ \; \hat{\pi}_{\lieg}^{\scriptscriptstyle \,\Psi} : \hat{U}_{q,M}^{\scriptscriptstyle \,\Psi}(\hskip0,8pt\liegd)
   \longrightarrow \hat{U}_{q,M}^{\scriptscriptstyle \,\Psi}(\hskip0,8pt\lieg) \; $  given by
   $ \, \hat{\pi}_\lieg^{\scriptscriptstyle \Psi}\big( E_i^{\scriptscriptstyle \Psi} \big) := E_i^{\scriptscriptstyle \Psi} \, $,
 $ \, \hat{\pi}_\lieg^{\scriptscriptstyle \Psi}\big( K_{\varpi_i^+} \big) := K_{\mu_i + 2 \psi_+(\mu_i)} \, $,
 $ \, \hat{\pi}_\lieg^{\scriptscriptstyle \Psi}\big( K_{\varpi_i^-} \big) := K_{\mu_i - 2 \psi_+(\mu_i)} \, $
 and  $ \, \hat{\pi}_\lieg^{\scriptscriptstyle \Psi}\big( F_i^{\scriptscriptstyle \Psi} \big) := F_i^{\scriptscriptstyle \Psi} \, $
 for all  $ \, i \in I \, $;  therefore  $ \, K_{\mu_i}^{\,2} =
 \hat{\pi}_\lieg^{\scriptscriptstyle \Psi}\big( K_{\varpi_i^+} K_{\varpi_i^-} \big) \in
 \hat{U}_{q,M}^{\scriptscriptstyle \,\Psi}(\hskip0,8pt\lieg) \, $.
 This implies that  $ \, K_{2 \psi_+(\mu_i)} \! \in \hat{U}_{q,M}^{\scriptscriptstyle \,\Psi}(\hskip0,8pt\lieg) \, $
 and consequently  $ \, K_{\mu_i} \! \in \hat{U}_{q,M}^{\scriptscriptstyle \,\Psi}(\hskip0,8pt\lieg) \, $
 for all  $ \, i \in I \, $;  hence, the toral part (the group of group-like elements) of both Hopf algebras coincide.
 Since the other generators of  $ \, \hat{U}_{q,M}^{\scriptscriptstyle \,\Psi}(\hskip0,8pt\lieg) \, $
 differ only by a group-like element, the two Hopf algebras
 $ \hat{U}_{q,M}^{\scriptscriptstyle \,\Psi}(\hskip0,8pt\lieg) $  and
 $ U_{q,M}^{\scriptscriptstyle \,\Psi}(\hskip0,8pt\lieg) $  must coincide.
\end{rmk}

\vskip11pt

   We are now ready for the result comparing TwQUEA's and MpQUEA's ``over  $ \lieg \, $'':

\vskip15pt

\begin{theorem}  \label{thm: Tw-Uqg=Mp-Uqg}
 Let  $ \, \Psi \in \lieso_n(\QQ) \, $  and  $ \, \bq \in \text{\sl $ q $--Mp}_\QQ \, $
 be such that  $ \; \Psi \leftrightsquigarrow \bq \; $.  Let  also $ M $  be any lattice in
 $ \QQ{}Q $  containing  $ Q $  and such that  $ \, \psi_+(M) \subseteq M \, $.
 Then there exists a Hopf algebra isomorphism
  $$ U_{\bq\,,M_\bullet}(\hskip0,8pt\lieg)  \,\; \cong \;\,  U_{q,M}^{\scriptscriptstyle \,\Psi}(\hskip0,8pt\lieg)  $$
 given by the formulas  $ \,\; K_{\mu_i} \longleftrightarrow K_{+(\id + \psi_+ - \psi_-)(\mu_i)} \;\, $,
 $ \,\; L_{\mu_i} \longleftrightarrow K_{-(\id + \psi_- - \psi_+)(\mu_i)} \;\, $,
 $ \,\; E_i \, \longleftrightarrow \, q_i \, K_{-\psi_-(\alpha_i)} \, E_i \;\, $  and
 $ \,\; F_i \longleftrightarrow K_{+\psi_+(\alpha_i)} \, F_i \;\, $,  \, for all  $ \, i \in I \, $.
                                              \par
   In other words, letting  $ \, \sigma = \sigma_\bq \, $  be the (rational)  $ 2 $--cocycle
   corresponding to  $ \bq $  as in  \S \ref{Twists<->Mp's<->2-coc's},  so that
   $ \, \Psi \leftrightsquigarrow \sigma \, $,  we have a Hopf algebra isomorphism
  $$  {\big( U_{q\,,M_\bullet}(\hskip0,8pt\lieg) \big)}_\sigma  \; \cong \;
  U_{q\,,M}^{\scriptscriptstyle \,\Psi}(\hskip0,8pt\lieg)  $$
given by the same formulas as above.
 \vskip5pt
   {\sl $ \underline{\text{Remark}} $:\,  as we assume  $ \Psi $  to be antisymmetric,
   which is  {\it not}  restrictive, we can re-cast the formulas above as
   $ \,\; K_{\mu_i} \longleftrightarrow K_{+(\id + 2 \psi_+)(\mu_i)} \;\, $,
   $ \,\; L_{\mu_i} \longleftrightarrow K_{-(\id + 2 \psi_-)(\mu_i)} \;\, $,
   $ \,\; E_i \, \longleftrightarrow \, q_i \, K_{+\psi_+(\alpha_i)} \, E_i \;\, $  and
   $ \,\; F_i \longleftrightarrow K_{-\psi_-(\alpha_i)} \, F_i \;\, $,  \, for all  $ \, i \in I \, $}.
\end{theorem}

\pf
 Let  $ {\big\{ \mu_i \big\}}_{i \in I} $  be any  $ \ZZ $--basis of  $ M $  and denote by
 $ {\big\{ \mu_i^{\pm} \big\}}_{i \in I} $  the corresponding  $ \ZZ $--basis of
 $ M_\pm \, $.  From  Remark \ref{rem: UpsiqgdG ---> UpsiqgG}  we know that
 $ \hat{U}_{q,M}^{\scriptscriptstyle \,\Psi}(\hskip0,8pt\lieg) $
 is the image of the Hopf algebra epimorphism
 $ \; \hat{\pi}_\lieg^{\scriptscriptstyle \Psi} : \hat{U}_{q,M}^{\scriptscriptstyle \,\Psi}\big(\hskip0,8pt\liegd\big)
 \relbar\joinrel\relbar\joinrel\relbar\joinrel\twoheadrightarrow \hat{U}_{q,M}^{\scriptscriptstyle \,\Psi}(\hskip0,8pt\lieg) \, $,
 where the latter Hopf algebra  equals  $ U_{q,M}^{\scriptscriptstyle \,\Psi}(\hskip0,8pt\lieg) $  by
 Remark \ref{rmk:Ug=Uhatg}.
                                                                          \par
   On the other hand, thanks to  Theorem \ref{thm: Tw-Uqgd=Mp-Uqgd}  we have a Hopf algebra
   isomorphism
   $ \; \Phi : U_{\bq\,,M_\bullet}\big(\hskip0,8pt\liegd\big)
   \lhook\joinrel\relbar\joinrel\relbar\joinrel\relbar\joinrel\twoheadrightarrow
   \hat{U}_{q\,,M}^{\scriptscriptstyle \,\Psi}\big(\hskip0,8pt\liegd\big) \; $.
   Hence, under the conditon that
   $ \, \big( \hat{\pi}_{\lieg}^{\scriptscriptstyle \Psi} \circ \Phi \big)(\II_{\scriptscriptstyle \Psi}) = 0 \, $,
   there exists a surjective Hopf algebra morphism  $ \; \varphi : U_{\bq\,,M_\bullet}(\hskip0,8pt\lieg)
 \relbar\joinrel\relbar\joinrel\relbar\joinrel\twoheadrightarrow U_{q,M}^{\scriptscriptstyle \,\Psi}(\hskip0,8pt\lieg) \; $
 defined on the generators by
\begin{align*}
  \varphi\big(  K_{\mu_i} \big)  &  \, := \,  K_{+(\id + \psi_+ - \psi_-)(\mu_i)}  =  K_{+(\id + 2\psi_+)(\mu_i)}  \;\; ,
  &  \quad  \varphi(E_i)  &  \, := \,  q_i \, K_{-\psi_-(\alpha_i)} \, E_i  \\
  \varphi\big( L_{\mu_i} \big)  &  \, := \,  K_{-(\id + \psi_- - \psi_+)(\mu_i)}  =  K_{-(\id + 2\psi_-	 )(\mu_i)}  \;\; ,  &
  \quad  \varphi(F_i)  &  \, = \,  K_{+\psi_+(\alpha_i)} \, F_i
\end{align*}
 for all  $ \, i \in I \, $,  where we denote by  $ K_{\mu_i} \, $,  $ L_{\mu_i} \, $, $ E_i $  and  $ F_i $
 the generators of  $ \, U_{\bq\,,M_\bullet}(\hskip0,8pt\lieg) \, $.
                                                \par
   Now, the fact that  $ \, \big( \hat{\pi}_\lieg^{\scriptscriptstyle \Psi} \circ \Phi \big)(\II_{\scriptscriptstyle \Psi}) = 0 \, $
   follows by direct computation.  In fact, since by assumption  $ M $  contains  $ Q $  and  $ \psi_+(M) \, $,
   and  $ \Psi $  is antisymmetric, we have that
   $ \, \big( \hat{\pi}_\lieg^{\scriptscriptstyle \Psi} \circ \Phi \big)\big( K_{\mu_i^+} \big) =
   \hat{\pi}_\lieg^{\scriptscriptstyle \Psi}\big( K_{+\varpi_i^+} \big) =  K_{\mu_i + 2\psi_+(\mu_{i})} \, $
   and  $ \, \big( \hat{\pi}_\lieg^{\scriptscriptstyle \Psi} \circ \Phi \big)\big( L_{\mu_i^-} \big) =
   \hat{\pi}_\lieg^{\scriptscriptstyle \Psi}\big( K_{-\varpi_i^-} \big) = K_{-(\id + 2\psi_-	 )(\mu_i)} \, $.
   Therefore, we have also
  $$  \big( \hat{\pi}_\lieg^{\scriptscriptstyle \Psi} \circ \Phi \big)\big( K_{\mu_i^+} \, L_{\mu_i^-} \big)  \;
  = \;  \big( \hat{\pi}_\lieg^{\scriptscriptstyle \Psi} \circ \Phi \big)\big( K_{2\psi_+ (\mu_i^+)} \, L_{2\psi_- (\mu_i^-)} \big)  $$
   \indent   Conversely, let  $ \; \hat{\varphi} : U_{q,M}^{\scriptscriptstyle \,\Psi}(\hskip0,8pt\lieg)
   \relbar\joinrel\longrightarrow U_{\bq\,,M_\bullet}(\hskip0,8pt\lieg) \; $  be the algebra morphism given by
  $$  \hat{\varphi}\big( K_{\mu_i} \big)  \, := \, K_{\mu_i^+ - \psi_+(\mu_i^+)} \, L_{\psi_+(\mu_i^-)} \;\;\; ,
  \;\quad  \hat{\varphi}\big( E_i^{\scriptscriptstyle \Psi} \big) \, := \,  q_i^{-1} \, E_i \;\;\; ,  \;\quad
  \hat{\varphi}\big( F_i^{\scriptscriptstyle \Psi} \big) \, := \,  F_i  $$
for all  $ \, i \in I \, $,  where  $ \, E_i^{\scriptscriptstyle \Psi} := q_i \, K_{-\psi_-(\alpha_i)} \, E_i \, $
and  $ \, F_i^{\scriptscriptstyle \Psi} := K_{+\psi_+(\alpha_i)} \, F_i \, $  are the images in
$ U_{\bq\,,M_\bullet}(\hskip0,8pt\lieg) $  of the same name elements in
$ U_{\bq\,,M_\bullet}\big(\hskip0,8pt\liegd\big) \, $.  With this definition we have
\begin{equation}  \label{eq: defvarphiK}
  \begin{matrix}
     \hat{\varphi}\big( \hat{\pi}_\lieg^{\scriptscriptstyle \Psi}\big( K_{\varpi_i^+} \big) \big) \, =
     \,  \hat{\varphi}\big( K_{(\id + 2\psi_+)(\mu_i)} \big) \, = \, K_{+\mu_i^+}  \\
     \hat{\varphi}\big( \hat{\pi}_\lieg^{\scriptscriptstyle \Psi}\big( K_{\varpi_i^-} \big) \big) \, =
     \,  \hat{\varphi}\big( K_{(\id + 2\psi_-)(\mu_i)} \big)  \, = \,  L_{-\mu_i^-}
  \end{matrix}
\end{equation}
 Indeed, since  $ \, K_{\mu_j}^{\,2} = \hat{\pi}_\lieg^{\scriptscriptstyle \Psi}\big( K_{\varpi_j^+} \, K_{\varpi_j^-} \big) \, $
 for all  $ \, j \in I \, $,  setting $ \, \psi_+(\mu_i) := \sum_{j \in I} m_{ji} \mu_j \, $  yields
 $ \; \hat{\pi}_\lieg^{\scriptscriptstyle \Psi}\big( K_{\varpi_i^+} \big) \, =
 \, K_{\mu_i + 2 \psi_+(\mu_i)} \, = \, K_{\mu_i} \, \prod_{j \in I} \! {\big( K_{\mu_j}^{\,2} \big)}^{m_{ji}} \; $.
 Therefore
\begin{align*}
  \hat{\varphi}\big( \hat{\pi}_\lieg^{\scriptscriptstyle \Psi}\big( K_{\varpi_i^+} \big) \big)  & \; =
  \;  \hat{\varphi}\Big( K_{\mu_i} \, {\textstyle \prod_{j \in I}} {\big( K_{\mu_j}^{\,2} \big)}^{m_{ji}} \Big)  \; =  \\
    &  \; = \;  K_{\mu_i^+ - \psi_+(\mu_i^+)} \, L_{\psi_+(\mu_i^-)} \cdot {\textstyle \prod_{j \in I}} K_{\mu_j^+ - \psi_+(\mu_j^+)}^{\,2\,m_{ji}} \cdot {\textstyle \prod_{j \in I}} L_{\psi_+(\mu_j^-)}^{\,2\,m_{ji}}  \; =  \\
\end{align*}
\begin{align*}
    &  \; = \;  K_{\mu_i^+ - \psi_+(\mu_i^+) + 2\,\psi_+(\mu_i^+)} \, L_{\psi_+(\mu_i^-)}
    \cdot {\textstyle \prod_{j \in I}} K_{-\psi_+(\mu_j^+)}^{\,2\,m_{ji}}
    \cdot {\textstyle \prod_{j \in I}} L_{\psi_+(\mu_j^-)}^{\,2\,m_{ji}}  \; =  \\
    &  \; = \;  K_{\mu_i^+ + \psi_+(\mu_i^+)} \, L_{\psi_+(\mu_i^-)}
    \cdot {\textstyle \prod_{j \in I}} {\big( K_{-2\,\psi_+(\mu_j^+)} \, L_{-2\,\psi_-(\mu_j^-)} \big)}^{\,m_{ji}}  \; =  \\
    &  \; = \;  K_{\mu_i^+ + \psi_+(\mu_i^+)} \, L_{\psi_+(\mu_i^-)}
    \cdot {\textstyle \prod_{j \in I}} {\big( K_{\mu_j^+} \, L_{\mu_j^-} \big)}^{-m_{ji}}  \; =  \\
    &  \; = \;  K_{\mu_i^+ + \psi_+(\mu_i^+)} \, L_{\psi_+(\mu_i^-)} \,
    K_{-\psi_+(\mu_i^+)} \, L_{-\psi_+(\mu_i^-)}  \; = \;  K_{\mu_i^+}
\end{align*}
 Similarly, one may check the equalities in  \eqref{eq: defvarphiK}.
 In particular,  $ \hat{\varphi} $  is surjective.
 \vskip4pt
   Since  $ \, Q \subseteq M \, $,  there exist  $ \, c_{ji} \in \ZZ \, $
   such that  $ \; \alpha_i \, = \, \sum_{j \in I} c_{ji} \, \mu_j \; $  for all  $ \, i \in I \, $.
   Then  $ \; K_{\alpha_i^+} \, = \, \prod_{j \in J} K_{\mu_j^+}^{\,c_{ji}} \; $  and
   $ \; K_{i,+}^{\scriptscriptstyle \,\Psi} \, = \, \prod_{j \in I} K_{\varpi_j^+}^{\,c_{ji}} \; $.
   This implies that  $ \; \tilde{\varphi}\big( \hat{\pi}_\lieg^{\scriptscriptstyle \Psi}\big( K_{i,+}^{\scriptscriptstyle \,\Psi} \big) \big)
   \, = \, \hat{\varphi}\big( \hat{\pi}_\lieg^{\scriptscriptstyle \Psi}\big( \prod_{j \in I} K_{\varpi_j^+}^{\,c_{ji}} \big) \big) \, =
   \, \prod_{j \in I} K_{\mu_j^+}^{\,c_{ji}} \, = \, K_{\alpha_i^+} \; $.  Similarly, one sees that
   $ \; \hat{\varphi}\big( \hat{\pi}_\lieg^{\scriptscriptstyle \Psi}\big( K_{i,-}^{\scriptscriptstyle \,\Psi} \big) \big) \, =
   \, L_{-\alpha_i^-} \; $,  and from this it is easy to verify that  $ \hat{\varphi} $  is indeed a Hopf algebra morphism.
 \vskip4pt
   Finally, since  $ \, \hat{\varphi} \circ \varphi \, $  and  $ \, \varphi \circ \hat{\varphi} \, $
   are the identity on the generators, we conclude that
   $ \, U_{\bq\,,M_\bullet\!}(\hskip0,8pt\lieg) \,\; \cong \;\,  U_{q,M}^{\scriptscriptstyle \Psi}(\hskip0,8pt\lieg) \, $,
   via the formulas given in the claim, q.e.d.
\epf

\vskip13pt

\begin{rmk}
 A final remark is in order.  In the previous results   --- Proposition \ref{prop: Tw-Uqb=Mp-Uqb},
 Theorem \ref{thm: Tw-Uqgd=Mp-Uqgd}  and  Theorem \ref{thm: Tw-Uqg=Mp-Uqg}  ---   we took from scratch
 $ \, \Psi \in \lieso_n(\QQ) \, $,  i.e., our ``twisting datum''  $ \Psi $  was antisymmetric.
 However, we can also start with  {\sl any\/}  twisting matrix  $ \, \Psi \in M_n(\QQ) \, $:
 then  {\sl those results read the same as soon as we replace  $ \, \Psi \leftrightsquigarrow \bq \, $
 with  $ \, (\chi \circ \vartheta)(\Psi) = \bq \, $  (notation of  \S \ref{Twists<->Mp's<->2-coc's})}.
 Notice then that one has
  $$  \bq  \; := \;  (\chi \circ \vartheta)(\Psi)  \; = \;  (\chi \circ \vartheta)\big(\Psi_a\big) $$
 where  $ \, \Psi_a := 2^{-1} \big( \Psi - \Psi^{\scriptscriptstyle T} \big) \, $  is the antisymmetric part of
 $ \Psi \, $.  Eventually, the outcome of this discussion, in short, is the following:
 \vskip4pt
   {\it The (polynomial) TwQUEA's built out of  {\sl any}  matrix  $ \, \Psi \in M_n(\QQ) \, $
   are exactly the same as those obtained just from  {\sl antisymmetric}  matrices  $ \, \Psi \in \lieso_n(\QQ) \, $}.
\end{rmk}

\vskip25pt

\end{document}